\documentclass[review,3p,square]{elsarticle}
\usepackage{amssymb}
\usepackage{amsmath}
\usepackage{amsthm}
\usepackage{stmaryrd}
\usepackage{natbib}
\usepackage{bbm}
\usepackage{amsmath}
\allowdisplaybreaks[4]
\usepackage{enumitem}
\usepackage{fancyhdr}
\usepackage{hyperref}
\hypersetup{%
  pdftitle={BSDEs},%
  pdfsubject={BSDEs},%
  pdfauthor={YaQi Zhang, ShengJun FAN},%
  pdfkeywords={BSDEs},%
  pdfstartview=FitH,%
  CJKbookmarks=true,%
  bookmarksnumbered=true,%
  bookmarksopen=true,%
  colorlinks=true, linkcolor=blue, urlcolor=blue, citecolor=blue, %
}

\usepackage{cleveref}
\crefname{thm}{Theorem}{Theorems}
\crefname{pro}{Proposition}{Propositions}
\crefname{lem}{Lemma}{Lemmas}
\crefname{rmk}{Remark}{Remarks}
\crefname{cor}{Corollary}{Corollaries}
\crefname{dfn}{Definition}{Definitions}
\crefname{ex}{Example}{Examples}
\crefname{section}{Section}{Sections}
\crefname{subsection}{Subsection}{Subsections}

\newcommand{\eps}{\varepsilon}

\newcommand{\To}{\rightarrow}
\newcommand{\as}{{\rm d}\mathbb{P}\times{\rm d} t-a.e.}
\newcommand{\ps}{\mathbb{P}-a.s.}

\newcommand{\essinf}{\mathop{\operatorname{ess\,inf}}}

\newcommand{\F}{\mathcal{F}}
\newcommand{\E}{\mathbb{E}}

\newcommand{\D}{\mathbb{D}}

\newcommand{\hcal}{\mathcal{H}}

\newcommand{\dif}{{\rm d}}
\newcommand{\T}{[0,T]}

\newcommand{\R}{{\mathbb R}}
\newcommand{\Q}{{\mathbb Q}}
\newcommand{\RE}{\forall}

\newcommand {\Dis}{\displaystyle}

\newtheorem{thm}{Theorem}[section]

\newtheorem{pro}[thm]{Proposition}
\newtheorem{rmk}[thm]{Remark}
\newtheorem{cor}[thm]{Corollary}
\newtheorem{dfn}[thm]{Definition}
\newtheorem{ex}[thm]{Example}

\journal{arXiv}
\begin{document}
\begin{frontmatter}

\title{{Solvability of BSDEs with possibly unbounded stochastic coefficients \\ on a general weighted $L^p$ space}\tnoteref{found}}
\tnotetext[found]{This work is partially supported by National Natural Science Foundation of China (No. 12171471), by Lebesgue Center of Mathematics ``Investissements d'avenir" program-ANR-11-LABX-0020-01, by CAESARS-ANR-15-CE05-0024 and by MFG-ANR-16-CE40-0015-01.
\vspace{0.2cm}}

\author[Fan]{Yaqi Zhang} \ead{TS22080028A31@cumt.edu.cn}
\author[Fan]{Xinying Li} \ead{lixinyingcumt@163.com}
\author[Hu]{Ying Hu} \ead{ying.hu@univ-rennes1.fr}
\author[Fan]{Shengjun Fan\corref{cor}} \ead{shengjunfan@cumt.edu.cn}\vspace{-0.6cm}

\address[Fan]{School of Mathematics, China University of Mining and Technology, Xuzhou 221116, China}

\address[Hu]{Univ. Rennes, CNRS, IRMAR-UMR6625, F-35000, Rennes, France}

\cortext[cor]{Corresponding \vspace{0.2cm}author}
\vspace{0.2cm}

\begin{abstract}
This paper is devoted to solving a multidimensional backward stochastic differential equation (BSDE for short) with a general random terminal time $\tau$ taking values in $[0,+\infty]$. The generator $g$ of such BSDE satisfies a stochastic monotonicity condition in the state variable $y$ and a stochastic Lipschitz condition in the state variable $z$ with possibly unbounded stochastic coefficients $\mu_\cdot\in\R$ and $\nu_\cdot\in\R_+$ satisfying $\int_0^\tau (|\mu_t|+\nu^2_t) {\rm d}t<+\infty$, along with a very general growth in $y$ that is more easily verified and weaker than existing ones. Let $p>1$ be a given constant and $\rho_\cdot\geq \mu_\cdot+\frac{\theta}{2[1\wedge(p-1)]}\nu_\cdot^2$ be a given real-valued process for some constant $\theta>1$ such that $\int_0^\tau |\rho_t|{\rm d}t<+\infty$. In a general weighted $L^p$ space with a weighted factor $e^{\int_0^t \rho_r{\rm d}r}$, we establish an existence and uniqueness result for the adapted solution of previous BSDE when the terminal value satisfies an associated weighted integrability condition, broadening the scope of the process $\rho_\cdot$ in the weighted factor and thereby unifying and strengthening some corresponding existing results obtained in \citet{DarlingandPardoux1997}, \citet{Briand2003}, \citet{LiFan2024SD} and \citet{Li2025}. Some innovative ideas are presented in order to address the general weighted space and the very general growth condition. As applications, we prove the existence of viscosity solutions for parabolic and elliptic PDEs linked with previous BSDEs under some general assumptions on their nonlinear terms, and establish a dual representation of an unbounded dynamic concave utility defined on a general weighted $L^p$ space via the weighted $L^p$ solutions of previous \vspace{0.3cm}BSDEs.
\end{abstract}

\begin{keyword}
Backward stochastic differential equation \sep Weighted $L^p$ space\sep Dynamic concave utility \sep\\
\hspace*{1.85cm} Unbounded stochastic coefficient \sep Viscosity solution\vspace{0.2cm}

\MSC[2021] 60H10, 60H30\vspace{0.2cm}
\end{keyword}

\end{frontmatter}
\vspace{-0.4cm}

\section{Introduction}
\label{sec:1-Introduction}
\setcounter{equation}{0}
For $a,b\in \R$, define $a\vee b=\max\{a,b\}$, $a\wedge b=\min\{a,b\}$, $a^+=a\vee 0$ and $a^-=(-a)^+$. Fix two positive integers $k$ and $d$. Let $(\Omega,\F, \mathbb{P})$ be a completed probability space carrying a standard $d$-dimensional Brownian motion $(B_t)_{t\geq0}$, $(\F_t)_{t\geq0}$ be the natural augmented filtration generated by $B_\cdot$ and $\F:=\F_\tau$ with $\tau$ being a general $(\F_t)$-stopping time taking values in $[0,+\infty]$. We are concerned with a classical multidimensional backward stochastic differential equation (BSDE) of the following form:
\begin{align}\label{BSDE1.1}
  y_t=\xi+\int_t^\tau g(s,y_s,z_s){\rm d}s-\int_t^\tau z_s{\rm d}B_s, \ \ t\in[0,\tau],
\end{align}
where the terminal value $\xi$ is an $\F_\tau$-measurable $k$-dimensional random vector and the random function
$$g(\omega,t,y,z): \Omega\times[0,\tau]\times \R^k\times \R^{k\times d}\mapsto \R^k$$
is $(\F_t)$-progressively measurable for each $(y,z)$, called the generator of BSDE \eqref{BSDE1.1}. Denote BSDE \eqref{BSDE1.1} with parameters $(\xi,\tau, g)$ by BSDE$(\xi,\tau, g)$. A pair of $(\F_t)$-progressively measurable processes $(y_t,z_t)_{t\in [0,\tau]}$ taking values in $\R^k \times R^{k\times d}$ is called an adapted solution to BSDE \eqref{BSDE1.1} if $\ps$, $y_\cdot$ is continuous, $\int_0^{\tau} \left(|g(t,y_t,z_t)|+|z_t|^2\right){\rm d}t<+\infty$, and \eqref{BSDE1.1} is satisfied. It is well known that \citet{PardouxPeng1990SCL} first studied nonlinear BSDEs and established an existence and uniqueness result for the $L^2$ solution of a multidimensional BSDE with a positive constant terminal time $T$, where both the terminal value $\xi$ and the process $\{g(t,0,0)\}_{t\in[0,T]}$ are square-integrable and the generator $g$ is uniformly Lipschitz continuous in the state variables $(y,z)$: there exist two constants $\mu$ and $\nu$ such that $\as$, for each $(y_1,y_2,z_1,z_2)\in \R^k\times\R^k\times\R^{k\times d}\times\R^{k\times d}$,
\begin{align}\label{eq:1.2}
|g(\omega,t,y_1,z_1)-g(\omega,t,y_2,z_2)|\leq \mu |y_1-y_2|+\nu |z_1-z_2|.
\end{align}
This work laid the foundation for the subsequent development on the BSDE theory. Over the past three decades, BSDEs have been extensively studied and applied in a wide range of fields including partial differential equations (PDEs), mathematical finance, and stochastic control and game. This is evidenced by numerous subsequent works, such as \cite{KarouiPengQuenez1997,DarlingandPardoux1997,
Pardoux1999,HuYing2005,BriandandConfortola2008,DelbaenTang2010,
PardouxandRascanu2014,Bahlali2015,FanHu2021SPA,Tian2023SIAM} and their references.

The uniform Lipschitz continuity condition \eqref{eq:1.2} employed in \cite{PardouxPeng1990SCL} have been greatly relaxed for intensive study of BSDEs with a constant terminal time. See \cite{Mao1995,Lepeltier1998,Pardoux1999,Kobylanski2000,Liu and Ren2002,Briand2003,Briand2006,Briand2007,Briand2008,FanJiang2010,
DelbaenHuBao2011,FanJiang2013,Fan2018,FanHu2021SPA,FanHuTang2023} for details. Among others, we would like to mention the monotonicity condition along with a general growth condition of $g$ in $y$ used in \cite{DarlingandPardoux1997,BriandandHu1998,Pardoux1999} and the uniform continuity condition of $g$ in $z$ employed in \cite{Ham2003,Jia2010}. For the study of BSDEs with a possibly unbounded random terminal time $\tau$, the condition \eqref{eq:1.2} usually needs to be replaced with the stochastic monotonicity condition in $y$ and stochastic Lipschitz condition in $z$: $\mathbb{P}-a.s.$, for each $(y_1,y_2,z_1,z_2)\in \R^k\times\R^k\times\R^{k\times d}\times\R^{k\times d}$,
\begin{align}\label{eq:1.3}
\langle y_1-y_2,~g(\omega,t,y_1,z_1)-g(\omega,t,y_2,z_2)\rangle\leq \mu_t(\omega) |y_1-y_2|^2+\nu_t(\omega) |y_1-y_2||z_1-z_2|,\ \ t\in [0,\tau(\omega)]
\end{align}
where both $\mu_\cdot\in \R$ and $\nu_\cdot\in\R_+$ are two $(\F_t)$-progressively measurable processes such that
\begin{align}\label{eq:1.4}
\ps,\ \ \ \int_0^\tau \left(|\mu_t|+\nu_t^2\right){\rm d}t<+\infty,
\end{align}
along with a general growth condition of $g$ in $y$. Generally speaking, some extra moment integrability conditions on $\int_0^\tau \left(\mu_t^+ +\nu_t^2\right){\rm d}t$ are also required in order to obtain existence and uniqueness of the usual adapted solution. For instance, it is forced to be bounded by a constant $M>0$ in \cite{ZChenBWang2000JAMS,Morlais2009,FanJiang2010SPL,FanJiangTian2011SPA,
FanWangXiao2015,Xiao2015,XiaoandFan2017,Liu2020,LiXuFan2021PUQR}, and it is supposed in \cite{Yong2006,Briand2008,Bahlali2015} to satisfy some certain moment integrability weaker than the boundedness. Instead of imposing these extra moment integrability conditions, another strategy is to study existence and uniqueness of the adapted solution in suitable weighted spaces with certain weighted factors associated with processes $\mu_\cdot$ and $\nu_\cdot$. See \cite{KarouiHuang1997,DarlingandPardoux1997, BriandandHu1998,Pardoux1999,BenderKohlmann2000,
Bahlali2004,WangRanChen2007,Bahlali2015,Li2023,LiFan2024SD,Li2025} for different weighted spaces and weighted factors. The present paper focuses on solving BSDEs with a possibly unbounded random terminal time $\tau$ and a generator $g$ satisfying \eqref{eq:1.3}-\eqref{eq:1.4} on a more general weighted space.

Let us briefly illustrate the theoretical motivation of this study. Let the generator $g$ satisfy the following general growth condition in the state variable $y$: there exist a constant $\beta>1$ and an $(\F_t)$-progressively real-valued process satisfying $\essinf\limits_{t\in[0,\tau]}\alpha_t>0$ such that
\begin{align}\label{eq:1.5}
\forall r\in \R_+,\ \ \ \E\left[\int_0^\tau e^{\beta \int_0^t \mu^+_s{\rm d}s}\psi(t,r\alpha_t)dt\right]<+\infty
\end{align}
with \vspace{-0.1cm}
\begin{align}\label{eq:1.6}
\psi(t,r):=\sup_{|y|\leq r} \left|g(t,y,0)-g(t,0,0)\right|.
\end{align}
Very recently, it was proved in \citet{Li2025} and \citet{LiFan2024SD} that if the generator $g$ satisfies \eqref{eq:1.3}-\eqref{eq:1.5} and the following integrability condition holds: there exists a constant $p>1$ such that
\begin{align}\label{eq:1.7}
\E\left[e^{p\int_0^\tau \rho_s\dif s}|\xi|^p+\left(\int_0^\tau e^{\int_0^s \rho_r \dif r} |g(s,0,0)|{\rm d}s\right)^{p}\right]<+\infty
\end{align}
with for some constant $\theta>1$, the process\vspace{-0.1cm}
\begin{align}\label{eq:1.8}
\rho_t:=\beta \mu^+_t +\frac{\theta}{2[1\wedge(p-1)]}\nu_t^2,\ \ t\in [0,\tau],
\end{align}
then BSDE$(\xi,\tau,g)$ admits a unique adapted solution $(y_t,z_t)_{t\in [0,\tau]}$ belonging to a weighted $L^p$ space with the weighted factor $e^{\int_0^t \rho_r \dif r}$ (see $H_\tau^p(\rho_\cdot;\R^{k}\times\R^{k\times d})$ in Section 2 for details). On the other hand, it was shown in \citet{DarlingandPardoux1997,Pardoux1999,Briand2003} that when the terminal time $\tau$ is finite (i.e., $P(\tau<+\infty)=1$) and the processes $\mu_\cdot$ and $\nu_\cdot$ appearing in \eqref{eq:1.3}-\eqref{eq:1.4} are degenerated to two constants $\mu\in\R$ and $\nu\in\R_+$, if \eqref{eq:1.5} is replaced by some appropriate growth conditions of $g$ in $y$ and the process $\rho_\cdot$ in \eqref{eq:1.7} is replaced by a constant $\rho$ satisfying
\begin{align}\label{eq:1.9}
\rho>\mu+\frac{1}{2(p-1)}\nu^2
\end{align}
for some constant $p\in (1,2]$, then the previous assertion remains true. We emphasize that the process $\rho_\cdot$ in \eqref{eq:1.8} is always nonnegative, while the constants $\mu$ and then $\rho$  in \eqref{eq:1.9} can take negative values. Observe that \eqref{eq:1.9} is equivalent to in the case of $\nu\neq 0$,
\begin{align}\label{eq:1.10}
\rho\geq \mu+\frac{\theta}{2[1\wedge (p-1)]}\nu^2
\end{align}
for some constant $\theta>1$, and $\rho>\mu$ in the case of $\nu=0$. Now, let $\tau$ be a possibly infinite terminal time and $\rho_\cdot$ be an $(\F_t)$-progressively measurable real-valued process satisfying $\int_0^\tau |\rho_t|{\rm d}t<+\infty$ and
\begin{align}\label{eq:1.11}
\rho_t\geq \mu_t +\frac{\theta}{2[1\wedge(p-1)]}\nu_t^2,\ \ t\in [0,\tau]
\end{align}
for two constants $p>1$ and $\theta>1$. Then, a natural question is asked: when the conditions \eqref{eq:1.3}-\eqref{eq:1.7} with \eqref{eq:1.11} instead of \eqref{eq:1.8} are fulfilled, can we expect that BSDE$(\xi,\tau,g)$ admits a unique adapted solution $(y_t,z_t)_{t\in [0,\tau]}$ in the space of $H_\tau^p(\rho_\cdot;\R^{k}\times\R^{k\times d})$? In addition, it is very strange in our opinion that the term ${\rm e}^{\beta\int_0^\tau \mu^+_s{\rm d}s}$ enters into the growth condition \eqref{eq:1.5}, and then another question is asked: can the condition \eqref{eq:1.5} be improved? In this paper, we will give an affirmative answer to these two questions.

More specifically, it is proved in this paper that when the generator $g$ satisfies \eqref{eq:1.3} and \eqref{eq:1.4} along with a very general growth condition in the state variable $y$: there exists an $(\F_t)$-progressively real-valued process $(\alpha_t)_{t\in [0,\tau]}$ satisfying $\essinf\limits_{t\in[0,\tau]}\alpha_t>0$ such that
\begin{align}\label{eq:1.12}
\forall r\in \R_+,\ \ \ \E\left[\int_0^\tau \alpha_t\psi(t,r\alpha_t)dt\right]<+\infty
\end{align}
with the random function $\psi(\cdot)$ being defined in \eqref{eq:1.6}, if the integrability condition \eqref{eq:1.7} with \eqref{eq:1.11} is fulfilled, then BSDE$(\xi,\tau,g)$ admits a unique adapted solution $(y_t,z_t)_{t\in [0,\tau]}$ in the general weighted $L^p$ space of $H_\tau^p(\rho_\cdot;\R^{k}\times\R^{k\times d})$. See \cref{thm:3.1} in Section 3.1 for details. It is noteworthy that the growth condition \eqref{eq:1.12} is more easily verified and weaker than \eqref{eq:1.5}, and that the process $\rho_\cdot$ and the constant $\rho$ satisfying respectively \eqref{eq:1.8} and \eqref{eq:1.10} must satisfy \eqref{eq:1.11}. Hence, \cref{thm:3.1} unifies and strengthens those corresponding results obtained in \citet{DarlingandPardoux1997}, \citet{LiFan2024SD} and \citet{Li2025}, and then many other existing ones. Under some general assumptions on the generator $g$ and the terminal time $\tau$, several unknown existence and uniqueness results on the usual $L^p$ solution and the bounded solution of BSDE$(\xi,\tau,g)$ are also derived as corollaries of \cref{thm:3.1}, and several examples and remarks are provided to demonstrate innovativeness of our theoretical results. See those corollaries, remarks and examples in Section 3 for more details. In particular, \cref{ex3.5} indicates necessariness of the growth condition like \eqref{eq:1.12} for existence of the adapted solution.

The proof of \cref{thm:3.1} is accomplished along the main line drawn in the proof of \cite[Theorem 3.1]{Li2025} and \cite[Theorem 3.1]{LiFan2024SD}, and needs to employ systemically some methods involving the delicate truncation technique, the approach proving convergence of the sequence via a priori estimates, and the contract mapping argument. In order to tackle the more general weighted factor and the more general growth of the generator $g$ in $y$, we establish new a priori estimates, adopt novel truncation forms, and construct delicate auxiliary functions. Particularly, it is for the first time explored in the proof of \cref{thm:3.1} that our elementary assumptions $\essinf\limits_{t\in[0,\tau]}\alpha_t>0$ and $\int_0^\tau (|\rho_t|+|\mu_t|+\nu_t^2)\dif t<+\infty$ seem to be necessary for existence of the adapted solution. More details are available in Section 2.3 and Section 4.

As an application of \cref{thm:3.1}, under some general assumptions we derive nonlinear Feynman-Kac formulas for both parabolic and elliptic PDEs. More specifically, it is verified in Section 5 that a viscosity solution for such PDE can be given via the adapted solutions on a general weighted $L^p$ space of the associated BSDEs coupled with a SDE. See \cref{thm:5.1}, \cref{rmk:5.2} and \cref{thm:5.4} for more details. The nonlinear term of the underlying PDE satisfies a nonuniform monotonicity condition with respect to the solution and a nonuniform Lipschitz condition with respect to the gradient of the solution with the monotonicity coefficient and the Lipschitz coefficient both depending on the state as in \citet{PardouxandRascanu2014,Bahlali2015,Li2025}. The salient feature of the present work lies in that the above-mentioned monotonicity coefficient can take values in $\R$, both the monotonicity coefficient and the Lipschitz coefficient can admit an arbitrarily growth in the state, and only a certain combination of the monotonicity coefficient and the Lipschitz coefficient needs to be dominated from above by a constant or a quadratic function. As a result, some existing results are improved by ours.

As another application of our theoretical results, in the sprit of \citet{FanHuTang2025} we study a class of dynamic concave utilities defined on a general weighted $L^p$ space via BSDEs with possibly unbounded stochastic coefficients. Let us roughly state our main assertion. Assume that $\rho_\cdot\in\R$, $\mu_\cdot\in\R$ and $\nu_\cdot\in\R_+$ are three $(\F_t)$-progressively measurable processes satisfying $\int_0^\tau (|\rho_t|+|\mu_t|+\nu_t^2)\dif t<+\infty$ and \eqref{eq:1.11} for two constants $p>1$ and $\theta>1$, and that the random core function
$$
f(\omega,t,q):\Omega \times [0,\tau]\times \R^{1\times d}\mapsto \R_+\cup\{+\infty\}
$$
is $(\F_t)$-progressively measurable for each $q\in\R^{1\times d}$ and satisfies that ${\rm d}\mathbb{P}\times{\rm d} t-a.e.$, $f(\omega,t,\cdot)$ is proper, lower semi-continuous, convex and $f(\omega,t,q)\equiv +\infty$ when $|q|>\nu_t(\omega)$. For each $\F_\tau$-measurable real-valued random variable $\xi$ such that $\E[{\rm e}^{p\int_0^\tau \rho_s {\rm d}s} |\xi|^p]<+\infty$, that is $\xi\in L_\tau^p(\rho_\cdot;\R)$, define
\begin{align}\label{eq:1.13}
U_{t\wedge \tau}(\xi):=e^{-\int_0^{t\wedge \tau}\mu_r\dif r }\essinf\limits_{q\in\hcal(\xi,f,\mu_\cdot,\nu_\cdot)}\E_{\Q^q}
\Big[e^{\int_0^\tau \mu_r\dif r}\xi+\int_{t\wedge \tau}^{\tau}e^{\int_0^s \mu_r\dif r}f(s,q_s)\dif s\bigg|\F_{t\wedge \tau}\Big], \ \ t\geq 0,
\end{align}
where $\hcal(\xi,f,\mu_\cdot,\nu_\cdot)$ is a set of density processes $(q_t)_{t\in [0,\tau]}$ of probability measures $Q^q$ equivalent to $\mathbb{P}$ associated with $(\xi,f,\mu_\cdot,\nu_\cdot)$. With the help of \cref{thm:3.1}, it is proved in \cref{thm:6.1} of Section 6 that there exists an appropriate set $\hcal(\xi,f,\mu_\cdot,\nu_\cdot)$ associated with $(\xi,f,\mu_\cdot,\nu_\cdot)$ such that for each $\xi\in L_\tau^p(\rho_\cdot;\R)$, the process $\{U_t(\xi),~ t\in [0,\tau]\}$ defined in \eqref{eq:1.13} can be
represented as the value process $y_\cdot$ of the unique adapted solution $(y_t,z_t)_{t\in[0,\tau]}$ in $H_\tau^p(\rho_\cdot;\R\times\R^{1\times d})$ of the following BSDE
\begin{align}\label{eq:1.14}
 y_t=\xi+\int_t^\tau [\mu_s y_s-g(s,-z_s)] {\rm d}s-\int_t^\tau z_s{\rm d}B_s, \ \ t\in[0,\tau],
\end{align}
where the random function $g$ is the Legendre-Fenchel transform of $f$, i.e.,
\begin{align*}
g(\omega,t,z):=\sup_{q\in\R^{1\times d}}\big(q\cdot z-f(\omega,t,q)\big),~~~ (\omega,t,z)\in\Omega\times[0,\tau(\omega)]\times\R^{1\times d},
\end{align*}
and then the operator $\{U_t(\cdot),t\in[0,\tau]\}$ constitutes a dynamic concave utility defined on $L_\tau^p(\rho_\cdot;\R)$. Note that the generator of BSDE \eqref{eq:1.14} satisfies a stochastic monotonicity condition in $y$ and a stochastic Lipschitz condition in $z$ with possibly unbounded stochastic coefficients $\mu_\cdot$ and $\nu_\cdot$. Moreover, in \cref{thm:6.1} we verify the accessibility of the infimum in \eqref{eq:1.13} under some extra mild conditions by developing novel methods and techniques. To the best of our knowledge, this is the first work to apply the adapted solution in a weighted space of BSDEs with stochastic coefficients to the study on dynamic concave utilities.

The remainder of this paper is organized as follows. Section 2 contains some notations used later, an inspiration example illustrating why we choose the aforementioned weighted factor, and several a priori estimates on the weighted $L^p$ solutions of BSDEs with possibly unbounded stochastic coefficients. In Section 3, we state the main existence and uniqueness result of this paper--\cref{thm:3.1}, prove the uniqueness part, outline the proof of the existence part, and verify a continuous dependence property and a stability theorem for the weighted $L^p$ solution of BSDEs. Some corollaries, remarks and examples are also provided in Section 3 to illustrate the novelty of our study. Section 4 is devoted to the detailed proof of the existence part of \cref{thm:3.1}. In section 5, we establish the existence of viscosity solutions for the associated parabolic and elliptic PDEs. Finally, in Section 6 we demonstrate the dual representation of a unbounded dynamic concave utility defined on $L_\tau^p(\rho_\cdot;\R)$ via the linked BSDEs.

\section{Preliminaries}
\setcounter{equation}{0}

In this section, we initiate with the introduction of some notations, followed by the presentation of an inspirational example, and then by some a priori estimates on the weighted $L^p$ solution of BSDEs.

\subsection{Notations}

Denote the usual Euclidean norm by $|\cdot|$ and the usual inner product of $x,y\in \R^k$ by $\langle x,y\rangle$. For a matrix $A$, we define $|A|:=\sqrt{tr(AA*)}$, where $A*$ represents the transpose of $A$. Let $\R_-:=(-\infty,0]$, $\R_+:=[0,+\infty)$ and $\hat{y}:=\frac{y}{|y|}{\bf 1}_{|y|\neq0}$ for each $y\in \R^k$, where ${\bf 1}_B$ represents the indicator of set $B$. Denote by $L^\infty(\Omega,\F_\tau, \mathbb{P};\R^k)$ the set of $\F_\tau$-measurable $\R^k$-valued random vectors such that
$$\|X\|_{\infty}:=\sup\{x: \mathbb{P}(|X|>x)>0\}<+\infty,$$
and by $S^\infty(\R^k)$ the set of $(\F_t)$-adapted, $\R^k$-valued, continuous and bounded processes. Every equality and inequality between random elements should be understood as holding $\mathbb{P}$-almost surely.

For an $(\F_t)$-progressively measurable process $\rho_t(\omega):\Omega \times [0,\tau]\mapsto \R$ satisfying $\int_{0}^{\tau}|\rho_t|{\rm d}t<+\infty$ and a constant $p>1$, we
denote the following weighted $L^p$ spaces.\vspace{0.2cm}

$\bullet$  $L_\tau^p(\rho_\cdot;\R^k)$ denotes the set of all $\xi$ that are $\F_\tau$-measurable $\R^k$-valued random vectors with
$$\|\xi\|_{p;\rho_\cdot}:=\left(\E\left[e^{p \int_0^\tau \rho_s{\rm d}s}|\xi|^p\right]\right)^{\frac{1}{p}}<+\infty.$$

$\bullet$  $S_\tau^p(\rho_\cdot;\R^k)$ denotes  the set of all $(Y_t)_{t\in[0,\tau]}$ that are $(\F_t)$-adapted, $\R^k$-valued, and continuous processes with
$$\|Y_\cdot\|_{p; \rho_\cdot,c}:=\left(\E\left[\sup_{t\in[0,\tau]}\left(e^{p \int_0^t \rho_r{\rm d}r}|Y_t|^p\right)\right]\right)^{\frac{1}{p}}<+\infty.$$

$\bullet$  $M_\tau^p(\rho_\cdot;\R^{k\times d})$ denotes the set of all $(Z_t)_{t\in[0,\tau]}$ that are  $(\F_t)$-progressively measurable $\R^{k\times d}$-valued processes with
$$\|Z_\cdot\|_{p; \rho_\cdot}:=\left(\E\left[{\left(\int_0^\tau e^{2 \int_0^t \rho_r{\rm d}r}|Z_t|^2{\rm d}t\right)}^\frac{p}{2}\right]\right)^{\frac{1}{p}}<+\infty.$$
Furthermore, define
$$H_\tau^p(\rho_\cdot;\R^{k}\times\R^{k\times d}):=S_\tau^p(\rho_\cdot;\R^k)\times M_\tau^p(\rho_\cdot;\R^{k\times d}).$$
It is clear that $H_\tau^p(\rho_\cdot;\R^{k}\times\R^{k\times d})$ is a Banach space with the norm
$$\|(Y_\cdot,Z_\cdot)\|_{p; \rho_\cdot}:=\|Y_\cdot\|_{p; \rho_\cdot,c}
+\|Z_\cdot\|_{p; \rho_\cdot}.$$
Finally, if a pair of processes $(y_t,z_t)_{t\in [0,\tau]}\in H_{\tau}^p(\rho_\cdot;\R^{k}\times\R^{k\times d})$ is an adapted solution of BSDE \eqref{BSDE1.1}, then it is called a weighted $L^p$ solution of BSDE \eqref{BSDE1.1} in the space of $H_\tau^p(\rho_\cdot;\R^{k}\times\R^{k\times d})$.

The following remark indicates what kind of random vector $\xi$ belongs to the weighted space of $L_\tau^p(\rho_\cdot;\R^k)$. Similar results can be obtained for the weighted spaces $S_\tau^p(\rho_\cdot;\R^k)$ and $M_\tau^p(\rho_\cdot;\R^{k\times d})$.

\begin{rmk}\label{rmk:2.1}
Let $p>1$ and $M>0$ be two constants. We have the following assertions.
\begin{itemize}
\item [(i)] If $\int_0^\tau |\rho_t|\dif t\leq M$, then $L_\tau^p(\rho_\cdot;\R^k)=L_\tau^p(0;\R^k)$; If $\int_0^\tau \rho_t^+\dif t\leq M$, then $L_\tau^p(0;\R^k)\subseteq L_\tau^p(\rho_\cdot;\R^k)$.

\item [(ii)] If $\E\big[e^{p\int_0^\tau \rho_t\dif t} \big]<+\infty$, then $L^\infty(\Omega,\F_\tau, \mathbb{P};\R^k)\subseteq L_\tau^p(\rho_\cdot;\R^k)$; If $\E\big[e^{\bar{p}\int_0^\tau \rho_t\dif t} \big]<+\infty$ for some $\bar{p}>p$, then by H\"{o}lder's inequality, $L_\tau^{\bar{q}}(0;\R^k)\subseteq L_\tau^p(\rho_\cdot;\R^k)$, where $\bar{q}$ is the constant such that $\frac{1}{\bar{p}}+\frac{1}{\bar{q}}=\frac{1}{p}$; If $\E\big[e^{\bar{p}\int_0^\tau \rho_t\dif t} \big]<+\infty$ for any $\bar{p}>1$, then $L_\tau^{q}(0;\R^k)\subseteq L_\tau^p(\rho_\cdot;\R^k)$ for each $q>p$.

\item [(iii)] If $|\xi|\leq Me^{-\int_0^\tau \rho_t\dif t}$, then $\xi\in L_\tau^p(\rho_\cdot;\R^k)$; Let $\rho_\cdot\in\R$, $b_\cdot\in\R$ and $\sigma_\cdot\in \R^{1\times d}$ be three $(\F_t)$-progressively measurable processes such that $\int_0^\tau (|\rho_t|+|b_t|+|\sigma_t|^2)\dif t<+\infty$, and let $(X_t)_{t\in [0,\tau]}$ solve the following SDE
    $$
    {\rm d}X_t= X_t\left(b_t{\rm d}t+ \sigma_t{\rm d}B_t\right),\ \ t\in [0,\tau];\ \ \ \ X_0=x_0\in\R.
    $$
    If
    $$
    \int_0^\tau \left(\rho_t+b_t+\frac{p-1}{2}|\sigma_t|^2\right){\rm d}t\leq M,
    $$
    then
    $$X_\tau=x_0e^{\int_0^\tau (b_t-\frac{1}{2}|\sigma_t|^2)\dif t+\int_0^\tau \sigma_t \dif B_t}\in L_\tau^p(\rho_\cdot;\R^k).
    $$
    In fact, we have
\begin{align*}
\E\left[e^{p\int_0^\tau \rho_t{\rm d}t} |X_\tau|^p \right]=|x_0|^p~\E\left[e^{p\int_0^\tau \sigma_t{\rm d}B_t-\frac{p^2}{2}\int_0^\tau |\sigma_t|^2{\rm d}t}~e^{p\int_0^{\color{red}\tau} \left(\rho_t+b_t+\frac{p-1}{2}|\sigma_t|^2\right){\rm d}t}\right]<+\infty.
\end{align*}
\end{itemize}
\end{rmk}

\subsection{An inspiration example}

In this subsection, we assume that $k=d=1$, $p>1$ and $T>0$ are four finite constants, and consider solvability of the following linear BSDE:
\begin{align}\label{linearBSDE}
  y_t=\xi+\int_t^T (\mu_sy_s+\nu_sz_s){\rm d}s-\int_t^T z_s{\rm d}B_s, \ \ t\in[0,T],
\end{align}
where $\mu_t(\omega)$: $\Omega \times [0,T]\mapsto \R$ and $\nu_t(\omega)$: $\Omega \times [0,T]\mapsto \R_+$ are two $(\F_t)$-progressively measurable processes satisfying $\int_{0}^{T}\left(|\mu_t|+\nu_t^2\right){\rm d}t<+\infty$. This kind of linear BSDEs has been investigated in \citet{Li2025}. However, the stochastic coefficient $\mu_\cdot$ therein is forced to be a nonnegative process. It is easily seen that when $\mu_\cdot, \nu_\cdot$ and $\xi$ are bounded, BSDE \eqref{linearBSDE} admits a unique bounded solution such that
\begin{align}\label{yt}
y_t=\E\left[\xi e^{\int_t^T\mu_r{\rm d}r+\int_t^T \nu_r{\rm d}B_r-\frac{1}{2}\int_t^T\nu_r^2{\rm d}r}\bigg|\F_t\right], \ \ \ t\in[0,T].
\end{align}
To ensure that the process $y_\cdot$ in \eqref{yt} is well defined, it is necessary and sufficient that
$$
\E\left[\xi e^{\int_0^T\mu_r{\rm d}r+\int_0^T \nu_r{\rm d}B_r-\frac{1}{2}\int_0^T\nu_r^2{\rm d}r}\right]<+\infty.
$$
We want to know what happens when $\mu_\cdot, \nu_\cdot$ and $\xi$ are unbounded, and want to avoid the stochastic integral term to appear in the above condition. By H\"{o}lder's inequality we have
\begin{align*}
\begin{split}
&\E\left[\xi e^{\int_0^T\mu_r{\rm d}r+\int_0^T \nu_r{\rm d}B_r-\frac{1}{2}\int_0^T\nu_r^2{\rm d}r}\right]\\
&=\E\left[\xi e^{\int_0^T\mu_r{\rm d}r+\frac{1}{2(p-1)}\int_0^T\nu_r^2{\rm d}r} e^{\int_0^T \nu_r{\rm d}B_r-(\frac{1}{2}+\frac{1}{2(p-1)})\int_0^T\nu_r^2{\rm d}r}\right]\\
&\leq\left(\E\left[|\xi|^p e^{p(\int_0^T\mu_r{\rm d}r+\frac{1}{2(p-1)}\int_0^T\nu_r^2{\rm d}r)}\right]\right)^{\frac{1}{p}}
\left(\E\left[e^{\frac{p}{p-1}\int_0^T \nu_r{\rm d}B_r-\frac{p^2}{2(p-1)^2}\int_0^T\nu_r^2{\rm d}r}\right]\right)^{\frac{p-1}{p}}\\
&\leq \left(\E\left[|\xi|^p e^{p(\int_0^T\mu_r{\rm d}r+\frac{1}{2(p-1)}\int_0^T\nu_r^2{\rm d}r)}\right]\right)^{\frac{1}{p}}.
\end{split}
\end{align*}
Consequently, the process $y_\cdot$ in \eqref{yt} can be well defined as long as there exists an $(\F_t)$-progressively measurable real-valued process $\rho_\cdot$ satisfying $\int_0^T |\rho_t| \dif t<+\infty$ such that
\begin{align}\label{pxi}
\begin{split}
\rho_\cdot\geq \mu_\cdot+\frac{1}{2(p-1)}\nu_\cdot^2 \ \ \text{and} \ \ \E\left[|\xi|^p e^{p\int_0^T\rho_r{\rm d}r}\right]<+\infty.
\end{split}
\end{align}
Next, we verify that if \eqref{pxi} holds, then the process $y_\cdot$ in \eqref{yt} satisfies
\begin{align}\label{pyt}
\begin{split}
\sup\limits_{t\in[0,T]}\E\left[|y_t|^pe^{p\int_0^t\rho_r{\rm d}r}\right]<+\infty.
\end{split}
\end{align}
Indeed, by virtue of \eqref{pxi} and H\"{o}lder's inequality we can deduce that
\begin{align}\label{yt*}
&y_te^{\int_0^t\rho_r{\rm d}r} =\E\left[\xi e^{\int_t^T\mu_r{\rm d}r+\int_t^T\nu_r{\rm d}B_r-\frac{1}{2}\int_t^T\nu_r^2{\rm d}r+\int_0^t\rho_r{\rm d}r}\bigg|\F_t\right]\nonumber\\
&\leq\E\left[\xi e^{\int_t^T(\rho_r-\frac{1}{2(p-1)}\nu_r^2){\rm d}r+\int_t^T\nu_r{\rm d}B_r-\frac{1}{2}\int_t^T\nu_r^2{\rm d}r+\int_0^t\rho_r{\rm d}r}\bigg|\F_t\right]\\
&\leq \left(\E\left[|\xi|^p e^{p\int_0^T\rho_r{\rm d}r}\bigg|\F_t\right]\right)^{\frac{1}{p}}
\left(\E\left[e^{\frac{p}{p-1}\int_t^T\nu_r{\rm d}B_r-\frac{p^2}{2(p-1)^2}\int_t^T\nu_r^2{\rm d}r}\bigg|\F_t\right]\right)^{\frac{p-1}{p}}\nonumber\\
&\leq \left(\E\left[|\xi|^p e^{p\int_0^T\rho_r{\rm d}r}\bigg|\F_t\right]\right)^{\frac{1}{p}},\ \ t\in [0,T], \nonumber
\end{align}
and then
$$
\sup\limits_{t\in[0,T]}\E\left[|y_t|^pe^{p\int_0^t\rho_r{\rm d}r}\right]\leq\sup\limits_{t\in[0,T]}\E\left[|\xi|^p e^{p\int_0^T\rho_r{\rm d}r}\right]<+\infty.\vspace{0.1cm}
$$

It is noteworthy that when \eqref{pxi} holds, the process $y_\cdot$ does not necessarily satisfy
\begin{align}\label{supyt}
\E\left[\sup\limits_{t\in[0,T]}\left(|y_t|^pe^{p\int_0^t\rho_r{\rm d}r}\right)\right]<+\infty.
\end{align}
Now, we give an example as follows. Let $\mu_\cdot=0$, $\nu_\cdot=b_\cdot$ and
$$
\xi:=e^{\frac{1}{p-1}\int_0^T b_r{\rm d}B_r-\frac{2p-1}{2(p-1)^2}\int_0^Tb_r^2{\rm d}r},\vspace{-0.2cm}
$$
where $b_\cdot$ is a given $(\F_t)$-progressively measurable non-negative process such that
$$
\left\{\bar y_t:= e^{\frac{p}{p-1}\int_0^tb_r{\rm d}B_r-\frac{p^2}{2(p-1)^2}\int_0^tb_r^2{\rm d}r}\right\}_{t\in[0,T]}
$$
is uniformly integrable, but $\E\big[\sup\limits_{t\in[0,T]}|\bar y_t|\big]=+\infty.$ Then, \eqref{pxi} holds with $\rho_\cdot= \frac{b_\cdot^2}{2(p-1)}$ since
\begin{align*}
\E\left[|\xi|^p e^{p\int_0^T \rho_r{\rm d}r}\right]=\E\left[e^{\frac{p}{p-1}\int_0^Tb_r{\rm d}B_r-\frac{p^2}{2(p-1)^2}\int_0^Tb_r^2{\rm d}r}\right]<+\infty,
\end{align*}
but the process $y_\cdot$ in \eqref{yt} does not satisfies \eqref{supyt} with $\rho_\cdot= \frac{b_\cdot^2}{2(p-1)}$ since
\begin{align*}
\begin{split}
y_t&=\E\left[\left.e^{\frac{1}{p-1}\int_0^Tb_r{\rm d}B_r-\frac{2p-1}{2(p-1)^2}\int_0^Tb_r^2{\rm d}r}
e^{\int_t^T b_r{\rm d}B_r-\frac{1}{2}\int_t^T b_r^2{\rm d}r}\right|\F_t \right]\\
&=\E\left[\left. e^{\frac{p}{p-1}\int_0^Tb_r{\rm d}B_r-\frac{p^2}{2(p-1)^2}\int_0^Tb_r^2{\rm d}r}\right|\F_t \right]~e^{-\int_0^t b_r{\rm d}B_r+\frac{1}{2}\int_0^t b_r^2{\rm d}r}\\
&=e^{\frac{1}{p-1}\int_0^t b_r{\rm d}B_r-\frac{2p-1}{2(p-1)^2}\int_0^t b_r^2{\rm d}r},\ \ \ t\in \T,
\end{split}
\end{align*}
and then
$$
\E\left[\sup\limits_{t\in[0,T]}\left(|y_t|^p e^{p\int_0^t\rho_r{\rm d}r}\right)\right]=\E\left[\sup\limits_{t\in[0,T]}e^{\frac{p}{p-1}
\int_0^tb_r{\rm d}B_r-\frac{p^2}{2(p-1)^2}\int_0^tb_r^2{\rm d}r}\right]=\E\left[\sup\limits_{t\in[0,T]}|\bar y_t|\right]=+\infty.\vspace{0.1cm}
$$

However, it can be proved that if \eqref{pxi} holds with $\rho_\cdot\geq \mu_\cdot+\frac{{\theta}}{2(p-1)}\nu_\cdot^2$ for some constant ${\theta}>1$, then the process $y_\cdot$ in \eqref{yt} must satisfy \eqref{supyt}. Indeed, set $\tilde{p}:=1+\frac{p-1}{{\theta}}\in (1,p)$ and $\tilde{q}:=1+\frac{{\theta}}{p-1}$ satisfying $1/\tilde{p}+1/\tilde{q}=1$. By virtue of \eqref{yt*} and H\"{o}lder's inequality, we obtain
\begin{align*}
&y_te^{\int_0^t\rho_r{\rm d}r}\leq\E\left[\xi e^{\int_0^T\rho_r{\rm d}r} e^{\int_t^T\nu_r{\rm d}B_r-(\frac{1}{2}+\frac{{\theta}}{2(p-1)})\int_t^T\nu_r^2{\rm d}r}\bigg|\F_t\right]\\
&\ \ \leq\left(\E\left[|\xi|^{\tilde{p}} e^{\tilde{p}\int_0^T\rho_r{\rm d}r}\bigg|\F_t\right]\right)^{\frac{1}{\tilde{p}}}
\left(\E\left[e^{\tilde{q}\int_t^T\nu_r{\rm d}B_r-\frac{\tilde{q}^2}{2}\int_t^T\nu_r^2{\rm d}r}\bigg|\F_t\right]\right)^{\frac{1}{\tilde{q}}}\\
&\ \ \leq \left(\E\left[|\xi|^{\tilde{p}} e^{\tilde{p}\int_0^T\rho_r{\rm d}r}\bigg|\F_t\right]\right)^{\frac{1}{\tilde{p}}}, \ \ t\in[0,T].
\end{align*}
In light of $p/\tilde{p}>1$, combining the last inequality and Doob's maximal inequality yields
\begin{align*}
\begin{split}
\E\left[\sup\limits_{t\in[0,T]}\left(|y_t|^pe^{p\int_0^t\rho_r{\rm d}r}\right)\right] &\leq \E\left[\sup\limits_{t\in[0,T]}\left(\E\left[|\xi|^{\tilde{p}} e^{\tilde{p}\int_0^T\rho_r{\rm d}r}\bigg|\F_t\right]\right)^{\frac{p}{\tilde{p}}}\right]\\
&\leq \E\left[|\xi|^p e^{p\int_0^T\rho_r{\rm d}r}\right].
\end{split}
\end{align*}
Hence, \eqref{supyt} comes true when \eqref{pxi} holds with $\rho_\cdot\geq \mu_\cdot+\frac{{\theta}}{2(p-1)}\nu_\cdot^2$ for some constant $\theta> 1$. \vspace{0.2cm}

To be conclusion, we have verified that if there exists an $(\F_t)$-progressively measurable real-valued process $(\rho_t)_{t\in\T}$ satisfying $\int_0^T |\rho_t| \dif t<+\infty$ such that $\xi\in L_T^p(\rho_\cdot;\R)$ and $\rho_\cdot\geq \mu_\cdot+\frac{{\theta}}{2(p-1)}\nu_\cdot^2$ for some ${\theta}>1$, then the linear BSDE \eqref{linearBSDE} admits an adapted solution $(y_t,z_t)_{t \in[0,T]}$ such that the value process $y_\cdot$ belongs to the weighted space $S_T^p(\rho_\cdot;\R)$. It should be especially emphasised that for the case of $p>2$, the process $z_\cdot$ does not necessarily belong to the weighted space $M_T^p(\rho_\cdot;\R)$ under the above-mentioned conditions. However, if $\rho_\cdot\geq \mu_\cdot+\frac{{\theta}}{2[1\wedge(p-1)]}\nu_\cdot^2$ for some ${\theta}>1$, then by It\^{o}'s formula it can be proved that $z_\cdot$ belongs to $M_T^p(\rho_\cdot;\R),$ and then $(y_\cdot,z_\cdot)$ belongs to $H_T^p(\rho_\cdot;\R^2)$. See the subsequent \cref{pro2.01} for details. In particular, when $\mu_\cdot$ is a nonnegative process, if we take $\rho_\cdot= \beta\mu_\cdot+\frac{{\theta}}{2[1\wedge(p-1)]}\nu_\cdot^2$ for some $\beta\geq 1$ and ${\theta}> 1$, then it comes to the case studied in \citet{LiFan2024SD}.\vspace{0.2cm}

The above argument further inspires us to study the existence and uniqueness of an adapted solution for a general nonlinear BSDE \eqref{BSDE1.1} with possibly unbounded stochastic coefficients $\mu_\cdot$ and $\nu_\cdot$ in a weighted $L^p~(p>1)$ space featuring a more general weighted factor $e^{\int_0^t\rho_r{\rm d}r}$ than one employed in \citet{LiFan2024SD}. More specifically, the process $\rho_\cdot$ is not necessarily nonnegative and it only needs to satisfy $\int_0^\tau |\rho_t| \dif t<+\infty$ and $\rho_\cdot\geq \mu_\cdot+\frac{\theta}{2[1\wedge(p-1)]}\nu_\cdot^2$ for some constant $\theta>1$.\vspace{0.2cm}

\subsection{A priori estimates}

In this subsection, we will establish several a priori estimates on the weighted $L^p$ solution to BSDEs with possibly unbounded stochastic coefficients. These estimates generalize some corresponding results obtained in existing literature, such as Propositions 2.3 and 2.4 in \citet{LiFan2024SD}, Proposition 2.2 in \citet{Li2025}, Proposition 2.1 in \citet{Xiao2015}, and Proposition 3.2 in \citet{Briand2003}, and play an important role in the proof of the main results of this paper.

To illustrate them, it is necessary to introduce the following assumption on the generator $g$.

\begin{enumerate}
\renewcommand{\theenumi}{(A)}
\renewcommand{\labelenumi}{\theenumi}
\item\label{A:A} There exist three $(\F_t)$-progressively measurable processes $\bar\rho_t(\omega):\Omega \times [0,\tau]\mapsto \R$, $\bar\mu_t(\omega):\Omega \times [0,\tau]\mapsto \R$ and $\bar\nu_t(\omega):\Omega \times [0,\tau]\mapsto \R_+$ satisfying $\int_{0}^{\tau}(|\bar\rho_t|+|\bar\mu_t|+\bar\nu_t^2){\rm d}t<+\infty$ and
    $$
    \bar{\rho}_\cdot\geq \bar{\mu}_\cdot+ \frac{\theta}{2[1\wedge(\bar{p}-1)]}{\bar\nu_\cdot}^2
    $$
   for two constants $\bar{p}>1$ and $\theta>1$ such that for each $(y,z)\in \R^k\times\R^{k\times{d}}$, we have
    $$
    \left<\hat{y},~g(\omega,t,y,z)\right>\leq f_{t}(\omega)+\bar\mu_t(\omega)|y|+\bar\nu_t(\omega)|z|, \ \ t\in [0,\tau(\omega)],
    $$
    where $(f_t)_{t\in [0,\tau]}$ is an $(\F_t)$-progressively measurable nonnegative real-valued process with $$\E\left[\left(\int_{0}^{\tau}e^{ \int_{0}^{t}\bar{\rho}_r{\rm d}r}f_t{\rm d}t\right)^{\bar{p}}\right]<+\infty.$$
\end{enumerate}
\begin{pro}\label{pro2.01}
Assume that the generator $g$ satisfies assumption (A) and $(y_t,z_t)_{t\in [0,\tau]}$ is an adapted solution of BSDE \eqref{BSDE1.1} such that $y_\cdot\in S_\tau^{\bar{p}}(\bar{\rho}_\cdot;\R^k)$. Then $z_\cdot\in M_\tau^{\bar{p}}(\bar{\rho}_\cdot;\R^{k\times d})$, and there exists a constant $C_{\bar{p},\theta}^1>0$ depending only on $\bar{p}$ and $\theta$ such that for each $0\leq r\leq t< +\infty$,
$$
\E\left[\left(\int_{t\wedge\tau}^{\tau}e^{2\int_{0}^{s}\bar{\rho}_{r}{\rm d}r}|z_s|^2{\rm d}s\right)^{\frac{\bar{p}}{2}}\bigg|\F_{r\wedge\tau}\right]
\leq
C_{p,\theta}^1\E\left[\sup_{s\in[t\wedge\tau,\tau]}\left(e^{\bar{p} \int_{0}^{s}\bar{\rho}_r{\rm d}r}|y_s|^{\bar{p}}\right) +\left(\int_{t\wedge\tau}^{\tau}e^{ \int_{0}^{s}\bar{\rho}_r{\rm d}r}f_s{\rm d}s\right)^{\bar{p}}\bigg|\F_{r\wedge\tau}\right].\vspace{0.2cm}
$$
\end{pro}

\begin{proof}[{\bf Proof.}]
Define $\bar{y}_t=e^{\int_{0}^{t}\bar{\rho}_r{\rm d}r}y_t$, $\bar{z}_t=e^{\int_{0}^{t}\bar{\rho}_r{\rm d}r}z_t$ and $\bar{f}_t=e^{\int_{0}^{t}\bar{\rho}_r{\rm d}r}f_t$ for each $t\in [0,\tau]$. Then $(\bar{y}_t)_{t\in[0,\tau]}\in S_\tau^{\bar{p}}(0;\R^k)$, and it suffices to prove that there exists a constant $C_{\bar{p},\theta}^1>0$ depending only on $\bar{p}$ and $\theta$ such that for each $0\leq r\leq t< +\infty$, we have
\begin{align}\label{2.01}
\E\left[\left(\int_{t\wedge\tau}^{\tau}|\bar{z}_s|^2{\rm d}s\right)^{\frac{\bar{p}}{2}}\bigg|\F_{r\wedge\tau}\right]
\leq
C_{\bar{p},\theta}^1\E\left[\sup_{s\in[t\wedge\tau,\tau]}|\bar{y}_s|^{\bar{p}} +\left(\int_{t\wedge\tau}^{\tau}\bar{f}_s{\rm d}s\right)^{\bar{p}}\bigg|\F_{r\wedge\tau}\right].
\end{align}
For each integer $n\geq1$, define the following $(\F_t)$-stopping time
$$\tau_{n}:=\inf \left\{t \geq0: \int_{0}^{t}|\bar{z}_s|^{2} {\rm d}s \geq n\right\}\wedge \tau,$$
with convention that $\inf \emptyset=+\infty$. Applying It\^{o}'s formula to $|\bar{y}_t|^2$ yields that for each $t\geq0$ and $n\geq1$,
\begin{align}\label{2.1}
\begin{split}
&|\bar{y}_{t\wedge\tau_n}|^2+\int_{t\wedge\tau_n}^{\tau_n} |\bar{z}_s|^2{\rm d}s+2 \int_{t\wedge\tau_n}^{\tau_n} \bar{\rho}_s|\bar{y}_s|^2{\rm d}s\\
&\ \ =|\bar{y}_{\tau_n}|^2+2\int_{t\wedge\tau_n}^{\tau_n} e^{ \int_{0}^{s}\bar{\rho}_r{\rm d}r}\langle \bar{y}_s,~g(s,e^{-\int_{0}^{s}\bar{\rho}_r{\rm d}r}\bar{y}_s,e^{-\int_{0}^{s}\bar{\rho}_r{\rm d}r}\bar{z}_s)\rangle{\rm d}s-2\int_{t\wedge\tau_n}^{\tau_n} \langle \bar{y}_s,\bar{z}_s{\rm d}B_s\rangle.
\end{split}
\end{align}
In light of assumption \ref{A:A} and inequality $2ab\leq \frac{\theta}{1\wedge(\bar{p}-1)} a^2+\frac{1\wedge(\bar{p}-1)}{\theta}b^2$, we have
\begin{align}\label{2.2}
\begin{split}
&2e^{\int_{0}^{t}\bar{\rho}_r{\rm d}r}\langle \bar{y}_t, ~g(t, e^{-\int_{0}^{t}\bar{\rho}_r{\rm d}r}\bar{y}_t, e^{-\int_{0}^{t}\bar{\rho}_r{\rm d}r}\bar{z}_t)\rangle
\leq 2\bar\mu_{t}|\bar{y}_t|^2+2\bar\nu_t|\bar{y}_t||\bar{z}_t|
+2|\bar{y}_t|\bar{f}_t\\
&\ \ \leq 2\bar\mu_{t}|\bar{y}_t|^2+\frac{\theta}{1\wedge(\bar{p}-1)} \bar\nu_{t}^2|\bar{y}_t|^2+\frac{1\wedge(\bar{p}-1)}{\theta}
|\bar{z}_t|^2+2|\bar{y}_t|\bar{f}_t, \ \ t\geq0.
\end{split}
\end{align}
Combining \eqref{2.1} and \eqref{2.2} yields that for each $t\geq0$ and $n\geq1$,
\begin{align}\label{2.3}
\begin{split}
&|\bar{y}_{t\wedge\tau_n}|^2+\left(1-\frac{1\wedge(\bar{p}-1)}{\theta}\right) \int_{t\wedge\tau_n}^{\tau_n} |\bar{z}_s|^2{\rm d}s+ \int_{t\wedge\tau_n}^{\tau_n} (2\bar{\rho}_s-2\bar{\mu}_s-\frac{\theta}{1\wedge(\bar{p}-1)}{\bar\nu_s}^2)
|\bar{y}_s|^2{\rm d}s\\
&\ \ \leq|\bar{y}_{\tau_n}|^2+2\int_{t\wedge\tau_n}^{\tau_n} |\bar{y}_s|\bar{f}_s{\rm d}s-2\int_{t\wedge\tau_n}^{\tau_n}\langle \bar{y}_s,\bar{z}_s{\rm d}B_s\rangle.
\end{split}
\end{align}
Furthermore, it follows from the inequality $2ab\leq a^2+b^2$ that for each $t\geq0$ and $n\geq1$,
\begin{align*}
2\int_{t\wedge\tau_n}^{\tau_n} |\bar{y}_s|\bar{f}_s{\rm d}s\leq \sup_{s\in[t\wedge\tau_n,\tau_n]}|\bar{y}_s|^2
+\left(\int_{t\wedge\tau_n}^{\tau_n}\bar{f}_s{\rm d}s\right)^2.
\end{align*}
Putting the last inequality into \eqref{2.3} and in light of the condition that $2\bar{\rho}_t-2\bar{\mu}_t-\frac{\theta}{1\wedge(\bar{p}-1)}
{\bar{\nu}_t}^2\geq0$, we can conclude that there exists a constant $C_{\bar{p}}^\theta$ depending only on $\bar{p}$ and $\theta$ such that for each $n\geq1$,
\begin{align}\label{2.4}
\left(\int_{t\wedge\tau_n}^{\tau_n}|\bar{z}_s|^2{\rm d}s\right)^{\frac{\bar{p}}{2}}
\leq C_{\bar{p}}^\theta \left[\sup_{s\in[t\wedge\tau_n,\tau_n]}|\bar{y}_s|^{\bar{p}}
+\left(\int_{t\wedge\tau_n}^{\tau_n}\bar{f}_s{\rm d}s\right)^{\bar{p}}+\left|\int_{t\wedge \tau_n}^{\tau_n} \langle \bar{y}_s,\bar{z}_s{\rm d}B_s\rangle\right|^{\frac{\bar{p}}{2}}\right], \ t\geq0.
\end{align}
Moreover, by the Burkholder--Davis--Gundy (BDG for short) inequality (see Proposition 2.5 in \citet{HuLiWen2025}) we know that for each $n\geq m\geq1$ and $0\leq r\leq t< +\infty$,
\begin{align}\label{2.501}
\begin{split}
&C_{\bar{p}}^\theta\E\left[\left|\int_{t\wedge \tau_n}^{\tau_n} \langle \bar{y}_s,\bar{z}_s{\rm d}B_s\rangle\right|^{\frac{\bar{p}}{2}}\bigg|\F_{r\wedge\tau_m}\right]
\leq d_{\bar{p},\theta}\E\left[\left(\int_{t\wedge\tau_n}^{\tau_n} |\bar{y}_s|^2|\bar{z}_s|^2{\rm d}s\right)^{\frac{\bar{p}}{4}}\bigg|\F_{r\wedge\tau_m}\right]\\
&\ \ \leq \frac{d_{\bar{p},\theta}^2}{2}\E\left[\sup_{s\in[t\wedge\tau_n,\tau_n]}
|\bar{y}_s|^{\bar{p}}\bigg|\F_{r\wedge\tau_m}\right]
+\frac{1}{2}\E\left[\left(\int_{t\wedge\tau_n}^{\tau_n} |\bar{z}_s|^2{\rm d}s\right)^{\frac{\bar{p}}{2}}\bigg|\F_{r\wedge\tau_m}\right],
\end{split}
\end{align}
where $d_{\bar{p},\theta}>0$ is a constant depending only on $\bar{p}$ and $\theta$.
Taking the conditional expectation with respect to $\F_{r\wedge\tau_m}$ on both sides of \eqref{2.4} and sending $n\rightarrow \infty$, it can be derived from \eqref{2.501}, Fatou's lemma and Lebesgue's dominated convergence theorem that there exists a constant $C_{\bar{p},\theta}^1>0$ depending only on $\bar{p}$ and $\theta$ such that for each $0\leq r\leq t<+\infty$ and $m\geq1$,
\begin{align*}
\E\left[\left(\int_{t\wedge\tau}^{\tau}|\bar{z}_s|^2{\rm d}s\right)^{\frac{\bar{p}}{2}}\bigg|\F_{r\wedge\tau_m}\right]
\leq C_{\bar{p},\theta}^1 \E\left[\sup_{s\in[t\wedge\tau,\tau]}|\bar{y}_s|^{\bar{p}}
+\left(\int_{t\wedge\tau}^{\tau}\bar{f}_s{\rm d}s\right)^{\bar{p}}\bigg|\F_{r\wedge\tau_m}\right].
\end{align*}
Finally, the desired assertion \eqref{2.01} follows by sending $m\rightarrow\infty$ and using the martingale convergence theorem (see Corollary A.9 in Appendix C of \citet{Oksendal2005}) on both sides of the last inequality.
\end{proof}

\begin{pro}\label{pro2.02}
Let the assumptions of \cref{pro2.01} hold. Then there exists a constant $C_{\bar{p},\theta}^2 >0$ depending only on $\bar{p}$ and $\theta$ such that for each $0\leq r\leq t< +\infty$,
\begin{align*}
\begin{split}
&\E\left[\sup_{s\in[t\wedge\tau,\tau]}\left(e^{\bar{p}
\int_{0}^{s}\bar{\rho}_{r}{\rm d}r}|y_s|^{\bar{p}}\right)\bigg|\F_{r\wedge\tau}\right]
+\E\left[\int_{t\wedge\tau}^{\tau}\bar{p}e^{\bar{p} \int_{0}^{s}\overline{\rho}_r{\rm d}r}(\bar{\rho}_s-\bar{\mu}_s-\frac{\theta}
{2[1\wedge(\bar{p}-1)]}{\bar\nu_s}^2)|y_s|^{\bar{p}}{\rm d}s\bigg|\F_{r\wedge\tau}\right]\\
&\ \  \leq
C_{\bar{p},\theta}^2\E\left[e^{\bar{p}\int_{0}^{\tau}\bar{\rho}_{r}{\rm d}r}|\xi|^{\bar{p}}+\left(
\int_{t\wedge\tau}^{\tau}e^{ \int_{0}^{s}\bar{\rho}_r{\rm d}r}f_s{\rm d}s\right)^{\bar{p}}\bigg|\F_{r\wedge\tau}\right].
\end{split}
\end{align*}
\end{pro}

\begin{proof}[{\bf Proof.}]
As in the proof of the \cref{pro2.01}, we define $\bar{y}_t=e^{\int_{0}^{t}\bar{\rho}_r{\rm d}r}y_t$, $\bar{z}_t=e^{\int_{0}^{t}\bar{\rho}_r{\rm d}r}z_t$ and $\bar{f}_t=e^{\int_{0}^{t}\bar{\rho}_r{\rm d}r}f_t$ for each $t\in [0,\tau]$. Then $(\bar{y}_t,\bar{z}_t)_{t\in[0,\tau]} \in H_\tau^{\bar{p}}(0;R^k\times R^{k\times d})$,  and it suffices to prove that there exists a constant $C_{\bar{p},\theta}^2 >0$ depending only on $\bar{p}$ and $\theta$ such that for each $0\leq r\leq t< +\infty$,
\begin{align}\label{2.600}
\begin{split}
&\E\left[\sup_{s\in[t\wedge\tau,\tau]}|\bar{y}_s|^{\bar{p}}\bigg|
\F_{r\wedge\tau}\right]
+\E\left[\int_{t\wedge\tau}^{\tau}\bar{p}(\bar{\rho}_s-\bar{\mu}_s
-\frac{\theta}{2[1\wedge(\bar{p}-1)]}{\bar\nu_s}^2)|\bar{y}_s|^{\bar{p}}{\rm d}s\bigg|\F_{r\wedge\tau}\right]\\
&\ \ \leq
C_{\bar{p},\theta}^2\E\left[|\bar{\xi}|^{\bar{p}}+\left(
\int_{t\wedge\tau}^{\tau}\bar{f}_s{\rm d}s\right)^{\bar{p}}\bigg|\F_{r\wedge\tau}\right].
\end{split}
\end{align}
By applying It\^{o}-Tanaka's formula to $|\bar{y}_t|^{\bar{p}}$ and assumption \ref{A:A}, we obtain that for each $t\geq0$,
\begin{align}\label{2.60}
\begin{split}
&|\bar{y}_{t\wedge\tau}|^{\bar{p}}+c(\bar{p})\int_{t\wedge\tau}^{\tau} |\bar{y}_s|^{\bar{p}-2}{\bf 1}_{|\bar{y}_s|\neq0}|\bar{z}_s|^2{\rm d}s+\bar{p} \int_{t\wedge\tau}^{\tau} \bar{\rho}_s|\bar{y}_s|^{\bar{p}}{\rm d}s\\
& \ \ \ \ \ \leq |\bar{\xi}|^{\bar{p}}+\bar{p}\int_{t\wedge\tau}^{\tau} (\bar\mu_{s}|\bar{y}_s|^{\bar{p}}+\bar\nu_s
|\bar{y}_s|^{\bar{p}-1}|\bar{z}_s|+|\bar{y}_s|^{\bar{p}-1}\bar{f}_{s}){\rm d}s
-\bar{p}\int_{t\wedge\tau}^{\tau} |\bar{y}_s|^{\bar{p}-2}{\bf 1}_{|\bar{y}_s|\neq0}\langle \bar{y}_s,\bar{z}_s{\rm d}B_s\rangle,
\end{split}
\end{align}
where $c(\bar{p}):=\frac{\bar{p}[(\bar{p}-1)\wedge1]}{2}$.

It can be verified that $\{M_t=\int_{0}^{t} |\bar{y}_s|^{\bar{p}-2}{\bf 1}_{|\bar{y}_s|\neq0}\langle \bar{y}_s,\bar{z}_s{\rm d}B_s\rangle\}_{t\in[0,\tau]}$ is a uniformly integrable martingale. Indeed, by $(\bar{y}_t,\bar{z}_t)_{t\in[0,\tau]} \in H_\tau^{\bar{p}}(0;\R^k\times \R^{k\times d})$, the BDG
inequality (see Theorem 1 in \citet{Ren2008BDG}) and Young's inequality, we know that
\begin{align}\label{2.8}
\begin{split}
&\E\left[\sup_{t\in[0,\tau]}\bigg|\int_{0}^{t}|\bar{y}_s|^{\bar{p}-2}{\bf 1}_{|\bar{y}_s|\neq0}\langle \bar{y}_{s}, \bar{z}_{s} {\rm d} B_{s}\rangle\bigg|\right]
\leq 2\sqrt{2}\E\left[\left(\int_{0}^{\tau}|\bar{y}_s|^{2\bar{p}-2}{\bf 1}_{|\bar{y}_s|\neq0}|\bar{z}_{s}|^2 {\rm d}{s}\right)^{\frac{1}{2}}\right]\\
&\ \ \ \leq 2\sqrt{2}\E\left[\left(\sup_{s\in[0,\tau]}|\bar{y}_s|^{\bar{p}-1}\right)
\left(\int_{0}^{\tau}|\bar{z}_{s}|^2{\rm d}{s}\right)^{\frac{1}{2}}\right]\\
&\ \ \ \leq \frac{2\sqrt{2}(\bar{p}-1)}{\bar{p}}\E\left[\sup_{s\in[0,\tau]}
|\bar{y}_s|^{\bar{p}}\right]+\frac{2\sqrt{2}}{\bar{p}}
\E\left[\left(\int_{0}^{\tau}|\bar{z}_{s}|^2{\rm d}{s}\right)^{\frac{\bar{p}}{2}}\right]<+\infty.
\end{split}
\end{align}
Moreover, in light of assumption \ref{A:A}, by Young's inequality and H\"{o}lder's inequality we deduce that
\begin{align}\label{uyvz}
\begin{split}
&\int_0^{\tau}(|\bar{\mu}_s||\bar{y}_s|^{\bar{p}}+\bar\nu_s
|\bar{y}_s|^{\bar{p}-1}|\bar{z}_s|+|\bar{y}_s|^{\bar{p}-1}\bar{f}_{s}){\rm d}s\\
&\ \ \leq  \left(\sup\limits_{s\in [0,\tau]}|\bar{y}_s|^{\bar{p}}\right)\int_0^\tau |\bar{\mu}_s|{\rm d}s+\left(\sup\limits_{s\in [0,\tau]}|\bar{y}_s|^{\bar{p}-1}\right) \int_0^\tau \bar\nu_s |\bar{z}_s|{\rm d}s+ \left(\sup\limits_{s\in [0,\tau]}|\bar{y}_s|^{\bar{p}-1}\right) \int_0^\tau \bar{f}_{s}{\rm d}s\\
&\ \ \leq \left(\sup\limits_{s\in [0,\tau]}|\bar{y}_s|^{\bar{p}}\right)\int_0^\tau |\bar{\mu}_s|{\rm d}s +\frac{2(\bar{p}-1)}{\bar{p}}\sup\limits_{s\in [0,\tau]}|\bar{y}_s|^{\bar{p}}\\
&\ \ \ \ \ \ +\frac{1}{\bar{p}}\left(\int_0^\tau v_s^2{\rm d}s \right)^{\frac{\bar{p}}{2}}\left(\int_0^\tau |\bar{z}_s|^2{\rm d}s \right)^{\frac{\bar{p}}{2}}+\frac{1}{\bar{p}}\left(\int_0^\tau \bar{f}_s{\rm d}s\right)^{\bar{p}}<+\infty.
\end{split}
\end{align}
Combining \eqref{2.60}-\eqref{uyvz} yields that for each $t\geq0$,
\begin{align}\label{2.09*}
\int_{t\wedge\tau}^{\tau} |\bar{y}_s|^{\bar{p}-2}{\bf 1}_{|\bar{y}_s|\neq0}|\bar{z}_s|^2{\rm d}s<+\infty.
\end{align}
Furthermore, using the inequality $\bar{p}ab\leq \frac{{\bar{p}}^2 \theta}{4c(\bar{p})} a^2+\frac{c(\bar{p})}{\theta}b^2$ we obtain
\begin{align}\label{2.10}
\bar{p}\bar\nu_s|\bar{y}_s|^{\bar{p}-1}|\bar{z}_s|&=\bar{p}(\bar\nu_s
|\bar{y}_s|^{\frac{\bar{p}}{2}})(|\bar{y}_s|^{\frac{\bar{p}}{2}-1}|\bar{z}_s|)
\leq \frac{{\bar{p}}^2 \theta}{4c(\bar{p})} \bar\nu_{s}^2|\bar{y}_s|^{\bar{p}}+\frac{c(\bar{p})}{\theta}
|\bar{y}_s|^{\bar{p}-2}{\bf 1}_{|\bar{y}_s|\neq0}|\bar{z}_s|^2, \ \ s\in[0,\tau].
\end{align}
In light of \eqref{2.09*}, putting \eqref{2.10} into \eqref{2.60} yields that for each $t \geq0$,
\begin{align}\label{2.11}
\begin{split}
&|\bar{y}_{t\wedge\tau}|^{\bar{p}}+(1-\frac{1}{\theta})c(\bar{p})
\int_{t\wedge\tau}^{\tau} |\bar{y}_s|^{\bar{p}-2}{\bf 1}_{|\bar{y}_s|\neq0}|\bar{z}_s|^2{\rm d}s
+\bar{p} \int_{t\wedge\tau}^{\tau} \left[\bar{\rho}_s-(\bar{\mu}_s+\frac{\theta}{2[1\wedge(\bar{p}-1)]}
{\bar\nu_s}^2)\right]|\bar{y}_s|^{\bar{p}}{\rm d}s\\
&\ \ \leq |\bar{\xi}|^{\bar{p}}+\bar{p}\int_{t\wedge\tau}^{\tau} |\bar{y}_s|^{\bar{p}-1}\bar{f}_s{\rm d}s
-\bar{p}\int_{t\wedge\tau}^{\tau} |\bar{y}_s|^{\bar{p}-2}{\bf 1}_{|\bar{y}_s|\neq0}\langle \bar{y}_s,\bar{z}_s{\rm d}B_s\rangle.
\end{split}
\end{align}
Then, in light of \eqref{2.8}, taking the conditional mathematical expectation on both sides of \eqref{2.11} yields that for each $0\leq r\leq t<+\infty$,
\begin{align}\label{2.12}
\E\left[\int_{t\wedge\tau}^{\tau} |\bar{y}_s|^{\bar{p}-2}{\bf 1}_{|\bar{y}_s|\neq0}|\bar{z}_s|^2{\rm d}s\bigg|\F_{r\wedge\tau}\right]\leq \frac{\theta}{(\theta-1)c(\bar{p})}\E\left[|\bar{\xi}|^{\bar{p}}+\bar{p}
\int_{t\wedge\tau}^{\tau} |\bar{y}_s|^{\bar{p}-1}\bar{f}_s{\rm d}s\bigg|\F_{r\wedge\tau}\right].
\end{align}

On the other hand, by virtue of the conditional BDG inequality in \cite[Theorem 1]{Ren2008BDG} and the elementary inequality $2ab\leq a^2+b^2$, we also deduce that for each $0\leq r\leq t<+\infty$,
\begin{align*}
&\bar{p}\E\left[\sup_{u \in [t\wedge\tau,\tau]}\bigg|\int_{u}^{\tau}|\bar{y}_s|^{\bar{p}-2}{\bf 1}_{|\bar{y}_s|\neq0}\langle \bar{y}_{s}, \bar{z}_{s} {\rm d} B_{s}\rangle\bigg| \bigg|\F_{r\wedge\tau}\right]\\
&\ \  \leq
2\sqrt{2}\bar{p}\E\left[\left(\int_{t\wedge\tau}^{\tau}|\bar{y}_s|^{2\bar{p}-2}{\bf 1}_{|\bar{y}_s|\neq0}|\bar{z}_{s}|^2 {\rm d}{s}\right)^{\frac{1}{2}}\bigg|\F_{r\wedge\tau}\right]\\
&\ \ \leq 2\sqrt{2}\bar{p}\E\left[\left(\sup_{s\in[t\wedge\tau,\tau]}
|\bar{y}_s|^{\frac{\bar{p}}{2}}\right)\left(\int_{t\wedge\tau}^{\tau}
|\bar{y}_s|^{\bar{p}-2}{\bf 1}_{|\bar{y}_s|\neq0}|\bar{z}_{s}|^2{\rm d}{s}\right)^{\frac{1}{2}}\bigg|\F_{r\wedge\tau}\right]\\
&\ \ \leq \frac{1}{2}\E\left[\sup_{s\in[t\wedge\tau,\tau]}|\bar{y}_s|^{\bar{p}}
\bigg|\F_{r\wedge\tau}\right]+4\bar{p}^2
\E\left[\int_{t\wedge\tau}^{\tau}|\bar{y}_s|^{\bar{p}-2}{\bf 1}_{|\bar{y}_s|\neq0}|\bar{z}_{s}|^2{\rm d}{s}\bigg|\F_{r\wedge\tau}\right].
\end{align*}
Then, in light of the last inequality and the condition that $\bar{\rho}_t\geq \bar{\mu}_t+\frac{\theta}{2[1\wedge(\bar{p}-1)]}{\bar\nu}_t^2$ along with \eqref{2.11} and \eqref{2.12}, we have that for each $0\leq r\leq t< +\infty$,
\begin{align}\label{2.14}
\begin{split}
&\E\left[\sup_{s\in[t\wedge\tau,\tau]}|\bar{y}_s|^{\bar{p}}\bigg|
\F_{r\wedge\tau}\right]
+\E\left[\int_{t\wedge\tau}^{\tau}\bar{p}(\bar{\rho}_s-\bar{\mu}_s
-\frac{\theta}{2[1\wedge(\bar{p}-1)]}{\bar\nu_s}^2)|\bar{y}_s|^{\bar{p}}{\rm d}s\bigg|\F_{r\wedge\tau}\right]\\
&\ \ \leq \left(2+\frac{8\bar{p}^2 \theta}{(\theta-1)c(\bar{p})}\right)\E\left[|\bar{\xi}|^{\bar{p}}
+\bar{p}\int_{t\wedge\tau}^{\tau} |\bar{y}_s|^{\bar{p}-1}\bar{f}_s{\rm d}s\bigg|\F_{r\wedge\tau}\right].
\end{split}
\end{align}
Letting $M_{\bar{p},\theta}=\bar{p}(2+\frac{8\bar{p}^2 \theta}{(\theta-1)c(\bar{p})})$, by virtue of Young's inequality we deduce that for each $0\leq r\leq t< +\infty$,
\begin{align}\label{2.15}
\begin{split}
&M_{\bar{p},\theta}\E\left[\int_{t\wedge\tau}^{\tau} |\bar{y}_s|^{\bar{p}-1}\bar{f}_s{\rm d}s\bigg|\F_{r\wedge\tau}\right]\\
&\ \ \leq \E\left[\sup_{s\in[t\wedge\tau,\tau]}\left(\left(\frac{\bar{p}}{2\bar{p}-2}\right)^{\frac{\bar{p}-1}{\bar{p}}}|\bar{y}_s|^{\bar{p}-1}\right)\left(M_{\bar{p},\theta}\left(\frac{\bar{p}}{2\bar{p}-2}\right)^{\frac{1-\bar{p}}{\bar{p}}}\int_{t\wedge\tau}^{\tau} \bar{f}_s{\rm d}s\right)\bigg|\F_{r\wedge\tau}\right]\\
&\ \ \leq \frac{1}{2}\E\left[\sup_{s\in[t\wedge\tau,\tau]}|\bar{y}_s|^{\bar{p}}
\bigg|\F_{r\wedge\tau}\right]+\frac{M_{\bar{p},\theta}}{\bar{p}}
\left(\frac{\bar{p}}{2\bar{p}-2}\right)^{1-\bar{p}}\E\left[
\left(\int_{t\wedge\tau}^{\tau}\bar{f}_s{\rm d}s\right)^{\bar{p}}\bigg|\F_{r\wedge\tau}\right].
\end{split}
\end{align}
Thus, the desired assertion \eqref{2.600} follows from \eqref{2.14} and \eqref{2.15} and then \cref{pro2.02} is proved.
\end{proof}

Combining \cref{pro2.01,pro2.02}, we have the following proposition.

\begin{pro}\label{pro2.1}
Let the assumptions of \cref{pro2.01} hold. Then there exists a constant $C_{\bar{p},\theta}>0$ depending only on $\bar{p}$ and $\theta$ such that for each $0\leq r\leq t< +\infty$,
\begin{align*}
\begin{split}
&\E\left[\sup_{s\in[t\wedge\tau,\tau]}\left(e^{\bar{p}\int_{0}^{s}
\bar{\rho}_{r}{\rm d}r}|y_s|^{\bar{p}}\right)\bigg|\F_{r\wedge\tau}\right]+ \E\left[\left(\int_{t\wedge\tau}^{\tau}e^{2 \int_{0}^{s}\bar{\rho}_r{\rm d}r}|z_s|^2{\rm d}s\right)^\frac{\bar{p}}{2}\bigg|\F_{r\wedge\tau}\right]\\
&\hspace{0.6cm}+\E\left[\int_{t\wedge\tau}^{\tau}\bar{p}e^{\bar{p} \int_{0}^{s}\bar{\rho}_r{\rm d}r}(\bar{\rho}_s-\bar{\mu}_s-\frac{\theta}{2[1\wedge(\bar{p}-1)]}
{\bar\nu_s}^2)|y_s|^{\bar{p}}{\rm d}s\bigg|\F_{r\wedge\tau}\right]\\
&\ \ \leq
C_{\bar{p},\theta}\E\left[e^{\bar{p}\int_{0}^{\tau}\bar{\rho}_{r}{\rm d}r}|\xi|^{\bar{p}}+\left(
\int_{t\wedge\tau}^{\tau}e^{ \int_{0}^{s}\bar{\rho}_r{\rm d}r}f_s{\rm d}s\right)^{\bar{p}}\bigg|\F_{r\wedge\tau}\right].
\end{split}
\end{align*}
\end{pro}

\begin{rmk}\label{rmk:2.5}
We make the following two comments.
\begin{itemize}
\item [(i)] For the case of $\bar{p}=2$, if the term $2 |\bar{y}_s|\bar{f}_s$ of \eqref{2.3} and \eqref{2.11} is left to be unaltered in the proof of \cref{pro2.01,pro2.02}, then it is straightforward to arrive at a similar result to that presented in (2.11) of \citet{Li2025}. This result demonstrates that under the assumptions of \cref{pro2.01} with $\bar{p}=2$, there exists a constant $C_\theta>0$ depending only on $\theta$ such that for each $0\leq r\leq t< +\infty$,
\begin{align*}
\begin{split}
&\E\left[\sup_{s\in[t\wedge\tau,\tau]}\left(e^{2 \int_{0}^{s}\bar{\rho}_r{\rm d}r}|y_s|^2\right)\bigg|\F_{r\wedge\tau}\right]+
\E\left[\int_{t\wedge\tau}^{\tau}e^{2 \int_{0}^{s}\bar{\rho}_r{\rm d}r}|z_s|^2{\rm d}s\bigg|\F_{r\wedge\tau}\right]\\
&\ \ \leq
C_\theta\E\left[e^{2 \int_{0}^{\tau}\bar{\rho}_r{\rm d}r}|\xi|^2+
\int_{t\wedge\tau}^{\tau}e^{2 \int_{0}^{s}\bar{\rho}_r{\rm d}r}|y_s|f_s{\rm d}s\bigg|\F_{r\wedge\tau}\right].
\end{split}
\end{align*}

\item [(ii)] Note that when $\beta\geq 1$ and $\bar{\mu}_\cdot$ is a nonnegative process, we have
$$
\bar{\rho}_\cdot:=\beta\bar{\mu}_\cdot+\frac{\theta}{2[1\wedge(p-1)]}
{\bar\nu}_\cdot^2\geq \bar{\mu}_\cdot+\frac{\theta}{2[1\wedge(p-1)]}{\bar\nu}_\cdot^2.
$$
In light of (i), we know that \cref{pro2.1} incorporates Proposition 2.4 in \citet{LiFan2024SD} and Proposition 2.2 in \cite{Li2025} as its particular case.
\end{itemize}
\end{rmk}

\section{Well-posedness for the weighted $L^p$ solutions}
\setcounter{equation}{0}
In this section, we will establish an existence and uniqueness theorem, a continuous dependence property and a stability theorem for the weighted $L^p$ solution of BSDE \eqref{BSDE1.1}. The generator $g$ satisfies a  stochastic monotonicity condition with a very general growth in the state variable $y$, and a stochastic Lipschitz continuity condition in the state variable $z$. In addition, we will present some corollaries, all of which can be derived from the existence and uniqueness theorem, as well as some remarks and examples to compare them with existing relevant findings.

In the rest of this paper, we always assume that $\beta\geq1$, $p>1$ and $\theta>1$ are three given constants, $\rho_t(\omega), \mu_t(\omega)$: $\Omega \times [0,\tau]\mapsto \R$ and $\nu_t(\omega)$: $\Omega \times [0,\tau]\mapsto \R_+$ are three $(\F_t)$-progressively measurable processes satisfying $\int_{0}^{\tau}\left(|\rho_t|+|\mu_t|+\nu_t^2\right){\rm d}t<+\infty$ and
$$
\rho_\cdot\geq \mu_\cdot+\frac{\theta}{2[1\wedge(p-1)]}\nu_\cdot^2.
$$
\subsection{Existence and uniqueness\vspace{0.2cm}}
Let us first introduce the following assumptions on the generator $g:\Omega\times[0,\tau]\times \R^k\times \R^{k\times d}\mapsto \R^k$.
\begin{enumerate}
\renewcommand{\theenumi}{(H1)}
\renewcommand{\labelenumi}{\theenumi}
\item\label{A:H1} $\E\left[\left(\int_0^\tau e^{\int_{0}^{s}\rho_r{\rm d}r}|g(s,0,0)|{\rm d}s\right)^p\right]<+\infty;$
\renewcommand{\theenumi}{(H2)}
\renewcommand{\labelenumi}{\theenumi}
\item\label{A:H2} ${\rm d}\mathbb{P}\times{\rm d} t-a.e.$, for each $z\in\R^{k\times d}$, $y \mapsto g(\omega,t,y,z)$ is continuous;
\renewcommand{\theenumi}{(H3)}
\renewcommand{\labelenumi}{\theenumi}
\item\label{A:H3} $g$ has a very general growth in $y$, i.e., there exists an $(\F_t)$-progressively measurable real-valued process $(\alpha_t)_{t\in[0,\tau]}$ satisfying $\essinf\limits_{t\in[0,\tau]}\alpha_t>0$ such that for each $r\in \R_+$, it holds that
\begin{align}\label{3.01}
\E\left[\int_0^\tau \alpha_t\psi(t,r\alpha_t){\rm d}t\right]<+\infty
\end{align}
with
$$
\psi(\omega,t,r):=\sup_{|y|\leq r} \left|g(\omega,t,y,0)-g(\omega,t,0,0)\right|, \ \ (\omega,t,r)\in \Omega\times [0,\tau(\omega)]\times \R_+;
$$
\renewcommand{\theenumi}{(H4)}
\renewcommand{\labelenumi}{\theenumi}
\item\label{A:H4} $g$ satisfies a stochastic monotonicity condition in $y$, i.e., for each $y_1, y_2\in\R^k$  and $z\in\R^{k\times d}$,
$$\left\langle y_1-y_2,g(\omega,t,y_1,z)-g(\omega,t,y_2,z)\right\rangle\leq \mu_t(\omega)|y_1-y_2|^2, \ \ t\in[0,\tau(\omega)];$$
\renewcommand{\theenumi}{(H5)}
\renewcommand{\labelenumi}{\theenumi}
\item\label{A:H5} $g$ satisfies a stochastic Lipschitz continuity condition in $z$, i.e., for each $y\in\R^k$  and $z_1,z_2\in\R^{k\times d}$,
$$\left|g(\omega,t,y,z_1)-g(\omega,t,y,z_2)\right|\leq \nu_t(\omega)|z_1-z_2|, \ \ t\in[0,\tau(\omega)].$$
\end{enumerate}
\begin{rmk}\label{rmk:3.0*}
Note that \ref{A:H2}, \ref{A:H4} and \ref{A:H5} come from \citet{Li2025} and \citet{LiFan2024SD}. However, the process $\mu_\cdot$ appearing in these two references can only take values in $\R_+$. Here, we would like to provide two typical examples satisfying assumption \ref{A:H4} with a general real-valued process $\mu_\cdot$. Let $k=1$, $(a_t)_{t\in [0,\tau]}$ be an $(\F_t)$-progressively measurable nonnegative process and $h_i(y), i=1,2,3$ be three first-order continuously differentiable functions such that $h'_1(y)\leq1$, $h'_2(y)\leq-1$ and $h'_3(y)\leq0$ for each $y\in\R$. For each $(\omega,t,y,z)\in \Omega\times [0,\tau(\omega)]\times\R\times\R^{1\times d}$, define\vspace{-0.1cm}
\begin{align*}
g(\omega,t,y,z):=\mu_t^{+}(\omega) h_1(y) +\mu_t^{-}(\omega) h_2(y)\vspace{-0.1cm}
\end{align*}
and\vspace{-0.2cm}
\begin{align*}
\bar{g}(\omega,t,y,z):=\mu_t(\omega)y +a_t(\omega) h_3(y).
\end{align*}
By Lagrange's Mean Value Theorem, it is easy to verify that both $g$ and $\bar g$ satisfy \ref{A:H4}.
\end{rmk}

Next, let us introduce several existing assumptions that are closely related to \ref{A:H1} and \ref{A:H3}.\vspace{0.1cm}

\begin{enumerate}
\renewcommand{\theenumi}{(H1a)}
\renewcommand{\labelenumi}{\theenumi}
\item\label{A:H1a} $\Dis \E\left[\left(\int_0^\tau e^{\int_0^s (\beta \mu_r^+ +\frac{\theta}{2[1\wedge(p-1)]}\nu_r^2)\dif r} |g(s,0,0)|{\rm d}s\right)^{p}\right]<+\infty$;
\renewcommand{\theenumi}{(H1b)}
\renewcommand{\labelenumi}{\theenumi}
\item\label{A:H1b} $\Dis \E\left[\left(\int_0^\tau |g(s,0,0)|{\rm d}s\right)^p\right]<+\infty$;
\renewcommand{\theenumi}{(H3a)}
\renewcommand{\labelenumi}{\theenumi}
\item\label{A:H3a}
For each $r\in \R_+$, it holds that
\begin{align*}
\E\left[\int_0^\tau e^{\beta \int_{0}^{t}\mu^+_s{\rm d}s}\psi(t,r\alpha_t){\rm d}t\right]<+\infty,
\end{align*}
where the function $\psi(\cdot,\cdot)$ and the process $(\alpha_t)_{t\in [0,\tau]}$ are defined in \ref{A:H3};
\renewcommand{\theenumi}{(H3b)}
\renewcommand{\labelenumi}{\theenumi}
\item\label{A:H3b}
For each $r\in \R_+$, it holds that
\begin{align*}
\E\left[\int_0^\tau \psi(t,r)\dif t\right]<+\infty,
\end{align*}
where the function $\psi(\cdot,\cdot)$ are defined in \ref{A:H3}.
\end{enumerate}

\begin{rmk}\label{rmk:3.1}
Regarding to these assumptions, we have the following two remarks.
\begin{itemize}
\item [(i)] \ref{A:H1a} is used in \citet{LiFan2024SD} and stronger than \ref{A:H1b} employed in classical literature studying the adapted solution for BSDEs in non-weighted spaces. Since the process $\rho_\cdot$ in \ref{A:H1} only needs to satisfy $\rho_\cdot\geq \mu_\cdot+\frac{\theta}{2[1\wedge(p-1)]}\nu_\cdot^2$ and the process $\mu_\cdot$ can take values in $\R$, \ref{A:H1} can be strictly weaker than \ref{A:H1a} and \ref{A:H1b}, particularly when $\mu_\cdot$ is a negative unbounded process.

\item [(ii)] As demonstrated in Example 3.8 of \citet{Li2025}, \ref{A:H3a} with $\mu^+_\cdot\equiv 0$ is strictly weaker than \ref{A:H3b} that was frequently used in the study of BSDEs with (weak) monotonicity generators such as \citet{Briand2003}, \citet{FanJiang2013}, \citet{Xiao2015} and so on. It is obvious that \ref{A:H3} is weaker than \ref{A:H3a}. However, up to now we do not know whether \ref{A:H3} is strictly weaker than \ref{A:H3a}.
\end{itemize}
\end{rmk}

The following theorem is one of the main results in this paper.

\begin{thm}\label{thm:3.1}
Let the generator $g$ satisfy assumptions \ref{A:H1}-\ref{A:H5}. Then, for each $\xi\in L_\tau^p(\rho_\cdot;\R^k)$, BSDE \eqref{BSDE1.1} admits a unique weighted $L^p$ solution $(y_t,z_t)_{t\in[0,\tau]}$ in the space of $H_\tau^p(\rho_\cdot;\R^{k}\times\R^{k\times d})$.
\end{thm}

\begin{proof}[The Proof of the uniqueness part.]
Assume that $(y_\cdot,z_\cdot)$ and $(y_\cdot',z_\cdot')$ are two weighted $L^p$ solutions in the space of $H_\tau^p(\rho_\cdot;\R^{k}\times\R^{k\times d})$ of BSDE \eqref{BSDE1.1}. The process $(y_\cdot-y_\cdot',z_\cdot-z_\cdot')$ is denoted by $(\tilde{y}_\cdot,\tilde{z}_\cdot)$. It is obvious that $(\tilde{y}_\cdot,\tilde{z}_\cdot)$ is a weighted $L^p$ solution in $H_\tau^p(\rho_\cdot;\R^{k}\times\R^{k\times d})$ of the following BSDE:
\begin{align}\label{2.5}
\tilde{y}_t=\int_t^\tau\tilde{g}(s,\tilde{y}_s,\tilde{z}_s){\rm d}s-\int_t^\tau\tilde{z}_s{\rm d}B_s, \ \ t\in[0,\tau],
\end{align}
where the generator $\tilde g$ is defined by
$$
\tilde{g}(t,y,z):=g(t,y+y_t',z+z_t')-g(t,y_t',z_t'),\ \ (t,y,z)\in [0,\tau]\times\R^k\times\R^{k\times d}.
$$
It follows from \ref{A:H4} and \ref{A:H5} that for each $(y,z)\in\R^k\times\R^{k\times{d}}$,
\begin{align*}
\begin{split}
\left<\hat{y},\tilde{g}(t,y,z)\right>&=\left<\hat{y},g(t,y+y_t',z+z_t')-g(t,y_t',z+z_t')+g(t,y_t',z+z_t')-g(t,y_t',z_t')\right>\\
&\leq \mu_t|y|+\nu_t|z|, \ \ t\in[0,\tau],
\end{split}
\end{align*}
which means that assumption (A) is satisfied by the generator $\tilde{g}$ of BSDE \eqref{2.5} with $\bar{p}:=p$, $\bar\mu_t:=\mu_t$, $\bar\nu_t:=\nu_t$, $\bar\rho_t:=\rho_t$ and $f_t\equiv 0$. Thus, it follows from \cref{pro2.1} with $r=t=0$ that
\begin{align*}
\begin{split}
&\E\left[\sup_{s\in[0,\tau]}\left(e^{p \int_{0}^{s}\rho_r{\rm d}r}|\tilde{y}_s|^p\right)\right]+\E\left[\left(\int_{0}^{\tau}e^{2 \int_{0}^{s}\rho_r{\rm d}r}|\tilde{z}_s|^2{\rm d}s\right)^{\frac{p}{2}}\right]=0.
\end{split}
\end{align*}
Therefore, $(\tilde{y}_t,\tilde{z}_t)_{t\in[0,\tau]}=(0,0)$. The proof of the uniqueness part is then completed.
\end{proof}

The proof of the existence part of \cref{thm:3.1} will be delineated in Section 4 due to its complexity. Here, we would like to briefly outline the proof.

\begin{proof}[Outline of the proof of the existence part of \cref{thm:3.1}.]
The main idea of the proof originates from \citet[Theorem 3.1]{Li2025} and \cite[Theorem 3.1]{LiFan2024SD}. However, due to the general weighted factor and the very general growth condition of the generator $g$ in $y$, we will encounter some new barriers in the proof of \cref{thm:3.1}, requiring us to improve original techniques and develop novel methods.\vspace{0.2cm}

We first study the special case that $p=2$ and the generator $g$ does not depend on $z$. More specifically, we aim to prove that for each $\xi\in L_\tau^2(\rho_\cdot;\R^{k})$ with $\rho_\cdot\geq \mu_\cdot$, the following BSDE
\begin{align}\label{eq:3.5}
  y_t=\xi+\int_t^\tau g(s,y_s){\rm d}s-\int_t^\tau z_s{\rm d}B_s, \ \ t\in[0,\tau]
\end{align}
admits a weighted $L^2$ solution in $H_\tau^2(\rho_\cdot;\R^{k}\times\R^{k\times d})$ under \ref{A:H1}-\ref{A:H4} in subsequent three steps.

(i) Under \ref{A:H2} and \ref{A:H4}, by virtue of
\cite[Proposition 4.1]{Li2025} we conclude that BSDE \eqref{eq:3.5} admits a unique weighted $L^2$ solution in $H_\tau^2(\rho_\cdot^+;\R^{k}\times\R^{k\times d})$ provided that there is a constant $K\geq 0$ such that
\begin{align*}
|\xi|\leq K\alpha_\tau \ \ {\rm and}\ \ |g(t,\cdot)|\leq Ke^{-t}\alpha_t, \ \  t\in[0,\tau].
\end{align*}

(ii) Under \ref{A:H2}-\ref{A:H4}, we prove that the conclusion of (i) remains true provided that there exists a nonnegative constant $K$ such that
\begin{align*}
|\xi|\leq K \alpha_\tau^3 \ \ {\rm and}\ \ |g(t,0)|\leq Ke^{-t}\alpha_t^3, \ \  t\in[0,\tau].
\end{align*}
In this step, by introducing a suitable auxiliary function $\theta_{r}^{\alpha_\cdot}$ we first construct an approaching sequence $\{g^n\}$ of the generator $g$ satisfying the conditions of (i) such that for each $n\geq 1$, BSDE$(\xi,\tau,g^n)$ admits a unique weighted $L^2$ solution $(y_t^n,z_t^n)_{t\in[0,\tau]}$ in $H_\tau^2(\rho_\cdot^+;\R^{k}\times\R^{k\times d})$. Then, with the help of \cref{pro2.1}, we verify that $\{(y_\cdot^n,z_\cdot^n)\}^{+\infty}_{n=1}$ is a Cauchy sequence in $H_\tau^2(\rho_\cdot^+;\R^{k}\times\R^{k\times d})$ and has a limit process $(y_\cdot,z_\cdot)$. Finally, by a delicate argument we verify that $(y_\cdot,z_\cdot)$ is the desired solution of BSDE \eqref{eq:3.5}, breaking through the restriction of the very general growth condition.

(iii) We remove the extra conditions in (ii) and prove the existence of a weighted $L^2$ solution to BSDE \eqref{eq:3.5} in the space of $H_\tau^2(\rho_\cdot;\R^{k}\times\R^{k\times d})$ under assumptions \ref{A:H1}-\ref{A:H4}. In this step, by introducing a truncation function as usual, we first construct the approximate sequence $\{\xi_n\}$ of terminal values and the approximate sequence $\{\bar g_n\}$ of generators satisfying the conditions of (ii) such that for each $n\geq 1$, there exists a unique weighted $L^2$ solution $(y_t^n,z_t^n)_{t\in[0,\tau]}$ of BSDE$(\xi_n,\tau,\bar g_n)$ in $H_\tau^2(\rho_\cdot^+;\R^{k}\times\R^{k\times d}) \subseteq H_\tau^2(\rho_\cdot;\R^{k}\times\R^{k\times d})$. Then by virtue of \cref{pro2.1} we prove that $\{(y_\cdot^n,z_\cdot^n)\}^{+\infty}_{n=1}$ is a Cauchy sequence in $H_\tau^2(\rho_\cdot;\R^{k}\times\R^{k\times d})$ and the limit process is exactly the desired solution of BSDE \eqref{eq:3.5}, further breaking through the restrictions of the general weighted space and the general growth condition.\vspace{0.2cm}

Next, under the general case that $g$ can depend on $z$, we aim to prove the existence part of \cref{thm:3.1} with $p=2$ and $\rho_\cdot\geq \mu_\cdot +\frac{\theta}{2}\nu_\cdot^2$ in subsequent two steps.

(iv) We prove that for each $\xi\in L_\tau^2(\rho_\cdot;\R^{k})$ and $V_\cdot\in M_\tau^2(\rho_\cdot;\R^{k\times d})$, the following BSDE
\begin{align}\label{eq:3.6}
  y_t=\xi+\int_t^\tau g(s,y_s,V_s){\rm d}s-\int_t^\tau z_s{\rm d}B_s, \ \ t\in[0,\tau].
\end{align}
admits a unique weighted $L^2$ solution in $H_\tau^2(\rho_\cdot;\R^{k}\times\R^{k\times d})$ under assumptions \ref{A:H1}-\ref{A:H5} with $p=2$.
In this step, based on the conclusion of (iii), we first derive that BSDE \eqref{eq:3.6} admits a unique adapted solution $(y_t,z_t)_{t\in[0,\tau]}$ in the space of $H_\tau^2(\rho_\cdot-\frac{\theta}{2}\nu_\cdot^2;\R^{k}\times\R^{k\times d})$. Then, by employing techniques analogous to those used in the first step of the proof of the existence part of \citet[Theorem 3.1]{Li2025} and combining with It\^{o}'s formula, Doob's martingale inequality and \cref{pro2.01}, we verify that the process $(y_t,z_t)_{t\in[0,\tau]}$ also belongs to the space of $H_\tau^2(\rho_\cdot;\R^{k}\times\R^{k\times d})$.

(v) Under assumptions \ref{A:H1}-\ref{A:H5} with $p=2$, by the conclusion of (iv) and the fixed point theorem we prove that for each $\xi\in L_\tau^2(\rho_\cdot;\R^{k})$, BSDE \eqref{BSDE1.1} admits a unique weighted $L^2$ solution $(y_t,z_t)_{t\in[0,\tau]}$ in the space of $H_\tau^2(\rho_\cdot;\R^{k}\times\R^{k\times d})$. To do this, we choose the process $z_\cdot$ in the unique weighted $L^2$ solution of BSDE \eqref{eq:3.6}
to be the image of $V_\cdot$ and construct a mapping $\Phi$ from $M_\tau^2(\rho_\cdot;\R^{k\times d})$ to itself such that $\Phi(V_\cdot)=z_\cdot$, and then prove that $\Phi$ is a strictly contractive mapping on $M_\tau^2(\rho_\cdot;\R^{k\times d})$.\vspace{0.2cm}

Finally, based on (v) we prove the existence part of \cref{thm:3.1}.

(vi) In this final step, we first construct the approximate sequence $\{\xi^n\}$ of terminal values and the approximate sequence $\{\bar g^n\}$ of generators satisfying \ref{A:H1}-\ref{A:H5} with $p=2$ via the same truncation function used in (iii) such that, in light of $\rho_\cdot\geq \mu_\cdot+\frac{\theta}{2[1\wedge(p-1)]}\nu_\cdot^2\geq \mu_\cdot+\frac{\theta}{2}\nu_\cdot^2$ and \cref{thm:3.1} with $p=2$, BSDE $(\xi^n,\tau,\overline{g}^n)$ admits a unique solution $(y_t^n,z_t^n)_{t\in[0,\tau]}$ in  $H_\tau^2(\rho_\cdot;\R^{k}\times\R^{k\times d})$ for each $n\geq 1$. Then, we use \cref{pro2.01,pro2.02} to prove that $(y_\cdot^n,z_\cdot^n)$ also belongs to $H_\tau^p(\rho_\cdot;\R^{k}\times\R^{k\times d})$. Finally, the desired conclusion follows by verifying that $\{(y_\cdot^n,z_\cdot^n)\}^{+\infty}_{n=1}$ is a Cauchy sequence in $H_\tau^p(\rho_\cdot;\R^{k}\times\R^{k\times d})$ and then taking the limit under the uniform convergence in probability for BSDE$(\xi^n,\tau,\overline{g}^n)$. During this procedure, some new difficulties arise naturally due to the general weighted space and the general growth condition, and are successfully overcome via some delicate arguments.\vspace{0.2cm}

We would like to especially mention that in order to ensure the validity of the procedure of taking limit in previous (ii), (iii) and (vi), it seems to be necessary to suppose that the process $(\alpha_t)_{t\in[0,\tau]}$ appearing in assumption \ref{A:H3} satisfies the condition of $\essinf\limits_{t\in[0,\tau]}\alpha_t>0$.
\end{proof}

The following proposition can be directly derived from \cref{thm:3.1}, and it generalizes Theorem 3.1 of \citet{LiFan2024SD}, where the growth condition of the generator $g$ in $y$ is \ref{A:H3a} rather than \ref{A:H3}.

\begin{pro}\label{pro3.1}
Assume that the generator $g$ satisfies assumptions \ref{A:H1a} and \ref{A:H2}-\ref{A:H5} with
$$
\rho_\cdot:=\beta \mu_\cdot^+ +\frac{\theta}{2[1\wedge(p-1)]}\nu_\cdot^2.
$$
Then for each $\xi\in L_\tau^p(\rho_\cdot;\R^k)$, BSDE \eqref{BSDE1.1} admits a unique weighted $L^p$ solution in $H_\tau^p(\rho_\cdot;\R^{k}\times\R^{k\times d})$.\vspace{0.1cm}
\end{pro}

\begin{rmk}\label{rmk:3.0}
\cref{pro3.1} can be compared with \citet[Theorem 2.1]{Bahlali2015}, where the terminal time $\tau$ is a finite positive constant, the generator $g$ is forced to have a super-linear growth in $y$ (see assumption (H.3) therein), and the stochastic coefficients are forced to have a certain exponential moment (see Example 4 therein). On the other hand, \cref{thm:3.1} strengthens \cref{pro3.1} by broadening the scope of the process $\rho_\cdot$ to be only bigger than $\mu_\cdot +\frac{\theta}{2[1\wedge(p-1)]}\nu_\cdot^2$. In particular, when $\mu_\cdot$ and $\rho_\cdot$ are processes which may take negative values, in light of (iii) of \cref{rmk:2.1}, the conclusion of \cref{thm:3.1} may hold for some non-integrable random variables. See \cref{ex3.3,ex3.2} blow for more details.\vspace{0.1cm}
\end{rmk}

The following simple example illustrates that in \cref{thm:3.1}, a growth condition like \ref{A:H3} seems to be necessary for existence of the adapted solution to BSDE \eqref{BSDE1.1}.

\begin{ex}\label{ex3.5}
Let $k=1$, $\tau\equiv+\infty$, $\xi=c$ for a constant $c\neq 0$, and
$$g(t,y,z):=-y,\ \ (t,y,z)\in [0,\tau]\times\R\times\R^{1\times d}.$$
It is obvious that $\xi\in L_\tau^2(0;\R)$ and this generator $g$ satisfies assumptions \ref{A:H1}, \ref{A:H2}, \ref{A:H4} and \ref{A:H5} with $g(\cdot,0,0)\equiv 0$ and $\rho_\cdot=\mu_\cdot=\nu_\cdot\equiv 0$. However, it is certain that this $g$ does not satisfy \ref{A:H3}. In fact, for any process $(\alpha_t)_{t\in [0,\tau]}$ satisfying $\essinf\limits_{t\in[0,\tau]}\alpha_t>0$, we have
$$
\RE r>0,\ \ \E\left[\int_0^\tau \alpha_t \psi(t,r\alpha_t){\rm d}t\right]=r\E\left[\int_0^\tau \alpha_t^2{\rm d}t\right] =+\infty.
$$
Now, we prove that BSDE$(\xi,\tau,g)$ does not admit an adapted solution in the space of $H_\tau^2(0;\R\times\R^{1\times d})$. Indeed, if BSDE$(\xi,\tau,g)$ has an adapted solution $(y_t,z_t)_{t\in[0,\tau]}\in H_\tau^2(0;\R\times\R^{1\times d})$, then by \cref{pro2.1} we know that $y_\cdot\in S^\infty(\R)$, and we have
\begin{align}\label{eq:3.7}
\RE T\geq 0,\ \ y_t=y_{T\wedge\tau}+\int_{t\wedge\tau}^{T\wedge\tau} g(s,y_s,z_s){\rm d}s-\int_{t\wedge\tau}^{T\wedge\tau} z_s{\rm d}B_s, \ \ t\in [0,T].
\end{align}
On the other hand, it follows from \citet[Theorem 3.3]{BriandandHu1998} that BSDE \eqref{eq:3.7} admits a unique bounded solution $(\bar y_\cdot, \bar z_\cdot)\equiv (0,0)$. Consequently, $(y_\cdot,z_\cdot)$ must be $(0,0)$ which contradicts $y_\tau=\xi=c\neq 0$.\vspace{0.1cm}
\end{ex}

Now, we present the following two assumptions \ref{A:H3c} and \ref{A:H3d} stronger than assumption \ref{A:H3}. We emphasize that these two assumptions can be more easily verified than assumption \ref{A:H3}. See the following \cref{ex:3.7,ex3.4} as well as those examples in the subsequent subsection for more details.

\begin{enumerate}
\renewcommand{\theenumi}{(H3c)}
\renewcommand{\labelenumi}{\theenumi}
\item\label{A:H3c} There exist two continuous and nondecreasing functions $\varphi(\cdot),~\bar{\varphi}(\cdot):\R_+\rightarrow\R_+$ satisfying that $\bar\varphi(\cdot)$ is convex and $\bar\varphi(0)=0$ along with three $(\F_t)$-progressively measurable nonnegative processes $(\tilde{\mu}_t)_{t\in[0,\tau]}$, $(\hat{\mu}_t)_{t\in[0,\tau]}$ and $(\check{\mu}_t)_{t\in[0,\tau]}$ satisfying $\int_0^\tau \tilde{\mu}_t{\rm d}t+\sup\limits_{t \in[0,\tau]}(\hat{\mu}_t+\check{\mu}_t)<+\infty$ such that
    $$
    \RE (t,y)\in [0,\tau]\times\R^k,\ \
    |g(t,y,0)-g(t,0,0)|\leq \tilde{\mu}_t\varphi(\hat{\mu}_t\bar{\varphi}(\check{\mu}_t|y|));
    $$
\renewcommand{\theenumi}{(H3d)}
\renewcommand{\labelenumi}{\theenumi}
\item\label{A:H3d}
 It holds that
 $$
 \RE (t,y)\in [0,\tau]\times\R^k,\ \  |g(t,y,0)|\leq \tilde{\mu}_t\varphi(\hat{\mu}_t|y|),
 $$
 where the processes $\tilde\mu_\cdot$ and $\hat\mu_\cdot$ along with the function $\varphi(\cdot)$ are all defined in \ref{A:H3c}.
 \end{enumerate}

\begin{rmk}\label{rmk:3.4}
It is obvious that assumption \ref{A:H3c} is weaker than the corresponding assumption used in \citet{Li2025} and \citet{LiFan2024SD} (see (H3a) therein), where $\hat\mu=\check\mu\equiv 1$ and $\varphi(x)=x$. Now, we verify that \ref{A:H3c} implies \ref{A:H3}. In fact, let \ref{A:H3c} hold for the generator $g$. By the assumption of function $\bar\varphi(\cdot)$ we know that for each $x\in\R_+$ and $\lambda\in (0,1)$,
$\bar\varphi(\lambda x)\leq \lambda\bar\varphi(x)$. Then, by taking the following $(\F_t)$-progressively measurable process
$$\alpha_t:=\frac{e^{-\int_0^t \tilde{\mu}_s{\rm d}s}}{(1+\sup\limits_{t\in[0,\tau]}\hat{\mu}_t)
(1+\sup\limits_{t\in[0,\tau]}\check{\mu}_t)}, \ \ t\in[0,\tau]$$
satisfying $\essinf\limits_{t\in[0,\tau]}\alpha_t>0$, we get that for each $t\in [0,\tau]$ and $r\in \R_+$,
\begin{align*}
\psi(t,r\alpha_t)&=\sup_{|y|\leq r\alpha_t} \left|g(t,y,0)-g(t,0,0)\right|\leq \tilde{\mu}_t\varphi(\hat{\mu}_t \bar\varphi(\check{\mu}_t\alpha_t r))\\
&\leq \tilde{\mu}_t\varphi(\hat{\mu}_t \check{\mu}_t\alpha_t \bar\varphi(r))
\leq \tilde{\mu}_t\varphi(\bar\varphi(r)),
\end{align*}
and then for each $r\in \R_+$, we have
\begin{align*}
\begin{split}
\E\left[\int_0^\tau \alpha_t\psi(t,r\alpha_t){\rm d}t\right]\leq\varphi(\bar\varphi(r))\E\left[\int_0^\tau e^{-\int_0^t \tilde{\mu}_s{\rm d}s}\tilde{\mu}_t{\rm d}t\right]\leq \varphi(\bar\varphi(r))<+\infty,
\end{split}
\end{align*}
which means that \ref{A:H3} comes true. Furthermore, we verify that assumption \ref{A:H3d} implies \ref{A:H3c} and then \ref{A:H3}. In fact, if \ref{A:H3d} holds, then we have for each $(t,y)\in [0,\tau]\times\R^k$,
$$
|g(t,y,0)-g(t,0,0)|\leq |g(t,y,0)|+|g(t,0,0)|\leq \tilde{\mu}_t\left(\varphi(\hat{\mu}_t|y|)+\varphi(0)\right)\leq 2\tilde{\mu}_t\varphi(\hat{\mu}_t|y|),
$$
which implies that \ref{A:H3c} holds with $\check{\mu}_\cdot\equiv 1$, $\bar{\varphi}(x)=x$ and $2\tilde\mu_\cdot$ instead of $\tilde\mu_\cdot$.
\end{rmk}

\begin{ex}\label{ex:3.7}
Let $k=1$ and $\tau$ be a finite stopping time, i.e., $\mathbb{P}(\tau<+\infty)=1$, and let
$$
g(t,y,z):=e^{|B_t|e^{-|B_t|^2y}}, \  \ (t,y,z)\in [0,\tau]\times\R\times\R^{1\times d}.
$$
By virtue of the inequality $|e^x-1|\leq e^{|x|}-1$ for $x\in\R$, we deduce that for each $(t,y,z)\in [0,\tau]\times\R\times\R^{1\times d}$,
\begin{align*}
\begin{split}
|g(t,y,0)-g(t,0,0)|&=e^{|B_t|}\bigg|e^{|B_t|(e^{-|B_t|^2y}-1)}-1\bigg|
\leq e^{|B_t|}(e^{|B_t||e^{-|B_t|^2y}-1|}-1)\\
&\leq e^{|B_t|}(e^{|B_t|(e^{|B_t|^2|y|}-1)}-1).
\end{split}
\end{align*}
Then, we can conclude that this generator $g$ satisfies assumption \ref{A:H3c} with
$$
\tilde{\mu}_t:=e^{|B_t|},\ \ \hat{\mu}_t:=|B_t|,\ \ \check{\mu}_\cdot:=|B_t|^2\ \ {\rm and}\ \ \varphi(x)=\bar\varphi(x)=e^x-1.
$$
Consequently, it follows from \cref{rmk:3.4} that \ref{A:H3} holds for this generator $g$.\vspace{0.1cm}
\end{ex}

Finally, we present two examples where \cref{thm:3.1} is applicable, but \cref{pro3.1} and any existing results are not applicable.

\begin{ex}\label{ex3.3}
Let $k=1$, $\tau$ be a stopping time taking values in $[0,+\infty]$ and $\sigma>0$ be a constant. For each $(t,y,z)\in [0,\tau]\times\R\times\R^{1\times d}$, let
    $$ g(t,y,z):=\left(e^{-2t}-3|B_t|^4{\bf 1}_{0\leq t\leq \sigma}\right)y+e^{\int_0^t |B_r|^4{\rm d}r}|B_t|^4 {\bf 1}_{0\leq t\leq \sigma}.$$
It is easy to verify that this generator $g$ satisfies assumptions \ref{A:H1}-\ref{A:H5} with
$$
\mu_t:=e^{-2t}-3|B_t|^4{\bf 1}_{0\leq t\leq \sigma}, \  \ \nu_t\equiv 0,\ \ \rho_t:=\mu_t \ \ \text{and} \ \ \alpha_t:=e^{-\int_0^{t\wedge \sigma} |B_r|^4 {\rm d}r}.
$$
It follows from \cref{thm:3.1} that for each $\xi\in L_\tau^p(\rho_\cdot;\R)$, for example $\xi:=e^{3\int_0^\sigma |B_t|^4 {\rm d}t}$, BSDE \eqref{BSDE1.1} admits a unique weighted $L^p$ solution in the space of $H_\tau^p(\rho_\cdot;\R\times\R^{1\times d})$. However, this assertion cannot be obtained by \cref{pro3.1} nor by any existing results. Indeed, by virtue of $\E[e^{\varepsilon\int_0^\sigma |B_t|^4{\rm d}t}]=+\infty$ for each $\varepsilon>0,$ this generator $g$ does not satisfy \ref{A:H1a} and $\xi=e^{3\int_0^\sigma |B_t|^4 {\rm d}t}$ does not belong to $L_\tau^p(0;\R)$.
\end{ex}

\begin{ex}\label{ex3.4}
Let $k=2$, $\tau$ be a stopping time taking values in $[0,+\infty]$ and $\sigma$ be a finite stopping time. For each $(t,y,z)\in [0,\tau]\times\R^2\times\R^{2\times d}$ with $y=(y_1,y_2)$, let
$$g(t,y,z):=(B_t{\bf 1}_{0\leq t\leq \sigma}+e^{-t})\begin{bmatrix}
y_1-y_2\\
y_2+y_1\\
\end{bmatrix}+B_t{\bf 1}_{0\leq t\leq \sigma}\begin{bmatrix}
\sin|z|\\
|z|\\
\end{bmatrix},
$$
In light of \cref{rmk:3.0*}, this generator $g$ satisfies assumptions \ref{A:H1}, \ref{A:H2}, \ref{A:H3d}, \ref{A:H4} and \ref{A:H5} with
$$
\mu_t:=B_t{\bf 1}_{0\leq t\leq \sigma}+e^{-t}, \ \ \nu_t:=\sqrt{2}|B_t|{\bf 1}_{0\leq t\leq \sigma}, \ \ \rho_t:=\mu_t+\frac{\theta}{2[1\wedge(p-1)]}\nu_t^2, \ \ \tilde{\mu}_t:=\sqrt{2}|\mu_t|,\ \ \hat{\mu}_t\equiv 1
$$
and $\varphi(x):=x$. It then follows from \cref{thm:3.1} along with \cref{rmk:3.4} that for each $\xi\in L_\tau^p(\rho_\cdot;\R^2)$, BSDE \eqref{BSDE1.1} admits a unique weighted $L^p$ solution in the space of $H_\tau^p(\rho_\cdot;\R^2\times\R^{2\times d})$. It is obvious that this assertion cannot be obtained by \cref{pro3.1} nor by any existing results.
\end{ex}

\subsection{Illustrating Corollaries and Examples as well as Remarks\vspace{0.2cm}}

In this subsection, we present several corollaries of \cref{thm:3.1} and compare them with existing relevant conclusions through some remarks and examples.

First, the following corollary establishes existence and uniqueness of both the usual $L^p$ solution and the bounded solution of BSDE \eqref{BSDE1.1} with possibly unbounded stochastic coefficients.

\begin{cor}\label{cor:3.4}
Let $\xi\in L_\tau^p(0;\R^k)$ and the generator $g$ satisfy assumptions \ref{A:H1b} and \ref{A:H2}-\ref{A:H5}. Assume further that for some constant $M>0$,
\begin{align}\label{eq:3.8}
\int_0^\tau \left(\mu_t+\frac{\theta}{2[1\wedge(p-1)]}\nu_t^2\right)^+ {\rm d}t\leq M.
\end{align}
Then, BSDE \eqref{BSDE1.1} admits a unique $L^p$ solution $(y_t,z_t)_{t\in[0,\tau]}$ in the space of $H_\tau^p(0;\R^{k}\times\R^{k\times d})$. Moreover, if $|\xi|+\int_0^\tau |g(s,0,0)|{\rm d}s\in L^\infty(\Omega,\F_\tau, \mathbb{P};\R)$, then $(y_t)_{t\in[0,\tau]}\in S^\infty(\R^k)$.\vspace{0.1cm}
\end{cor}

\begin{proof}
Let $\rho_\cdot:=\mu_\cdot+\frac{\theta}{2[1\wedge(p-1)]}\nu_\cdot^2$. We have
$\int_0^\tau \rho^+_t {\rm d}t\leq M$. Then, it is easy to see that $\xi\in L_\tau^p(\rho_\cdot^+;\R^k)$ and the generator $g$ satisfies \ref{A:H1} with $\rho^+_\cdot$ instead of $\rho_\cdot$, and by \cref{thm:3.1} that BSDE \eqref{BSDE1.1} admits a unique $L^p$ solution $(y_t,z_t)_{t\in[0,\tau]}$ in $H_\tau^p(\rho^+_\cdot;\R^{k}\times\R^{k\times d})=H_\tau^p(0;\R^{k}\times\R^{k\times d})$. Moreover, suppose that $|\xi|+\int_0^\tau |g(s,0,0)|{\rm d}s\leq \bar M$ for some constant $\bar M>0$. It follows from \ref{A:H1b}, \ref{A:H4} and \ref{A:H5} that
\begin{align*}
\left<\hat{y},~g(t,y,z)\right>&=\left<\hat{y},~g(t,y,z)-g(t,0,z)+g(t,0,z)-g(t,0,0)+g(t,0,0)\right>\\
&\leq \mu_t|y|+\nu_t|z|+|g(t,0,0)|, \ \ (t,y,z)\in [0,\tau]\times\R^k\times\R^{k\times d},
\end{align*}
which means that $g$ satisfies \ref{A:A} with $\bar{p}:=p$, $\bar\mu_t:={\mu}_t$, $\bar\nu_t:=\nu_t$, ${\bar\rho_t}:=\rho_t^+$ and $f_t:=|g(t,0,0)|$. Then, by \cref{pro2.02} with $r=t$ we can conclude that there exists a constant $C_{p,\theta}>0$ depending only on $p$ and $\theta$ such that for each $t\geq 0$,
\begin{align*}
\begin{split}
|y_{t\wedge\tau}|^p&\leq e^{p\int_{0}^{t\wedge\tau}\rho^+_r{\rm d}r}|y_{t\wedge\tau}|^p\leq\E\left[\sup_{s\in[t\wedge\tau,\tau]}\left(e^{p\int_{0}^{s}\rho^+_r{\rm d}r}|y_s|^p\right)\bigg|\F_{t\wedge\tau}\right]\\
&\leq
C_{p,\theta}\E\left[e^{p\int_{0}^{\tau}\rho^+_r{\rm d}r}|\xi|^p+\left(
\int_{t\wedge\tau}^{\tau}e^{ \int_{0}^{s}\rho^+_r{\rm d}r}|g(s,0,0)|{\rm d}s\right)^p\bigg|\F_{t\wedge\tau}\right]\\
&\leq C_{p,\theta}e^{pM}\bar M^p.
\end{split}
\end{align*}
Therefore, $y_\cdot\in S^\infty(\R^k)$. The proof is then complete.
\end{proof}

\begin{rmk}\label{rmk:3.2}
Note that the condition of \eqref{eq:3.8} is strictly weaker than the following condition
\begin{align}\label{eq:3.9}
\int_0^\tau (\mu_t^+ +\nu_t^2){\rm d}t\leq M.
\end{align}
\cref{cor:3.4} strengthens \citet[Corollary 3.6]{LiFan2024SD} and \citet[Corollary 3.4]{Li2025}, and then some existing results on the $L^p$ solution and the bounded solution of BSDEs.
\end{rmk}

We especially emphasis that two unbounded stochastic processes $\mu_\cdot\in\R$ and $\nu_\cdot\in\R_+$ that do not satisfy \eqref{eq:3.9} may satisfy \eqref{eq:3.8}. See the following example for more details.

\begin{ex}\label{ex3.1}
We provide three examples to which \cref{cor:3.4} can be applied, but any existing findings can not be applied. We mention that in these three examples, $k=1$ is always supposed and the processes $\mu_\cdot\in\R$ and $\nu_\cdot\in\R_+$ satisfy \eqref{eq:3.8}, but they do not satisfy \eqref{eq:3.9}.

\begin{itemize}
\item [(i)] Let $\tau$ be a bounded stopping time, i.e., $\tau\leq T$ for some constant $T>0$, and let
\begin{align*}
g(t,y,z):=-|B_t|y+\sqrt{|B_t|+1}\left(|z|\wedge |z|^2\right), \ \ (t,y,z)\in [0,\tau]\times\R\times\R^{1\times d}.
\end{align*}
Obviously, this generator $g$ satisfies assumptions \ref{A:H1b} and \ref{A:H2}-\ref{A:H5} with
\begin{align*}
p:=2,\ \ \mu_t:=-|B_t|, \ \ \nu_t:=\sqrt{|B_t|+1} \ \ \text{and} \ \ \alpha_t:=e^{-\int_0^t |B_r|\dif r}.
\end{align*}
And, \eqref{eq:3.8} is fulfilled for $\theta=p=2$. Then, \cref{cor:3.4} with $p=2$ is applicable in this case.

\item [(ii)] Let $\tau$ be a finite stopping time, i.e., $\mathbb{P}(\tau<+\infty)=1$, and let
\begin{align*}
g(t,y,z):=e^{-|B_t|y}-y+\sin|z|-1, \ \ (t,y,z)\in [0,\tau]\times\R\times\R^{1\times d}.
\end{align*}
In light of \cref{rmk:3.0*}, it is not very difficult to verify that this generator $g$ satisfies assumptions \ref{A:H1b}, \ref{A:H2}, \ref{A:H3d}, \ref{A:H4} and \ref{A:H5} with
\begin{align*}
p>3/2,\ \ \mu_t\equiv -1, \ \ \nu_t\equiv 1, \ \ \tilde{\mu}_t\equiv 1,\ \  \hat{\mu}_t:=|B_t|+1\ \ \text{and} \ \ \varphi(x):=e^{x}+x.
\end{align*}
In addition, it is obvious that \eqref{eq:3.8} is fulfilled for $p>3/2$ and $\theta=2[1\wedge(p-1)]>1$. Thus, \cref{cor:3.4} with $p>3/2$ is applicable in this case.

\item [(iii)] Let $p,\theta>1$ be two arbitrarily given constants and for each $(t, y,z)\in [0,\tau]\times \R\times\R^{1\times d}$, let
$$g(t,y,z):=-\frac{\theta}{2[1\wedge(p-1)]}\left(|B_t|^2{\bf 1}_{0\leq t\leq 1}+2e^{-t}\right)^2(y+y^3)+\left(|B_t|^2{\bf 1}_{0\leq t\leq 1}+e^{-t}\right)|z|.$$
In light of \cref{rmk:3.0*}, it is not very difficult to verify that this generator $g$ satisfies assumptions \ref{A:H1b}, \ref{A:H2}, \ref{A:H3d}, \ref{A:H4} and \ref{A:H5} with
\begin{align*}
\mu_t:=-\frac{\theta}{2[1\wedge(p-1)]}\left(|B_t|^2{\bf 1}_{0\leq t\leq 1}+2e^{-t}\right)^2, \ \ \nu_t:=|B_t|^2{\bf 1}_{0\leq t\leq 1}+e^{-t}, \ \ \tilde{\mu}_t\equiv |\mu_t|,\ \  \hat{\mu}_t:=1
\end{align*}
and $\varphi(x):=x+x^3$. Obviously, \eqref{eq:3.8} is fulfilled. Then, \cref{cor:3.4} is applicable in this case.\vspace{0.1cm}
\end{itemize}
\end{ex}

Second, the following corollary demonstrates that when the condition of \eqref{eq:3.8} in \cref{cor:3.4} is further relaxed to the following \eqref{eq:3.8*}, we can obtain a similar result.

\begin{cor}\label{cor:3.5}
Let $\xi\in L_\tau^p(0;\R^k)$ and the generator $g$ satisfy assumptions \ref{A:H1b} and \ref{A:H2}-\ref{A:H5}. Assume further that for any $\bar p>1$,
\begin{align}\label{eq:3.8*}
\E\left[\exp\left(\bar p \int_0^\tau\left(\mu_t+\frac{\theta}{2[1\wedge(p-1)]}\nu_t^2\right)^+ {\rm d}t\right)\right]<+\infty.
\end{align}
Then, BSDE \eqref{BSDE1.1} admits a unique solution $(y_t,z_t)_{t\in[0,\tau]}\in \cap_{q\in(1,p)}H_\tau^q(0;\R^{k}\times\R^{k\times d})$.\vspace{0.1cm}
\end{cor}

\begin{proof}
Let $\rho_\cdot:=\mu_\cdot+\frac{\theta}{2[1\wedge(p-1)]}\nu_\cdot^2$. We have $\E\left[e^{\bar p\int_0^\tau \rho^+_t \dif t}\right]<+\infty$ for any $\bar p>1$. Then, it follows from Young's inequality that for each $ q\in(1,p)$,
\begin{align}\label{3.20}
\E\left[e^{q\int_0^\tau \rho_t^+ \dif t}|\xi|^q\right]
\leq \frac{q}{p}\E\left[|\xi|^p\right]+\frac{p-q}{p}\E\left[e^{\frac{pq}{p-q}\int_0^\tau \rho^+_t \dif t}\right]<+\infty
\end{align}
and
\begin{align}\label{3.21}
\begin{split}
\E\left[\left(\int_{0}^{\tau}e^{\int_0^t \rho_r^+ \dif r}|g(t,0,0)|\dif t\right)^q\right]&\leq \E\left[e^{q\int_0^\tau \rho_r^+ \dif r}\left(\int_{0}^{\tau}|g(t,0,0)|\dif t\right)^q\right]\\
&\leq
\frac{q}{p}\E\left[\left(\int_{0}^{\tau}|g(t,0,0)|\dif t\right)^p\right]+\frac{p-q}{p}\E\left[e^{\frac{pq}{p-q}\int_0^\tau \rho^+_t \dif t}\right]<+\infty.
\end{split}
\end{align}
In light of \eqref{3.20} and \eqref{3.21} along with assumptions \ref{A:H2}-\ref{A:H5}, by \cref{thm:3.1} we can conclude that BSDE \eqref{BSDE1.1} admits a unique adapted solution $(y_t,z_t)_{t\in[0,\tau]}$ which belongs to the space of  $H_\tau^q(\rho_\cdot^+;\R^{k}\times\R^{k\times d})$ for any $q\in (1,p)$, and then $(y_\cdot,z_\cdot)\in \cap_{q\in(1,p)}H_\tau^q(0;\R^{k}\times\R^{k\times d})$.

Now, we prove the uniqueness part. Assume that both $(y_\cdot,z_\cdot)$ and $(y_\cdot',z_\cdot')$ are two adapted solutions of BSDE \eqref{BSDE1.1} in $H_\tau^q(0;\R^{k}\times\R^{k\times d})$ for some $q\in (1,p)$. An identical argument to that in \eqref{3.20} and \eqref{3.21} yields that both $(y_\cdot,z_\cdot)$ and $(y_\cdot',z_\cdot')$ also belong to $H_\tau^{\bar q}(\rho_\cdot^+;\R^{k}\times\R^{k\times d})$ for any $\bar q\in (1,q)$. However, in light of \eqref{3.20} and \eqref{3.21} along with assumptions \ref{A:H2}-\ref{A:H5}, by \cref{thm:3.1} we know that for each $\bar q\in (1,q)$, BSDE \eqref{BSDE1.1} admits a unique solution in $H_\tau^{\bar q}(\rho_\cdot^+;\R^{k}\times\R^{k\times d})$. Hence, $(y_\cdot,z_\cdot)=(y_\cdot',z_\cdot')$.
\end{proof}

\begin{rmk}\label{rmk:3.4*}
For the case of $\tau=T$ for some constant $T>0$, if there exists a constant $\alpha\in(0,1)$ and a bounded mean oscillation martingale $(K_t)_{t\in [0,T]}$, i.e.,
$$\sup_{t\geq 0}\left\|\E\left[\int_t^T K_s^2{\rm d}s\bigg|\F_t\right]\right\|_{\infty}<+\infty,$$
such that for each $t\in \T$,
\begin{align}\label{eq:3.11}
\left(\mu_t+\frac{\theta}{2[1\wedge(p-1)]}\nu_t^2\right)^+\leq K_t^{2\alpha},
\end{align}
then by virtue of the John-Nirenberg inequality and the elementary inequality that for each $\eps>0$,
$$
x^{2\alpha}\leq \eps x^2+(1-\alpha)\left(\frac{\alpha}{\eps}\right)^{\frac{\alpha}{1-\alpha}},\ \ x\geq 0,
$$
we can derive that the condition \eqref{eq:3.8*} is fulfilled. Furthermore, readers are referred to \citet[Theorem 10]{BriandandConfortola2008} for a result related to \cref{cor:3.5} under a similar condition to \eqref{eq:3.11}, where the generator $g$ is forced to have a sub-quadratic growth in $y$ (see assumption (A4) therein). In addition, from the proof of \cref{cor:3.5}, it is not hard to see that if $\xi\in L_\tau^p(0;\R^k)$, the generator $g$ satisfy \ref{A:H1b} and \ref{A:H2}-\ref{A:H5}, and \eqref{eq:3.8*} holds for some $\bar p>1$, then BSDE \eqref{BSDE1.1} admits a unique solution $(y_t,z_t)_{t\in[0,\tau]}$ in the space of $H_\tau^q(0;\R^{k}\times\R^{k\times d})$ with $q:=\frac{p\bar p}{p+\bar p}\in (1, p\wedge \bar p)$ such that $\frac{pq}{p-q}=\bar p$.
\end{rmk}

Third, we present an existence and uniqueness result for the weighted $L^2$ solution of BSDE \eqref{BSDE1.1} with a finite random terminal time and constant coefficients.

\begin{cor}\label{cor:3.3}
Let $\tau$ be a finite stopping time, i.e., $P(\tau<+\infty)=1$, and let $\rho,\mu\in\R$ and $\nu\in\R_+$ be three given constants satisfying $\rho> \mu+\frac{1}{2}\nu^2$. Assume that the generator $g$ satisfies
\begin{enumerate}
\renewcommand{\theenumi}{(H1c)}
\renewcommand{\labelenumi}{\theenumi}
\item\label{A:H1c} $\Dis \E\left[\int_0^\tau e^{2\rho s} |g(s,0,0)|^2{\rm d}s\right]<+\infty$
\end{enumerate}
and assumptions \ref{A:H2}-\ref{A:H5} with $\mu_\cdot\equiv \mu$ and $\nu_\cdot\equiv \nu$. Then, for each $\xi\in L_\tau^2(\rho;\R^k)$, BSDE \eqref{BSDE1.1} admits a unique weighted $L^2$ solution $(y_t,z_t)_{t\in[0,\tau]}$ in the space of $H_\tau^2(\rho;\R^{k}\times\R^{k\times d})$. Moreover, we have
\begin{align*}
\E\left[\int_0^{\tau}e^{2 \rho s}|y_s|^2{\rm d}s\right]<+\infty.
\end{align*}
\end{cor}

\begin{proof}
It is obvious that $\int_{0}^{\tau}\left(|\rho|+|\mu|+\nu^2\right){\rm d}t<+\infty$ since $P(\tau<+\infty)=1$. It follows from $\rho> \mu+\frac{1}{2}\nu^2$ that there must exist an $\varepsilon>0$ such that $\rho= \mu+\frac{1+\varepsilon}{2}\nu^2+\varepsilon$. Then, by taking $\theta=\varepsilon+1>1$, we have
\begin{align}\label{3.25}
\rho= \mu+\frac{\theta}{2}\nu^2+\theta-1>\mu+\frac{\theta}{2}\nu^2.
\end{align}
We first prove that if the generator $g$ satisfies \ref{A:A} with $\bar p=2$, $\bar\rho_\cdot\equiv \rho$, $\bar\mu_\cdot\equiv \mu$, $\bar\nu_\cdot\equiv \nu$ and $f_\cdot$ satisfying
\begin{align}\label{3.30}
\E\left[\int_{0}^{\tau}e^{2\rho t}f_t^2{\rm d}t\right]<+\infty,
\end{align}
and $(y_t,z_t)_{t\in [0,\tau]}$ is an adapted solution of BSDE \eqref{BSDE1.1} such that $y_\cdot\in S_\tau^2(\rho;\R^k)$, then $z_\cdot\in M_\tau^2(\rho;\R^{k\times d})$, and there exists a constant $C_{\theta}>0$ depending only on $\theta$ such that for each $0\leq r\leq t< +\infty$,
\begin{align}\label{3.51}
\begin{split}
&\E\left[\sup_{s\in[t\wedge\tau,\tau]}\left(e^{2\rho s}|y_s|^2\right)\bigg|\F_{r\wedge\tau}\right]+ \E\left[\int_{t\wedge\tau}^{\tau}e^{2\rho s}|z_s|^2{\rm d}s\bigg|\F_{r\wedge\tau}\right]
+(\theta-1)\E\left[\int_{t\wedge\tau}^{\tau}e^{2 \rho s}|y_s|^2{\rm d}s\bigg|\F_{r\wedge\tau}\right]\\
&\ \ \leq
C_{\theta}\E\left[e^{2\rho\tau}|\xi|^2+
\int_{t\wedge\tau}^{\tau}e^{2\rho s}f_s^2{\rm d}s\bigg|\F_{r\wedge\tau}\right].
\end{split}
\end{align}
In fact, in light of \eqref{3.25} and the inequality $2ab\leq (\theta-1)a^2 +\frac{1}{\theta-1}b^2$ for any $a,b\in \R$, \eqref{2.3} and \eqref{2.11} appearing in the proof of \cref{pro2.01,pro2.02} can be respectively transformed into
\begin{align*}
\begin{split}
&|\bar{y}_{t\wedge\tau_n}|^2+\left(1-\frac{1}{\theta}\right) \int_{t\wedge\tau_n}^{\tau_n} |\bar{z}_s|^2{\rm d}s+ (\theta-1)\int_{t\wedge\tau_n}^{\tau_n}|\bar{y}_s|^2{\rm d}s\\
&\ \ \leq|\bar{y}_{\tau_n}|^2+\frac{1}{\theta-1}\int_{t\wedge\tau_n}^{\tau_n} \bar{f}_s^2{\rm d}s-2\int_{t\wedge\tau_n}^{\tau_n}\langle \bar{y}_s,\bar{z}_s{\rm d}B_s\rangle,\ \ t\geq 0
\end{split}
\end{align*}
and the last inequality with $\tau$ instead of $\tau_n$. Thus, in light of \eqref{3.30}, by an identical argument to that in the proof of \cref{pro2.01,pro2.02} we can conclude that there exist two constants $C_{\theta}^1>0$ and $C_{\theta}^2>0$ depending only on $\theta$ such that for each $0\leq r\leq t< +\infty$,
\begin{align*}
\E\left[\int_{t\wedge\tau}^{\tau}e^{2\rho s}|z_s|^2{\rm d}s\bigg|\F_{r\wedge\tau}\right]
\leq
C_{\theta}^1\E\left[\sup_{s\in[t\wedge\tau,\tau]}\left(e^{2 \rho s}|y_s|^2\right) +\int_{t\wedge\tau}^{\tau}e^{2\rho s}f_s^2{\rm d}s\bigg|\F_{r\wedge\tau}\right]
\end{align*}
and
\begin{align*}
\begin{split}
&\E\left[\sup_{s\in[t\wedge\tau,\tau]}\left(e^{2\rho s}|y_s|^2\right)\bigg|\F_{r\wedge\tau}\right]
+(\theta-1)\E\left[\int_{t\wedge\tau}^{\tau}e^{2 \rho s}|y_s|^2{\rm d}s\bigg|\F_{r\wedge\tau}\right]\\
&\ \ \leq C_\theta^2\E\left[e^{2\rho \tau}|\xi|^2+
\int_{t\wedge\tau}^{\tau}e^{2\rho s}f_s^2{\rm d}s\bigg|\F_{r\wedge\tau}\right].
\end{split}
\end{align*}
Then, the desired assertion \eqref{3.51} follows immediately from the last two inequalities.

Finally, with the a priori estimate \eqref{3.51} along with \ref{A:H1c} in the hand, by following closely the proof procedure of \cref{thm:3.1} with $p=2$ we can obtain the desired conclusions of \cref{cor:3.3}.\vspace{0.1cm}
\end{proof}

In light of \cref{rmk:3.4}, it is not difficult to verify that \cref{cor:3.3} strengthens
\citet[Theorem 3.4]{DarlingandPardoux1997} and
\citet[Theorem 4.1]{Pardoux1999} for the case of $\tau$ being a finite stopping time, where the generator $g$ is forced to respectively satisfy the following assumptions \ref{A:H3e} and \ref{A:H3f}, either of which is stronger than assumption \ref{A:H3} when assumption \ref{A:H1c} is in force:
\begin{enumerate}
\renewcommand{\theenumi}{(H3e)}
\renewcommand{\labelenumi}{\theenumi}
\item\label{A:H3e} There exist two constants $\kappa>0$ and $\kappa'=0$ or $1$ such that
    $$\RE (t,y)\in [0,\tau]\times \R^k,\ \ \Dis |g(t,y,0)|\leq |g(t,0,0)|+\kappa(|y|+\kappa').$$
    Moreover, we assume that $\E\left[e^{2\rho \tau}\kappa'\right]<+\infty$.
\end{enumerate}
\begin{enumerate}
\renewcommand{\theenumi}{(H3f)}
\renewcommand{\labelenumi}{\theenumi}
\item\label{A:H3f} There exists a nondecreasing continuous function $\phi(\cdot):\R_+\rightarrow \R_+$ such that
    $$\RE (t,y)\in [0,\tau]\times \R^k,\ \ \Dis |g(t,y,0)|\leq |g(t,0,0)|+\phi(|y|).$$
    Moreover, we assume that
\begin{align}\label{3.23}
\E\left[\int_0^\tau e^{2\rho t} |g(t,e^{-\rho t}\xi_t,e^{-\rho t}\eta_t)|^2{\rm d}t\right]<+\infty,\vspace{0.1cm}
\end{align}
where $\xi_t=\E\left[\xi\big|\F_{t}\right]$ and $\xi=\E\left[\xi\right]+\int_0^\infty \eta_t {\rm d}B_t$ with $\eta_\cdot\in M_\tau^2(0;\R^{k\times d})$.
\end{enumerate}
In particular, the proof strategy of \cref{cor:3.3} originates from the proof of \cref{thm:3.1} distinct with that used in \cite{DarlingandPardoux1997,Pardoux1999}, which utilizes approximations of BSDEs with constant terminal times.\vspace{0.2cm}

In the sequel, we provide two concrete examples to which \cref{cor:3.3} can be applied, while both \cite[Theorem 3.4]{DarlingandPardoux1997} and \cite[Theorem 4.1]{Pardoux1999} cannot be applied.

\begin{ex}\label{ex3.2}
(i) Let $k=1$, $P(\tau<+\infty)=1$ and for each $(t,y,z)\in [0,\tau]\times\R\times\R^{1\times d}$, let
$$
g(t,y,z):=(-y+\frac{1}{4}){\bf 1}_{y\leq -\frac{1}{4}}+\sqrt{-y}{\bf 1}_{-\frac{1}{4}\leq y\leq 0}+\ln(1+|z|).
$$
In light of \cref{rmk:3.0*}, it is easy to verify that this $g$ satisfies all assumptions in \cref{cor:3.3} with
$$
\mu:=-1, \ \ \nu:=1 \ \ \text{and} \ \ \rho:=1>\mu+\frac{1}{2}\nu^2=-\frac{1}{2}.
$$
It then follows from \cref{cor:3.3} that for each $\xi\in L_\tau^2(\rho;\R)$, for example $\xi:=e^{-\tau}B_\tau$, BSDE \eqref{BSDE1.1} admits a unique weighted $L^2$ solution in the space of $H_\tau^2(\rho;\R\times\R^{1\times d})$. However, it can also be easily verified that this $g$ satisfies the first condition in \ref{A:H3e} with $\kappa=\kappa'=1$ and the first condition in \ref{A:H3f} with $\phi(x)=x+1$, while it may satisfy neither $\E\left[e^{2\rho \tau}\kappa'\right]=\E\left[e^{2\tau}\right]<+\infty$ nor \eqref{3.23}. Consequently, the previous assertion can not be obtained by
\cite[Theorem 3.4]{DarlingandPardoux1997} and \cite[Theorem 4.1]{Pardoux1999}.\vspace{0.2cm}

(ii) Let $k=2$, $P(\tau<+\infty)=1$ and for each $(t,y,z)\in [0,\tau]\times\R^2\times\R^{2\times d}$ with $y=(y_1,y_2)$, let
$$
g(t,y,z):=-3
\begin{bmatrix}
y_1(1+|y_2|)\\
y_2(1+e^{y_1})\\
\end{bmatrix}
+
\begin{bmatrix}
\sqrt{1+2|z|}\\
|z|\\
\end{bmatrix}.
$$
In light of \cref{rmk:3.0*}, it is easy to verify that this $g$ satisfies all assumptions in \cref{cor:3.3} with
$$
\mu:=-3, \ \ \nu:=\sqrt{2} \ \ \text{and} \ \ \rho:=-1>\mu+\frac{1}{2}\nu^2=-2.
$$
It then follows from \cref{cor:3.3} that for each $\xi\in L_\tau^2(\rho;\R^2)$, for example $\xi:=e^{\tau}$, BSDE \eqref{BSDE1.1} admits a unique weighted $L^2$ solution in $H_\tau^2(\rho;\R^2\times\R^{2\times d})$. However, this assertion can not be obtained by
\cite[Theorem 3.4]{DarlingandPardoux1997} and \cite[Theorem 4.1]{Pardoux1999} since this $g$ may satisfy neither \ref{A:H3e} nor \eqref{3.23}.\vspace{0.1cm}
\end{ex}

Finally, the following corollary includes \cref{cor:3.3} as a particular case when $|g(\cdot,0,0)|\equiv 0$.

\begin{cor}\label{cor:3.3*}
Let $P(\tau<+\infty)=1$ and let $\rho,\mu\in\R$ and $\nu\in\R_+$ be three given constants satisfying
\begin{align*}
\rho\geq \mu+\frac{\theta}{2[1\wedge(p-1)]}\nu^2.\vspace{0.1cm}
\end{align*}
Assume that the generator $g$ satisfies assumptions \ref{A:H1}-\ref{A:H5} with $\mu_\cdot\equiv \mu$, $\nu_\cdot\equiv \nu$ and $\rho_\cdot\equiv \rho$. Then, for each $\xi\in L_\tau^p(\rho;\R^k)$, BSDE \eqref{BSDE1.1} admits a unique weighted $L^p$ solution $(y_t,z_t)_{t\in[0,\tau]}\in H_\tau^p(\rho;\R^{k}\times\R^{k\times d})$. Moreover, if $|\xi|+|g(\cdot,0,0)|\leq M$ and $\rho\leq -c$ for two constants $M,c>0$, then $(y_t)_{t\in[0,\tau]}\in S^\infty(\R^k)$.\vspace{0.1cm}
\end{cor}

\begin{proof}
By virtue of $P(\tau<+\infty)=1$, we have $\int_{0}^{\tau}\left(|\rho|+|\mu|+\nu^2\right){\rm d}t<+\infty$, and then by \cref{thm:3.1} we can deduce that BSDE \eqref{BSDE1.1} admits a unique weighted $L^p$ solution $(y_t,z_t)_{t\in[0,\tau]}\in H_\tau^p(\rho;\R^{k}\times\R^{k\times d})$. Moreover, in light of \ref{A:H1}, \ref{A:H4} and \ref{A:H5}, we have
\begin{align*}
\left<\hat{y},~g(t,y,z)\right>&=\left<\hat{y},~g(t,y,z)-g(t,0,z)+g(t,0,z)-g(t,0,0)+g(t,0,0)\right>\\
&\leq \mu|y|+\nu|z|+|g(t,0,0)|, \ \ (t,y,z)\in [0,\tau]\times\R^k\times\R^{k\times d},
\end{align*}
which means that $g$ satisfies \ref{A:A} with $\bar{p}:=p$, $\bar\mu_t\equiv \mu$, $\bar\nu_t\equiv \nu$, ${\bar\rho_t}\equiv \rho$ and $f_t:=|g(t,0,0)|$. It follows from \cref{pro2.02} with $r=t$ that there exists a constant $K_{p,\theta}>0$ depending only on $p$ and $\theta$ such that
\begin{align}\label{3.211}
e^{p\rho (t\wedge\tau)}|y_{t\wedge\tau}|^p\leq
K_{p,\theta}\E\left[e^{p\rho\tau}|\xi|^p+\left(
\int_{t\wedge\tau}^{\tau}e^{\rho s}|g(s,0,0)|{\rm d}s\right)^p\bigg|\F_{t\wedge\tau}\right],\ \ t\geq 0.
\end{align}
In light of $\rho\leq -c$ and $|\xi|+|g(\cdot,0,0)|\leq M$, by \eqref{3.211} we deduce that for each $t\geq 0$,
\begin{align*}
\begin{split}
|y_{t\wedge\tau}|^p
&\leq
K_{p,\theta}\E\left[e^{p\rho(\tau-t\wedge\tau)}|\xi|^p+\left(
\int_{t\wedge\tau}^{\tau}e^{\rho(s-t\wedge\tau)}|g(s,0,0)|{\rm d}s\right)^p\bigg|\F_{t\wedge\tau}\right]\\
&\leq
K_{p,\theta}\E\left[e^{-pc(\tau-t\wedge\tau)}|\xi|^p+\left(
\int_{t\wedge\tau}^{\tau}e^{-c(s-t\wedge\tau)}|g(s,0,0)|{\rm d}s\right)^p\bigg|\F_{t\wedge\tau}\right]\leq K_{p,\theta}M^p\left(1+\frac{1}{c}\right).
\end{split}
\end{align*}
Therefore, we obtain that $(y_t)_{t\in[0,\tau]} \in S^\infty(\R^k)$.
\end{proof}

\begin{rmk}\label{rmk3.5}
It is not difficult to verify that when $|g(\cdot,0,0)|\equiv 0$ is in force, \cref{cor:3.3*} includes \cref{cor:3.3} as a particular case of $p=2$ since $\rho>\mu+\frac{1}{2}\nu^2$ can imply $\rho\geq \mu+\frac{\theta}{2}\nu^2$ for some $\theta>1$. We mention that in the case of $\nu=0$, the constant $\rho$ in  \cref{cor:3.3*} may be bigger than or equal to the constant $\mu$, while $\rho$ in  \cref{cor:3.3} has to be bigger than $\mu$. In addition, readers are referred to \citet[Theorem 5.2]{Briand2003} for a result related to \cref{cor:3.3*}, where $p\in (1,2]$, $\tau$ is a stopping time taking values in $[0,+\infty]$ and the generator $g$ is forced to satisfy a assumption similar to \ref{A:H3f}.
\end{rmk}

\begin{ex}\label{ex3.6}
Let $k=1$, $P(\tau<+\infty)=1$ and for each $(t,y,z)\in [0,\tau]\times\R\times\R^{1\times d}$,
$$
g(t,y,z):=e^{-y}{\bf 1}_{y\leq 0}+(1-y){\bf 1}_{y>0}-1.
$$
In light of \cref{rmk:3.0*}, it is easy to verify that this $g$ satisfies all assumptions in \cref{cor:3.3*} with
$$
\mu:=-1, \ \ \nu:=0 \ \ \text{and} \ \ \rho:=-1.
$$
It then follows from \cref{cor:3.3*} that for each $\xi\in L_\tau^p(\rho;\R)$, BSDE \eqref{BSDE1.1} admits a unique solution in $H_\tau^p(\rho;\R\times\R^{1\times d})$. This conclusion can not be obtained by \cref{cor:3.3} or any existing results.
\end{ex}

\subsection{Continuous dependence and stability\vspace{0.2cm}}

In this subsection, we state a continuous dependence property and a stability theorem for the general weighted $L^p$ solutions of BSDEs under assumptions \ref{A:H4} and \ref{A:H5}. They extend Theorems 3.13 and 3.14 in \citet{Li2025}, where the process $\rho_\cdot:=\int_0^\cdot (\beta \mu_s^+ +\frac{\theta}{2}\nu_s^2) \dif s$. By \cref{pro2.1}, the proofs are straightforward; nevertheless, a complete argument is provided here for readers' convenience.

\begin{thm}[\bf Continuous dependence]\label{thm:3.2}
Let $\xi, \xi'\in L_\tau^p(\rho_\cdot;\R^k)$, both $g$ and $g'$ be the generator of BSDEs satisfying assumptions \ref{A:H4} and \ref{A:H5}, and $(Y_\cdot,Z_\cdot)$ and $(Y_\cdot',Z_\cdot')$ be the weighted $L^p$ solution of BSDE$(\xi,\tau,g)$ and BSDE$(\xi',\tau,g')$, respectively, in the space of $H_\tau^p(\rho_\cdot;\R^{k}\times\R^{k\times d})$.
Then there exists a constant $C_{p,\theta}>0$ depending only on $p$ and $\theta$ such that
\begin{align*}
\begin{split}
&\E\left[\sup_{s\in[0,\tau]}\left(e^{p \int_{0}^{s}\rho_r{\rm d}r}|Y_s-Y_s'|^p\right)\right]+\E\left[\left(\int_0^{\tau} e^{2\int_0^s \rho_r{\rm d}r}|Z_s-Z_s'|^2{\rm d}s\right)^{\frac{p}{2}}\right]\\
&\ \ \leq  C_{p,\theta}\left(\E\left[e^{p \int_{0}^{\tau}\rho_r{\rm d}r}|\xi-\xi'|^p\right]+\E\left[\left(\int_0^\tau e^{\int_0^s\rho_r{\rm d}r}\left|g(t,Y'_t,Z'_t)-g'(t,Y'_t,Z'_t)\right|{\rm d}s\right)^p\right]\right).
\end{split}
\end{align*}
\end{thm}

\begin{proof}
Without loss of generality, we can assume that
$$\E\left[\left(\int_0^\tau e^{\int_0^s\rho_r{\rm d}r}\left|g(t,Y'_t,Z'_t)-g'(t,Y'_t,Z'_t)\right|{\rm d}s\right)^p\right]<+\infty.$$
Let\vspace{-0.2cm}
\begin{align*}
\begin{split}
\tilde{\xi}:=\xi-\xi', \ \ \tilde{Y}_\cdot:=Y_\cdot-Y_\cdot' \ \ {\rm and}\ \ \tilde{Z}_\cdot:=Z_\cdot-Z_\cdot'.
\end{split}
\end{align*}
Then, $(\tilde{Y}_\cdot,\tilde{Z}_\cdot)$ is a weighted $L^p$ solution of the following BSDE:
\begin{align*}
\tilde{Y}_t=\tilde{\xi}+\int_t^\tau \tilde{g}(s,\tilde{Y}_s,\tilde{Z}_s){\rm d}s-\int_t^\tau\tilde{Z}_s{\rm d}B_s, \ \ t\in[0,\tau],
\end{align*}
where the generator $\tilde{g}$ is defined by
$$
\tilde{g}(t,y,z):=g(t,y+Y_t',z+Z_t')-g'(t,Y_t',Z_t'), ~~(t,y,z)\in[0,\tau]\times\R^k\times\R^{k\times{d}}.
$$
It follows from \ref{A:H4} and \ref{A:H5} that for each $(y,z)\in\R^k\times\R^{k\times{d}}$,
\begin{align*}
\begin{split}
\left<\hat{y},\tilde{g}(t,y,z)\right>&=\left<\hat{y},g(t,y+Y_t',z+Z_t')-g'(t,Y_t',Z_t')\right>\\
&\leq \mu_t|y|+\nu_t|z|+|g(t,Y_t',Z_t')-g'(t,Y_t',Z_t')|, \ \ t\in[0,\tau],
\end{split}
\end{align*}
which implies that the generator $\tilde{g}$ satisfies assumption (A) with $\bar{p}:=p,$ $\bar\mu_t:=\mu_t$, $\bar\nu_t:=\nu_t$, $\bar\rho_t:=\rho_t$ and $f_t:=|g(t,Y_t',Z_t')-g'(t,Y_t',Z_t')|$.
Then, the desired conclusion follows from \cref{pro2.1}.
\end{proof}

\begin{thm}[\bf Stability]\label{thm:3.3}
For each $n\geq1$, let $\xi^n, \xi\in L_\tau^p(\rho_\cdot;\R^k)$, both $g^n$ and $g$ be the generator of BSDEs   satisfying assumptions \ref{A:H4} and \ref{A:H5}, and $(Y_\cdot^n, Z_\cdot^n)$ and $(Y_\cdot, Z_\cdot)$ be the weighted $L^p$ solution of BSDE$(\xi^n,\tau,g^n)$ and BSDE$(\xi,\tau,g)$, respectively, in the space of $H_\tau^p(\rho_\cdot;\R^{k}\times\R^{k\times d})$. If
$$\lim_{n\rightarrow\infty}\E\left[e^{p \int_{0}^{\tau}\rho_r{\rm d}r}|\xi^n-\xi|^p+\left(\int_0^\tau e^{\int_0^s\rho_r{\rm d}r}\left|g^n(s,Y_s,Z_s)-g(s,Y_s,Z_s)\right|{\rm d}s\right)^p\right]=0,$$
then
$$
\lim_{n\rightarrow\infty}\E\left[\sup_{s\in[0,\tau]}\left(e^{p \int_{0}^{s}\rho_r{\rm d}r}|Y_s^n-Y_s|^p\right)+\left(\int_0^{\tau} e^{2\int_0^s \rho_r{\rm d}r}|Z_s^n-Z_s|^2{\rm d}s\right)^{\frac{p}{2}}\right]=0.\vspace{0.2cm}
$$
\begin{proof}
For each $n\geq1$, let
$$\hat{Y}_\cdot^n:=Y_\cdot^n-Y_\cdot, \ \ \hat{Z}_\cdot^n:=Z_\cdot^n-Z_\cdot \ \ {\rm and}\ \ \hat{\xi}^n:=\xi^n-\xi.$$
Then, for each $n\geq1$, $(\hat{Y}_\cdot^n,\hat{Z}_\cdot^n)$ is a weighted $L^p$ solution of the following BSDE:
\begin{align*}
\hat{Y}_t^n=\hat{\xi}^n+\int_t^\tau \hat{g}^n(s,\hat{Y}_s^n,\hat{Z}_s^n){\rm d}s-\int_t^\tau\hat{Z}_s^n{\rm d}B_s, \ t\in[0,\tau],
\end{align*}
where the generator $\hat{g}^n$ is defined by
$$\hat{g}^n(t,y,z):=g^n(t,y+Y_t,z+Z_t)-g(t,Y_t,Z_t),
~~(t,y,z)\in[0,\tau]\times\R^k\times\R^{k\times{d}}.$$
It follows from \ref{A:H4} and \ref{A:H5} that for each $n\geq1$ and $(y,z)\in\R^k\times\R^{k\times{d}}$,
\begin{align*}
\begin{split}
\left<\hat{y},\hat{g}^n(t,y,z)\right>&=\left<\hat{y},g^n(t,y+Y_t,z+Z_t)-g(t,Y_t,Z_t)\right>\\
&\leq \mu_t|y|+\nu_t|z|+|g^n(t,Y_t,Z_t)-g(t,Y_t,Z_t)|, \ \ t\in[0,\tau],
\end{split}
\end{align*}
which implies that for each $n\geq1$, the generator $\hat g^n$ satisfies assumption \ref{A:A} with $\bar{p}:=p$, $\bar\mu_t:=\mu_t$, $\bar\nu_t:=\nu_t$, $\bar\rho_t:=\rho_t$ and $f_t:=|g^n(t,Y_t,Z_t)-g(t,Y_t,Z_t)|$. Then, it follows from \cref{pro2.1} with $r=t=0$ that there exists a constant $C_{p,\theta}>0$ depending only on $p$ and $\theta$ such that
\begin{align*}
\begin{split}
&\E\left[\sup_{s\in[0,\tau]}\left(e^{p \int_{0}^{s}\rho_r{\rm d}r}|Y_s^n-Y_s|^p\right)+\left(\int_0^{\tau} e^{2\int_0^s \rho_r{\rm d}r}|Z_s^n-Z_s|^2{\rm d}s\right)^{\frac{p}{2}}\right]\\
&\ \ \leq C_{p,\theta}\E\left[e^{p \int_{0}^{\tau}\rho_r{\rm d}r}|\xi^n-\xi|^p+\left(\int_0^\tau e^{\int_0^s\rho_r{\rm d}r}\left|g^n(s,Y_s,Z_s)-g(s,Y_s,Z_s)\right|{\rm d}s\right)^p\right].
\end{split}
\end{align*}
Thus, the desired assertion follows by sending $n$ to infinity in the last inequality.
\end{proof}
\end{thm}
\section{Proof of the existence part of \cref{thm:3.1}}
\setcounter{equation}{0}
In this section, we will give the proof of the existence part of \cref{thm:3.1} by six steps. In the first three steps we study the special case that $g$ does not depend on the state variable $z$ and $p=2$, and the general case is tackled with in the subsequent three steps.

For the sake of clarity, let $\tilde{\rho}_\cdot$ be an $(\F_t)$-progressively measurable process satisfying
 $\tilde{\rho}_\cdot \geq \mu_\cdot$ and $\int_0^\tau |\tilde{\rho}_t|\dif t<+\infty$, and the generator $g$ be independent of $z$ and satisfy the following assumptions:
\begin{enumerate}
\renewcommand{\theenumi}{(H1')}
\renewcommand{\labelenumi}{\theenumi}
\item\label{A:H1'} $\E\left[\left(\int_0^\tau e^{\int_0^s\tilde{\rho}_r{\rm d}r}|g(s,0)|{\rm d}s\right)^2\right]<+\infty$;
\renewcommand{\theenumi}{(H2')}
\renewcommand{\labelenumi}{\theenumi}
\item\label{A:H2'} ${\rm d}\mathbb{P}\times{\rm d} t-a.e.$, $y \mapsto g(\omega,t,y)$ is continuous;
\renewcommand{\theenumi}{(H3')}
\renewcommand{\labelenumi}{\theenumi}
\item\label{A:H3'} There exists an $(\F_t)$-progressively measurable process $(\bar{\alpha}_t)_{t\in[0,\tau]}$ satisfying $\essinf\limits_{t\in[0,\tau]}\bar{\alpha}_t>0$ such that for each $r\in \R_+$, it holds that
$$\E\left[\int_0^\tau \bar{\alpha}_t\overline{\psi}(t,r\bar{\alpha}_t){\rm d}t\right]<+\infty$$
with
$$\overline{\psi}(\omega,t,r):=\sup_{|y|\leq r} \left|g(\omega,t,y)-g(\omega,t,0)\right|, \ \ t\in[0,\tau(\omega)];$$
\renewcommand{\theenumi}{(H4')}
\renewcommand{\labelenumi}{\theenumi}
\item\label{A:H4'} For each $(y_1, y_2)\in \R^k\times\R^k$,  we have
$$
\left\langle y_1-y_2,g(\omega,t,y_1)-g(\omega,t,y_2)\right\rangle\leq \mu_t(\omega)|y_1-y_2|^2,\ \ \ t\in[0,\tau(\omega)].
$$
\end{enumerate}
In the first three steps we aim to verify that for each $\xi\in L_\tau^2(\tilde{\rho}_\cdot;\R^{k})$, the following BSDE
\begin{align}\label{BSDE2.01}
  y_t=\xi+\int_t^\tau g(s,y_s){\rm d}s-\int_t^\tau z_s{\rm d}B_s, \ \ t\in[0,\tau]
\end{align}
admits a unique weighted $L^2$ solution in the space of $H_\tau^2(\tilde{\rho}_\cdot;\R^{k}\times\R^{k\times d})$.\vspace{0.2cm}

It is obvious that if assumption \ref{A:H3'} holds for a process $(\bar{\alpha}_t)_{t\in [0,\tau]}$, then it also true for the process
$$
(\essinf\limits_{0 \leq s\leq t}{\bar{\alpha}_s})\wedge (e^{-\int_0^t {\tilde{\rho}_s}^+{\rm d}s}),\ \ t\in [0,\tau].
$$
Thus, without loss of generality we can assume that the $\bar\alpha_\cdot$ in \ref{A:H3'} is nonincreasing and satisfies
\begin{align}\label{*}
0<\bar{\alpha}_\tau\leq \bar{\alpha}_t \leq e^{-\int_0^t{\tilde{\rho}_s}^+{\rm d}s}\leq 1, \ \ \ t\in[0,\tau].
\end{align}

{\bf First Step:} We prove that BSDE \eqref{BSDE2.01} admits a weighted $L^2$ solution in $H_\tau^2({\tilde{\rho}_\cdot}^+;\R^{k}\times\R^{k\times d})$ provided that assumptions \ref{A:H2'} and \ref{A:H4'} hold and there exists a nonnegative constant $K$ such that
\begin{align}\label{2.6}
\begin{split}
|\xi|\leq K\bar{\alpha}_\tau\ \ \ {\rm and}\ \  \ |g(t,\cdot)|\leq Ke^{-t}\bar{\alpha}_t, \ \ \ t\in[0,\tau].
\end{split}
\end{align}
In fact, in light of \eqref{2.6} and \eqref{*}, we have
$$
\E\left[e^{2 \int_0^\tau \tilde{\rho}^+_r {\rm d}r}|\xi|^2 +\left(\int_0^\tau e^{\int_0^s\tilde{\rho}^+_r{\rm d}r}|g(s,0)|{\rm d}s\right)^2\right]\leq 2K^2,
$$
and
$$
\E\left[\int_0^\tau e^{ \int_0^t\tilde{\rho}^+_s{\rm d}s}\overline{\psi}(t,r\bar{\alpha}_t){\rm d}t\right]
\leq 2K\E\left[\int_0^\tau e^{ \int_0^t\tilde{\rho}^+_s{\rm d}s} e^{-t}\bar{\alpha}_t~{\rm d}t\right]\leq 2K.\vspace{0.2cm}
$$
Note that \ref{A:H4'} is also true when the process $\mu_\cdot$ is replaced with a larger process $\tilde{\rho}^+_\cdot$. The desired conclusion follows immediately from Proposition 4.1 in \citet{Li2025}.\vspace{0.2cm}

{\bf Second Step:} We prove that BSDE \eqref{BSDE2.01} admits a weighted $L^2$ solution in $H_\tau^2({\tilde{\rho}_\cdot}^+;\R^{k}\times\R^{k\times d})$ provided that assumptions \ref{A:H2'}-\ref{A:H4'} hold and there exists a nonnegative constant $K$ such that
  \begin{align}\label{22.6}
  \begin{split}
  |\xi|\leq K \bar{\alpha}_\tau^3 \ \ \ {\rm and}\ \ \ |g(t,0)|\leq Ke^{-t}\bar{\alpha}_t^3, \ \ \ t\in[0,\tau].
  \end{split}
  \end{align}

Now, we assume that assumptions \ref{A:H2'}-\ref{A:H4'} and \eqref{22.6} are in force. Fix a real $r>0$ assigned later and define the following auxiliary function: for each $t\in[0,\tau]$,
\begin{align}\label{Matrix}
\theta_{r}^{\bar{\alpha}_\cdot}(t,u):=\left\{\begin{aligned}
&\bar{\alpha}_t^2, \quad \quad \quad \quad \quad \quad \quad 0\leq u\leq r\bar{\alpha}_t^2;\\
&-u+(r+1)\bar{\alpha}_t^2, \quad r\bar{\alpha}_t^2< u\leq(r+1)\bar{\alpha}_t^2;\\
&0, \quad \quad \quad \quad \quad \quad \quad \quad u>(r+1)\bar{\alpha}_t^2.
\end{aligned}\right.
\end{align}
It is obvious that for each $t\in[0,\tau]$ and $u,u_1,u_2\geq 0$,
\begin{align}\label{4.1}
0\leq\theta_{r}^{\bar{\alpha}_\cdot}(t,u)\leq \bar{\alpha}_t^2\vspace{-0.2cm}
\end{align}
and
\begin{align}\label{4.2}
|\theta_{r}^{\bar{\alpha}_\cdot}(t,u_1)-\theta_{r}^{\bar{\alpha}_\cdot}(t,u_2)| \leq|u_1-u_2|.
\end{align}
Then we define for each $n\geq1$ and $y\in \R^k$,
$$g^n(t,y):=\theta_{r}^{\bar{\alpha}_\cdot}(t,|y|)\left(g(t,y)-g(t,0)\right) \frac{n{e^{-t}}}{{\overline{\psi}(t,(r+1)\bar{\alpha}_t)}\vee \left(ne^{-t}\bar{\alpha}_t^2\right)}+g(t,0), \ \ t\in[0,\tau].$$
Clearly, $g^n$ satisfies \ref{A:H2'} for each $n\geq 1.$ In light of the definitions of $\overline{\psi}(\cdot,\cdot)$ and $\theta_{r}^{\bar{\alpha}_\cdot}$, along with \eqref{*}, \eqref{22.6} and \eqref{4.1}, we can deduce that for each $n\geq1$ and $y\in \R^k$,
\begin{align}\label{4.2*}
|g^n(t,y)|\leq (n+K)e^{-t}\bar{\alpha}_t, \ \ t\in[0,\tau].
\end{align}
Moreover, in light of \eqref{*}, \eqref{Matrix}-\eqref{4.2} and \ref{A:H4'} for the generator $g$, by a similar argument to that in \cite[page 24, lines 7-21]{Li2025} we can obtain that for each $n\geq 1$, $g^n$ satisfies \ref{A:H4'} with $\mu_\cdot$ being replaced by
$$\tilde{\mu}_t:=\tilde\rho^+_t +ne^{-t},\ \ t\in [0,\tau].$$
Thus, in light of the fact that the space $H_\tau^2(\tilde{\mu}_\cdot;\R^{k}\times\R^{k\times d})$ does not change when the process $\tilde{\mu}_\cdot$ is replaced with $\tilde\rho^+_\cdot$ along with \eqref{*}, \eqref{22.6} and  \eqref{4.2*}, by the analysis of the first step we can conclude that for each $n\geq1$, BSDE $(\xi,\tau,g^n)$ admits a weighted $L^2$ solution $(y_t^n,z_t^n)_{t\in[0,\tau]}$ in $H_\tau^2({\tilde{\rho}_\cdot}^+;\R^{k}\times\R^{k\times d})$.

On the other hand, by assumption \ref{A:H4'} for the generator $g$ along with \eqref{22.6} and \eqref{4.1} we can deduce that for each $n\geq1$ and $y\in \R^k$,
\begin{align*}
\begin{split}
\left\langle \hat{y},g^n(t,y)\right\rangle=&\left\langle \hat{y},g(t,y)-g(t,0)\right\rangle\theta_{r}^{\bar{\alpha}_\cdot}(t,|y|) \frac{ne^{-t}}{{\overline{\psi}(t,(r+1)\bar{\alpha}_t)}\vee \left(ne^{-t}\bar{\alpha}_t^2\right)}+\left\langle \hat{y},g(t,0)\right\rangle\\
\leq &\mu_t|y|+Ke^{-t}\bar{\alpha}_t^3, \ \ \ t\in[0,\tau],
\end{split}
\end{align*}
which means that assumption (A) is satisfied by the generator $g^n$ with $\bar{p}:=2$, $\bar\mu_t:=\mu_t$, $\bar\nu_t:=0$, ${\bar\rho_t}:=\tilde{\rho}_t^+$ and $f_t:=Ke^{-t}\bar{\alpha}_t^3$.
Then, in light of the fact that $\bar\alpha_\cdot$ is nonincreasing along with
\eqref{*} and \eqref{22.6}, \cref{pro2.02} yields that there exists a uniform constant $C>0$ such that for each $n\geq1$ and $t\geq0$,
\begin{align}\label{4.22}
\begin{split}
&|y_{t\wedge\tau}^{n}|^2\leq e^{2 \int_0^{t\wedge\tau}{\tilde{\rho}_r}^+{\rm d}r}|y_{t\wedge\tau}^{n}|^2\\
\leq& C\left(\E\left[e^{2 \int_0^\tau{\tilde{\rho}_r}^+{\rm d}r}|\xi|^2\bigg|\F_{t\wedge\tau}\right]+\E\left[\left(
\int_{t\wedge\tau}^{\tau}e^{ \int_0^s{\tilde{\rho}_r}^+{\rm d}r}Ke^{-s}\bar{\alpha}_s^3{\rm d}s\right)^2\bigg|\F_{t\wedge\tau}\right]\right)\\
\leq &C\left(\E\left[e^{2 \int_0^\tau{\tilde{\rho}_r}^+{\rm d}r}K^2 \bar{\alpha}_{\tau}^2\bar{\alpha}_{t\wedge\tau}^4\bigg|\F_{t\wedge\tau}\right] +\E\left[\left(\int_{t\wedge\tau}^{\tau}e^{ \int_0^s{\tilde{\rho}_r}^+{\rm d}r}Ke^{-s}\bar{\alpha}_s\bar{\alpha}_{t\wedge\tau}^2{\rm d}s\right)^2\bigg|\F_{t\wedge\tau}\right]\right)\\
\leq &C\left(\E\left[e^{2 \int_0^\tau{\tilde{\rho}_r}^+{\rm d}r}K^2 \bar{\alpha}_{\tau}^2\bigg|\F_{t\wedge\tau}\right] +\E\left[\left(\int_{0}^{\tau}e^{ \int_0^s{\tilde{\rho}_r}^+{\rm d}r}Ke^{-s}\bar{\alpha}_s{\rm d}s\right)^2\bigg|\F_{t\wedge\tau}\right]\right)\bar{\alpha}_{t\wedge\tau}^4\\
\leq&r^2\bar{\alpha}_{t\wedge\tau}^4
\end{split}
\end{align}
with $r:=K\sqrt{2C}$. Then for each $n\geq1$, we have
\begin{align}\label{2.32}
|y_t^n|\leq r\bar{\alpha}_t^2,\ \ t\in[0,\tau].
\end{align}
By the definition of $\theta_r^{\bar{\alpha}_\cdot}$ we know that $g^n$ can be written as
$$\tilde{g}^n(t,y):=\left(g(t,y)-g(t,0)\right) \frac{ne^{-t}\bar{\alpha}_t^2}{{\overline{\psi}(t,(r+1)\bar{\alpha}_t)}\vee \left(ne^{-t}\bar{\alpha}_t^2\right)}+g(t,0), \ \ t\in[0,\tau],$$
and then $(y_t^n,z_t^n)_{t\in[0,\tau]}\in H_\tau^2({\tilde{\rho}_\cdot}^+;\R^{k}\times\R^{k\times d})$ is a weighted $L^2$ solution of BSDE $(\xi,\tau,\tilde{g}^n)$.

Next, we prove that $\{(y_\cdot^n,z_\cdot^n)\}^{+\infty}_{n=1}$ is a Cauchy sequence in the space of $H_\tau^2({\tilde{\rho}_\cdot^+};\R^{k}\times\R^{k\times d})$. For each $n,i\geq1$, let $\hat{y}_\cdot^{n,i}:=y_\cdot^{n+i}-y_\cdot^n$, $\hat{z}_\cdot^{n,i}:=z_\cdot^{n+i}-z_\cdot^n$. Then we have
\begin{align*}
\hat y_t^{n,i}=\int_t^\tau \hat {g}^{n,i}(s,\hat y_s^{n,i}){\rm d}s-\int_t^\tau\hat z_s^{n,i}{\rm d}B_s, \ \ t\in[0,\tau],
\end{align*}
where for each $y\in \R^k$, $$\hat {g}^{n,i}(s,y):=\tilde{g}^{n+i}(s,y+y_s^n)-\tilde{g}^{n}(s,y_s^n), \ s\in[0,\tau].$$
In light of the definitions of $\tilde{g}^{n}$ and $\bar\psi(\cdot,\cdot)$, assumption \ref{A:H4'} for the generator $g$ and \eqref{2.32}, by an identical argument to that in \cite[page 25, lines 7-17]{Li2025} we can derive that for each $n, i\geq1$ and $y\in \R^k$,
$$\left\langle \hat{y}, ~\hat{g}^{n,i}(t,y)\right\rangle\leq \mu_t |y|+\overline{\psi}(t,(r+1)\bar{\alpha}_t){\bf 1}_{\overline{\psi}(t,(r+1)\bar{\alpha}_t)>{{n}e^{-t}\bar{\alpha}_t^2}}, \ \ t\in[0,\tau],$$
which means that assumption (A) is satisfied by $\hat{g}^{n,i}$ with $\bar p:=2$, $\bar\mu_t:=\mu_t$, $\bar\nu_t:=0$, ${\bar\rho_t}:=\tilde{\rho}_t^+$ and
$$
f_t:=\overline{\psi}(t,(r+1)\bar{\alpha}_t){\bf 1}_{\overline{\psi}(t,(r+1)\bar{\alpha}_t)>{{n}e^{-t}\bar{\alpha}_t^2}}.
$$
On the other hand, it follows from \eqref{4.22} and \eqref{*} that for each $n\geq1$,
\begin{align}\label{4.23}
e^{2\int_0^{t}{\tilde{\rho}_r}^+{\rm d}r}|y_{t}^{n}|\leq r\bar{\alpha}_t,\ \ t\in [0,\tau].
\end{align}
In light of \eqref{4.23}, by virtue of (i) of \cref{rmk:2.5} with $r=t=0$ we can deduce that there exists a uniform constant $\overline{C}>0$ such that for each $n,i\geq1$,
\begin{align}\label{2101.1}
\begin{split}
&\E\left[\sup_{s\in[0,\tau]}\left(e^{2 \int_0^s{\tilde{\rho}_r^+}{\rm d}r}|\hat{y}_s^{n,i}|^2\right)\right]+\E\left[\int_{0}^{\tau}e^{2 \int_0^s{\tilde{\rho}_r^+}{\rm d}r}|\hat{z}_s^{n,i}|^2{\rm d}s\right]\\
&\ \ \leq \overline{C}\E\left[
\int_{0}^{\tau}e^{2 \int_0^s{\tilde{\rho}_r^+}{\rm d}r}|\hat{y}_s^{n,i}|~\overline{\psi}(s,(r+1)\bar{\alpha}_s){\bf 1}_{\overline{\psi}(s,(r+1)\bar{\alpha}_s)>{{n}e^{-s}\bar{\alpha}_s^2}}{\rm d}s\right]\\
&\ \ \leq 2r\overline{C}\E\left[
\int_{0}^{\tau}\bar{\alpha}_s\overline{\psi}(s,(r+1)\bar{\alpha}_s){\bf 1}_{\overline{\psi}(s,(r+1)\bar{\alpha}_s)>{{n}e^{-s}\bar{\alpha}_s^2}}{\rm d}s\right].
\end{split}
\end{align}
Moreover, by \ref{A:H3'} and Lebesgue's dominated convergence theorem we have
$$
\lim\limits_{n\rightarrow\infty}\E\left[
\int_{0}^{\tau}\bar{\alpha}_s\overline{\psi}(t,(r+1)\bar{\alpha}_s){\bf 1}_{\overline{\psi}(t,(r+1)\bar{\alpha}_s)>{{n}e^{-s}\bar{\alpha}_s^2}}{\rm d}s\right]=0.
$$
Thus, by taking first the supremum with respect to $i$ and then the upper limit with respect to $n$ on both sides of \eqref{2101.1}, we derive that $\{(y_\cdot^n,z_\cdot^n)\}^{+\infty}_{n=1}$ is a Cauchy sequence in the space of $H_\tau^2({\tilde{\rho}_\cdot^+};\R^{k}\times\R^{k\times d})$.

It can be seen from the previous proof that in order to tackle with the more general growth of the generator $g$ (namely, \ref{A:H3'}), we adopt a new truncation for $\xi$ and $g(t,0)$ (see \eqref{22.6}), construct a new auxiliary function $\theta_{r}^{\bar{\alpha}_\cdot}$ (see \eqref{Matrix}), and then obtain the uniform estimates \eqref{2.32} and \eqref{4.23} with respect to the process $|y^n_\cdot|$. These are different from those used in \citet{Li2025}.

Finally, we denote by $(y_t,z_t)_{t\in[0,\tau]}$ the limit of the Cauchy sequence $\{(y_t^n,z_t^n)_{t\in[0,\tau]}\}_{n=1}^\infty$ in the space of $H_\tau^2({\tilde{\rho}^+_\cdot};\R^{k}\times\R^{k\times d})$ and show in detail that $(y_\cdot,z_\cdot)$ solves BSDE \eqref{BSDE2.01}. First, since $y^n_\cdot\rightarrow y_\cdot$ as $n\To \infty$ in the space of $S_\tau^2({\tilde{\rho}^+_\cdot};\R^{k})$, passing to a subsequence if needed, still denoted by $y^n_\cdot$, we know that
\begin{align}\label{2.02}
\lim\limits_{n\rightarrow\infty}\sup_{t\in[0,\tau]}|y_t^{n}-y_t|=0.
\end{align}
Then, for almost all $\omega,$ there exists a constant $M_1(\omega)>0$ depending only on $\omega$ such that
$$
\sup_{n\geq1} \sup_{t\in[0,\tau]}|y_t^{n}-y_t|\leq M_1,
$$
which yields the existence of a constant $M_2(\omega)>0$ depending only on $\omega$ such that
\begin{align}\label{2.011}
\sup_{n\geq1} \sup_{t\in[0,\tau]}|y_t^{n}|\leq M_2=M_3 \essinf_{t\in[0,\tau]}\bar{\alpha}_t\leq M_3\bar{\alpha}_t
\end{align}
with $M_3\triangleq \frac{M_2}{\essinf\limits_{t\in[0,\tau]}\bar{\alpha}_t}>0$ being a constant depending only on $\omega$ because of $\essinf\limits_{t\in[0,\tau]}\bar{\alpha}_t>0$. Next, since $z^n_\cdot\rightarrow z_\cdot$ as $n\To \infty$ in $M_\tau^2({\tilde{\rho}^+_\cdot};\R^{k\times d})$, passing to a subsequence if needed, still denoted by $z^n_\cdot$, we obtain
$$
\lim\limits_{n\rightarrow\infty}\E\left[\bigg|\int_0^\tau (z^n_s-z_s){\rm d}B_s\bigg|^2\right]=\lim\limits_{n\rightarrow\infty}\E\left[\int_0^\tau |z^n_s-z_s|^2{\rm d}s\right]=0,
$$
and then
\begin{align}\label{2.04}
\bigg|\int_0^\tau z^n_s{\rm d}B_s-\int_0^\tau z_s{\rm d}B_s\bigg|\rightarrow0, \ \ \ n\rightarrow \infty.
\end{align}
Furthermore, it follows from the definition of $\tilde{g}^n$ and \ref{A:H2} that for each $t\in[0,\tau],$
\begin{align}\label{2.05}
\tilde{g}^n(t,y_t^n)\rightarrow g(t,y_t), \ \ \ n\rightarrow \infty.
\end{align}
And, by virtue of the definition of $\tilde{g}^n$ and \eqref{2.011} we deduce that for each $n\geq1$,
\begin{align}\label{2.07}
|\tilde{g}^n(t,y_t^n)|\leq |g(t,y_t^n)-g(t,0)|+|g(t,0)|\leq \overline{\psi}(t,M_3\bar{\alpha}_t)+|g(t,0)|,\ \ t\in [0,\tau],
\end{align}
and then by virtue of \ref{A:H3'} and  $\essinf\limits_{t\in[0,\tau]}\bar{\alpha}_t>0$,
\begin{align}\label{2.08}
\begin{split}
\int_0^\tau \overline{\psi}(t,M_3\bar{\alpha}_t){\rm d}t&=\int_0^\tau \frac{1}{\bar{\alpha}_t} \bar{\alpha}_t\overline{\psi}(t,M_3\bar{\alpha}_t){\rm d}t\\
&\leq \frac{1}{\essinf\limits_{t\in[0,\tau]}\bar{\alpha}_t} \int_0^\tau \bar{\alpha}_t\overline{\psi}(t,M_3\bar{\alpha}_t){\rm d}t<+\infty.
\end{split}
\end{align}
Thus, it follows from \eqref{22.6}, \eqref{2.05}-\eqref{2.08} and Lebesgue's dominated convergence theorem that
\begin{align}\label{2.09}
\int_0^\tau \tilde{g}^n(t,y_t^n){\rm d}t\rightarrow \int_0^\tau g(t,y_t){\rm d}t, \ \ \ n\rightarrow \infty.
\end{align}
Finally, in light of \eqref{2.02}, \eqref{2.04}, \eqref{2.09}, passing to the limit under the uniform convergence in probability (ucp for short) for BSDE $(\xi,\tau,\tilde{g}^n)$, we can conclude that $(y_t,z_t)_{t\in[0,\tau]}$ solves BSDE \eqref{BSDE2.01}.\vspace{0.2cm}

{\bf Third Step:} We remove the truncation condition \eqref{22.6} and prove the existence of a weighted $L^2$ solution to BSDE \eqref{BSDE2.01} in the space of $H_\tau^2({\tilde{\rho}_\cdot};\R^{k}\times\R^{k\times d})$ under \ref{A:H1'}-\ref{A:H4'} and $\xi\in L_\tau^2(\tilde{\rho}_\cdot;\R^{k})$.

For each $x\in\R^k$, $r>0$ and $n\geq1$, let $q_r(x):=\frac{xr}{|x|\vee r}$, $\xi_n:=q_{n\bar{\alpha}_{\tau}^3} (\xi)$ and
\begin{align}\label{3.1}
\overline{g}_n(t,y):=g(t,y)-g(t,0)+q_{ne^{-t} \bar{\alpha}_t^3}(g(t,0)), \ \ t\in[0,\tau].
\end{align}
Then, for each $n\geq1$ we have
\begin{align}\label{1.13}
|\xi_n|\leq n\bar{\alpha}_{\tau}^3\ \ \ {\rm and}\ \  \ |\overline{g}_n(t,0)|\leq ne^{-t}\bar{\alpha}_{t}^3, \ \ t\in[0,\tau].
\end{align}
By combining \eqref{*}, \eqref{1.13} and assumptions \ref{A:H2'}-\ref{A:H4'} of $g$, we know that for each $n\geq1$, $\xi_n$ and $\overline{g}_n$ satisfy all assumptions in the second step. Then, BSDE $(\xi_n,\tau,\overline{g}_n)$ admits a weighted $L^2$ solution in the space of $H_\tau^2({\tilde{\rho}_\cdot^+};\R^{k}\times\R^{k\times d}) \subseteq H_\tau^2({\tilde{\rho}_\cdot};\R^{k}\times\R^{k\times d})$ for each $n\geq1$, denoted by $(y_t^n,z_t^n)_{t\in[0,\tau]}$.

In the sequel, for each pair of integers $n, i\geq1$, let
$$\hat{\xi}^{n,i}:=\xi_{n+i}-\xi_n, \  \ \hat y_\cdot^{n,i}:=y_\cdot^{n+i}-y_\cdot^n, \ \ \hat z_\cdot^{n,i}:=z_\cdot^{n+i}-z_\cdot^n.$$
Then
\begin{align*}
\hat y_t^{n,i}=\hat{\xi}^{n,i}+\int_t^\tau\hat{g}^{n,i}(s,\hat y_s^{n,i}){\rm d}s-\int_t^\tau\hat z_s^{n,i}{\rm d}B_s, \ \ t\in[0,\tau],
\end{align*}
where for each $y\in \R^k$, $$\hat{g}^{n,i}(s,y):=\overline{g}_{n+i}(s,y+y_s^n)-\overline{g}_n(s,y_s^n), \ s\in[0,\tau].$$
It follows from \ref{A:H4'} that for each $y\in \R^k$,
\begin{align*}
\begin{split}
\left<\hat{y},~\hat{g}^{n,i}(t,y)\right>&=\left<\hat{y}, ~\overline{g}_{n+i}(t,y+y_t^n)-\overline{g}_{n+i}(t,y_t^n)+ \overline{g}_{n+i}(t,y_t^n)-\overline{g}_n(t,y_t^n)\right>\\
&\leq \mu_t|y| +\bigg|q_{(n+i)e^{-t}\bar{\alpha}_{t}^3}(g(t,0)) -q_{ne^{-t}\bar{\alpha}_{t}^3}(g(t,0))\bigg|\\
&\leq \mu_t|y|+|g(t,0)|{\bf 1}_{|g(t,0)|>ne^{-t}\bar{\alpha}_{t}^3}, \ \ t\in[0,\tau],
\end{split}
\end{align*}
which means that $\hat{g}^{n,i}$ satisfies assumption \ref{A:A} with $\bar{p}:=2$, $\bar\mu_t:={\mu}_t$, $\bar\nu_t:=0$, ${\bar\rho_t}:=\tilde{\rho}_t$ and
$$
f_t=|g(t,0)|{\bf 1}_{|g(t,0)|>ne^{-t}\bar{\alpha}_{t}^3}.
$$
Then, by \cref{pro2.1} with $r=t=0$ we can deduce that there exists a uniform constant $C>0$ such that for each $n, \ i\geq1$,
\begin{align}\label{3.6}
\begin{split}
&\E\left[\sup_{s\in[0,\tau]}\left(e^{2 \int_0^s{\tilde{\rho}_r}{\rm d}r}|\hat{y}_s^{n,i}|^2\right)\right]+\E\left[\int_{0}^{\tau}e^{2 \int_0^s{\tilde{\rho}_r}{\rm d}r}|\hat{z}_s^{n,i}|^2{\rm d}s\right]\\
&\ \ \leq C\E\left[e^{2 \int_0^\tau {\tilde{\rho}_r}{\rm d}r}|\xi|^2{\bf 1}_{|\xi|>n\bar{\alpha}_{\tau}^3}+\left(\int_0^\tau e^{ \int_0^s{\tilde{\rho}_r}{\rm d}r}|g(s,0)|{\bf 1}_{|g(s,0)|>ne^{-s}\bar{\alpha}_{s}^3}{\rm d}s\right)^2\right].
\end{split}
\end{align}
According to \ref{A:H1'} and Lebesgue's dominated convergence theorem, by taking first the supremum with respect to $i$ and then the upper limit with respect to $n$ on both sides of \eqref{3.6}, we can conclude that $\{(y_\cdot^n,z_\cdot^n)\}^{+\infty}_{n=1}$ is a Cauchy sequence in the space of $H_\tau^2({\tilde{\rho}_\cdot};\R^{k}\times\R^{k\times d})$.

Finally, we denote by $(y_t,z_t)_{t\in[0,\tau]}$ the limit of the Cauchy sequence $\{(y_t^n,z_t^n)_{t\in[0,\tau]}\}_{n=1}^\infty$ in the space of $H_\tau^2({\tilde{\rho}_\cdot};\R^{k}\times\R^{k\times d})$ and show that $(y_t,z_t)_{t\in[0,\tau]}$ solves BSDE \eqref{BSDE2.01}. First, since $y^n_\cdot\rightarrow y_\cdot$ as $n\To \infty$ in the space of $S_\tau^2({\tilde{\rho}_\cdot};\R^{k})$, passing to a subsequence if needed, still denoted by $y^n_\cdot$, we have
$$
\lim\limits_{n\rightarrow\infty}\sup_{t\in[0,\tau]}\left(e^{\int_0^t
\tilde{\rho}_r{\rm d}r}|y_t^{n}-y_t|\right)=0.
$$
Observe by $\int_0^\tau |\tilde{\rho}_r|{\rm d}r<+\infty$ that for each $n\geq1$,
\begin{align}\label{3.61}
\begin{split}
\sup_{t\in[0,\tau]}|y_t^{n}-y_t|&=\sup_{t\in[0,\tau]}\left(e^{\int_0^t
\tilde{\rho}_r{\rm d}r}|y_t^{n}-y_t|~ e^{-\int_0^t\tilde{\rho}_r{\rm d}r}\right)\\
&\leq e^{\int_0^\tau|\tilde{\rho}_r|{\rm d}t} \sup_{t\in[0,\tau]}\left(e^{\int_0^t\tilde{\rho}_r{\rm d}r}|y_t^{n}-y_t|\right).
\end{split}
\end{align}
We know that \eqref{2.02} and \eqref{2.011} still hold. Next, since $z^n_\cdot\rightarrow z_\cdot$ as $n\To \infty$ in $M_\tau^2({\tilde{\rho}_\cdot};\R^{k\times d})$, passing to a subsequence if needed, still denoted by $z^n_\cdot$, we obtain
$$
\int_0^\tau e^{2\int_0^s\tilde{\rho}_r{\rm d}r}|z^n_s|^2{\rm d}s<+\infty, \ \ \int_0^\tau e^{2\int_0^s\tilde{\rho}_r{\rm d}r}|z_s|^2{\rm d}s<+\infty
$$
and $$\lim\limits_{n\rightarrow\infty}\int_0^\tau e^{2\int_0^s\tilde{\rho}_r{\rm d}r}|z^n_s-z_s|^2{\rm d}s=0.$$
Note that for each $n\geq1$,
\begin{align*}
\int_0^\tau |z^n_s-z_s|^2{\rm d}s&=\int_0^\tau e^{-2\int_0^s\tilde{\rho}_r{\rm d}r}~ e^{2\int_0^s\tilde{\rho}_r{\rm d}r}|z^n_s-z_s|^2{\rm d}s\\
&\leq e^{2\int_0^\tau|\tilde{\rho}_r|{\rm d}t} \int_0^\tau e^{2\int_0^s\tilde{\rho}_r{\rm d}r}|z^n_s-z_s|^2{\rm d}s.
\end{align*}
We have $\lim\limits_{n\rightarrow\infty}\int_0^\tau |z^n_s-z_s|^2{\rm d}s=0$ and
$$
\int_0^\tau |z^n_s|^2{\rm d}s<+\infty, \ \ \int_0^\tau |z_s|^2{\rm d}s<+\infty.
$$
Then, by Proposition 2.26 in Section 3.2 of \citet{Karatzas1991} we know that
\begin{align}\label{0.22}
\bigg|\int_0^\tau z^n_s{\rm d}B_s-\int_0^\tau z_s{\rm d}B_s\bigg|\rightarrow0, \ \ \ n\rightarrow \infty.
\end{align}
Furthermore, it follows from \ref{A:H2} and \eqref{3.1} that for each $t\in [0,\tau],$
\begin{align}\label{0.23}
\overline{g}_n(t,y_t^n)\rightarrow g(t,y_t), \ \ \ n\rightarrow \infty.
\end{align}
And, by \eqref{2.011} and \eqref{3.1} we can deduce that for each $n\geq1$ and $t\in [0,\tau],$
\begin{align}\label{0.24}
|\overline{g}_n(t,y_t^n)|\leq |g(t,y_t^n)-g(t,0)|+|g(t,0)|\leq \overline{\psi}(t,M_3\bar{\alpha}_t)+|g(t,0)|.
\end{align}
In addition, by \ref{A:H1'} and \ref{A:H3'} along with $\essinf\limits_{t\in[0,\tau]}\bar{\alpha}_t>0$ and $\int_0^\tau |\tilde{\rho}_r|{\rm d}r<+\infty$, we get \eqref{2.08} and
\begin{align}\label{0.26}
\begin{split}
\int_0^\tau |g(t,0)|{\rm d}t&=\int_0^\tau e^{-\int_0^t\tilde{\rho}_r{\rm d}r}~e^{\int_0^t\tilde{\rho}_r{\rm d}r}|g(t,0)|{\rm d}t\\
&\leq e^{\int_0^\tau|\tilde{\rho}_r|{\rm d}t} \int_0^\tau e^{\int_0^t\tilde{\rho}_r{\rm d}r}|g(t,0)|{\rm d}t<+\infty.
\end{split}
\end{align}
Thus, it follows from \eqref{2.08}, \eqref{0.23}-\eqref{0.26} and Lebesgue's dominated convergence theorem that
\begin{align}\label{0.27}
\int_0^\tau \overline{g}_n(t,y_t^n){\rm d}t\rightarrow \int_0^\tau g(t,y_t){\rm d}t, \ \ \ n\rightarrow \infty.
\end{align}
Finally, in light of \eqref{2.02}, \eqref{0.22}, \eqref{0.27}, passing to the limit under ucp for BSDE $(\xi_n,\tau,\overline{g}_n)$, we can conclude that $(y_t,z_t)_{t\in[0,\tau]}$ solves BSDE \eqref{BSDE2.01}. Thus, the proof of the first three steps has been completed.\vspace{0.2cm}

We now turn to consider the general case that $g$ can depend on $z$ and prove that the result of the existence part of \cref{thm:3.1} holds for the case of $p=2$. This is done in the subsequent two steps.\vspace{0.2cm}

{\bf Fourth Step:} In this step, under assumptions \ref{A:H1}-\ref{A:H5} of the generator $g$ with $p=2$ and $\rho_\cdot\geq\mu_\cdot+\frac{\theta}{2}\nu_\cdot^2$, we use the assertion of the third step to prove that for each $\xi\in L_\tau^2(\rho_\cdot;\R^k)$ and $V_\cdot\in M_\tau^2(\rho_\cdot;\R^{k\times d})$, the following BSDE
\begin{align}\label{BSDE3.1}
  y_t=\xi+\int_t^\tau g(s,y_s,V_s){\rm d}s-\int_t^\tau z_s{\rm d}B_s, \ \ t\in[0,\tau].
\end{align}
admits a unique weighted $L^2$ solution in the space of $H_\tau^2(\rho_\cdot;\R^{k}\times\R^{k\times d})$.

We first prove that the generator $g(t,y,V_t)$ satisfies assumptions \ref{A:H1'}-\ref{A:H4'} with $\rho_\cdot-\frac{\theta}{2}{\nu_\cdot}^2$ instead of $\tilde{\rho}_\cdot$. In fact, it is clear that $g(t,y,V_t)$ satisfies \ref{A:H2'} and \ref{A:H4'}. Moreover, by \ref{A:H1} and \ref{A:H5} of the generator $g$ and H\"{o}lder's inequality we derive that
\begin{align}\label{0.35}
&\E\left[\left(\int_0^\tau e^{ \int_0^s(\rho_r-\frac{\theta}{2}\nu_r^2){\rm d}r}|g(s,0,V_s)|{\rm d}s\right)^2\right]\leq\E\left[\left(\int_0^\tau e^{ \int_0^s(\rho_r-\frac{\theta}{2}\nu_r^2){\rm d}r}\left(|g(s,0,0)|+{\nu}_s|V_s|\right){\rm d}s\right)^2\right]\nonumber\\
&\ \ \leq 2\E\left[\left(\int_0^\tau e^{ \int_0^s{\rho}_r{\rm d}r}|g(s,0,0)|{\rm d}s\right)^2\right]+2\E\left[\left(\int_0^\tau e^{ \int_0^s{\rho}_r{\rm d}r}|V_s|e^{- \int_0^s\frac{\theta}{2} {\nu}_r^2{\rm d}r}{\nu}_s{\rm d}s\right)^2\right]\nonumber\\
&\ \ \leq 2\E\left[\left(\int_0^\tau e^{\int_0^s{\rho}_r{\rm d}r}|g(s,0,0)|{\rm d}s\right)^2\right]+2\E\left[\int_0^\tau e^{2 \int_0^s{\rho}_r{\rm d}r}|V_s|^2{\rm d}s\right]\E\left[\int_0^\tau e^{-\int_0^s\theta {\nu}_r^2{\rm d}r}{\nu}_s^2{\rm d}s\right]\nonumber\\
&\ \ \leq 2\E\left[\left(\int_0^\tau e^{\int_0^s{\rho}_r{\rm d}r}|g(s,0,0)|{\rm d}s\right)^2\right]+\frac{2}{\theta} \E\left[\int_0^\tau e^{2 \int_0^s{\rho}_r{\rm d}r}|V_s|^2{\rm d}s\right]<+\infty.
\end{align}
Hence, \ref{A:H1'} is true for $g(t,y,V_t)$. It follows from \ref{A:H5} that for each $r\in \R_+$ and $t\in[0,\tau]$,
\begin{align*}
\begin{split}
\overline{\psi}(t,r):&=\sup_{|y|\leq r} \left|g(t,y,V_t)-g(t,0,V_t)\right|\\
&=\sup_{|y|\leq r} \left|g(t,y,V_t)-g(t,y,0)+g(t,y,0)-g(t,0,0)+g(t,0,0)-g(t,0,V_t)\right|\\
&\leq2\nu_t|V_t|+\sup_{|y|\leq r}\left|g(t,y,0)-g(t,0,0)\right|=2\nu_t|V_t|+\psi(t,r),\\
\end{split}
\end{align*}
and then by taking $\bar{\alpha}_t:=\alpha_t \wedge e^{ \int_0^t(\rho_r-\frac{\theta}{2}\nu_r^2){\rm d}r}$ and
combining \eqref{0.35} and \ref{A:H3} of the generator $g$, we have
\begin{align*}
\begin{split}
\E\left[\int_0^\tau \bar{\alpha}_t\overline{\psi}(t,r\bar{\alpha}_t)dt\right]&\leq
\E\left[\int_0^\tau \bar{\alpha}_t(2\nu_t|V_t|+\psi(t,r\bar{\alpha}_t))\dif t\right]\\
&\leq
\E\left[2\int_0^\tau e^{ \int_0^t(\rho_r-\frac{\theta}{2}\nu_r^2){\rm d}r}\nu_t|V_t|\dif t+\int_0^\tau \alpha_t \psi(t,r\alpha_t)\dif t\right]<+\infty.
\end{split}
\end{align*}
So \ref{A:H3'} is also true for $g(t,y,V_t)$. Thus, according to the result of the third step with $\tilde{\rho}_\cdot=\rho_\cdot-\frac{\theta}{2} {\nu}_\cdot^2\geq\mu_\cdot$, we can conclude that BSDE \eqref{BSDE3.1} admits a weighted $L^2$ solution $(y_t,z_t)_{t\in[0,\tau]}$ in the space of
$$
H_\tau^2(\rho_\cdot-\frac{\theta}{2} {\nu}_\cdot^2;\R^{k}\times\R^{k\times d}).\vspace{-0.1cm}
$$

In the sequel, we prove that $(y_t,z_t)_{t\in[0,\tau]}$ belongs to $H_\tau^2(\rho_\cdot;\R^{k}\times\R^{k\times d})$. First of all, applying It\^{o}-Tanaka's formula to $e^{\int_{0}^{t}(\rho_r-\frac{\theta}{2} {\nu}_r^2){\rm d}r}|y_t|$ yields that for each $t\geq0,$
\begin{align}\label{0.37}
\begin{split}
e^{\int_{0}^{t\wedge\tau}(\rho_r-\frac{\theta}{2} {\nu}_r^2){\rm d}r}|y_{t\wedge\tau}|
\leq & e^{\int_{0}^{\tau}(\rho_r-\frac{\theta}{2} {\nu}_r^2){\rm d}r}|\xi|+\int_{t\wedge\tau}^{\tau} e^{\int_{0}^{s}(\rho_r-\frac{\theta}{2} {\nu}_r^2){\rm d}r}\left<\hat{y}_s,~g(s,y_s,V_s)\right>{\rm d}s\\
&-\int_{t\wedge\tau}^{\tau} e^{\int_{0}^{s}(\rho_r-\frac{\theta}{2} {\nu}_r^2){\rm d}r}(\rho_s-\frac{\theta}{2} {\nu}_s^2)|y_s|{\rm d}s-\int_{t\wedge\tau}^{\tau} e^{ \int_{0}^{s}(\rho_r-\frac{\theta}{2} {\nu}_r^2){\rm d}r}\langle \hat{y}_s,z_s{\rm d}B_s\rangle.
\end{split}
\end{align}
Using assumptions \ref{A:H4} and \ref{A:H5} we obtain
\begin{align}\label{0.38}
\begin{split}
\left<\hat{y}_t,~g(t,y_t,V_t)\right>-(\rho_t-\frac{\theta}{2} {\nu}_t^2)|y_t|\leq \nu_t|V_t|+|g(t,0,0)|, \ \ t\in[0,\tau].
\end{split}
\end{align}
In light of $(y_t,z_t)_{t\in[0,\tau]}\in H_\tau^2(\rho_\cdot-\frac{\theta}{2} {\nu}_\cdot^2;\R^{k}\times\R^{k\times d}),$ BDG's inequality yields that
$$\left(\int_0^{t\wedge\tau} e^{ \int_{0}^{s}(\rho_r-\frac{\theta}{2} {\nu}_r^2){\rm d}r}\langle \hat{y}_s,z_s{\rm d}B_s\rangle\right)_{t\geq0}$$
is a uniformly integrable martingale. In fact, in view of Theorem 1 in \citet{Ren2008BDG}, we know that
\begin{align}\label{2.06*}
\begin{split}
&\E\left[\sup_{t\in[0,\tau]}\left|\int_0^{t} e^{ \int_{0}^{s}(\rho_r-\frac{\theta}{2} {\nu}_r^2){\rm d}r}\langle \hat{y}_s,z_s{\rm d}B_s\rangle\right|\right]
\leq 2\sqrt{2}\E\left[\left(\int_0^\tau e^{2 \int_{0}^{s}(\rho_r-\frac{\theta}{2} {\nu}_r^2){\rm d}r}|z_s|^2{\rm d}s\right)^{\frac{1}{2}}\right]<+\infty.
\end{split}
\end{align}
It follows from \eqref{0.37}, \eqref{0.38} and \eqref{2.06*} that for each $t\geq0$,
\begin{align}\label{0.39}
\begin{split}
&e^{\int_{0}^{t\wedge\tau}(\rho_r-\frac{\theta}{2} {\nu}_r^2){\rm d}r}|y_{t\wedge\tau}|\\
&\ \ \leq \E\left[e^{\int_{0}^{\tau}(\rho_r-\frac{\theta}{2} {\nu}_r^2){\rm d}r}|\xi|\bigg|\F_{t\wedge\tau}\right] +\E\left[\int_{t\wedge\tau}^{\tau}e^{\int_{0}^{s}(\rho_r-\frac{\theta}{2} {\nu}_r^2){\rm d}r}\left(\nu_s|V_s|+|g(s,0,0)|\right){\rm d}s\bigg|\F_{t\wedge\tau}\right].
\end{split}
\end{align}
Multiplying $e^{\int_{0}^{t\wedge\tau}\frac{\theta}{2}\nu_r^2{\rm d}r}$ on both sides of \eqref{0.39}, we get that for each $t\geq0$,
\begin{align}\label{0.40}
\begin{split}
&e^{\int_{0}^{t\wedge\tau}\rho_r{\rm d}r}|y_{t\wedge\tau}|\\
&\ \ \leq \E\left[e^{\int_{0}^{\tau} \rho_r{\rm d}r}|\xi|\bigg|\F_{t\wedge\tau}\right]
+\E\left[\int_{t\wedge\tau}^{\tau} e^{\int_{0}^{s}(\rho_r-\frac{\theta}{2} {\nu}_r^2){\rm d}r}\left(\nu_s|V_s|+|g(s,0,0)|\right){\rm d}s~e^{\int_{0}^{t\wedge\tau}\frac{\theta}{2}\nu_r^2{\rm d}r}\bigg|\F_{t\wedge\tau}\right]\\
&\ \ \leq \E\left[e^{\int_{0}^{\tau} \rho_r{\rm d}r}|\xi|\bigg|\F_{t\wedge\tau}\right]+\E\left[\int_0^\tau e^{\int_{0}^{s} \rho_r{\rm d}r}|g(s,0,0)|{\rm d}s\bigg|\F_{t\wedge\tau}\right]\\
&\hspace{0.6cm}+\E\left[\sqrt{\int_{t\wedge\tau}^{\tau}e^{2\int_{0}^{s} \rho_r{\rm d}r}|V_s|^2{\rm d}s\int_{t\wedge\tau}^{\tau}e^{-\int_{0}^{s} \theta \nu_r^2{\rm d}r}\nu_s^2{\rm d}s~e^{\int_{0}^{t\wedge\tau} \theta\nu_r^2{\rm d}r}}\bigg|\F_{t\wedge\tau}\right]\\
&\ \ \leq \E\left[e^{\int_{0}^{\tau} \rho_r{\rm d}r}|\xi|\bigg|\F_{t\wedge\tau}\right]+\E\left[\int_0^\tau e^{\int_{0}^{s} \rho_r{\rm d}r}|g(s,0,0)|{\rm d}s\bigg|\F_{t\wedge\tau}\right]\\
&\hspace{0.6cm}+\E\left[\sqrt{\frac{1}{\theta} \int_{0}^{\tau}e^{2\int_{0}^{s} \rho_r{\rm d}r}|V_s|^2{\rm d}s}\bigg|\F_{t\wedge\tau}\right].
\end{split}
\end{align}
Then, in light of \eqref{0.40} along with the assumptions of $\xi$, $V_\cdot$, and $g(t,0,0)$, using Doob's martingale inequality on submartingales yields that
$$\E\left[\sup_{t\in[0,\tau]}\left(e^{2\int_0^t \rho_r{\rm d}r}|y_t|^2\right)\right]<+\infty.$$
Thus, $y_\cdot\in S_\tau^2(\rho_\cdot;\R^{k})$. Furthermore, in light of \ref{A:H4} and \ref{A:H5}, by a similar analysis to that in the proof of \cref{pro2.01} with $r=t=0$, we can derive that for each $n\geq1,$
\begin{align}\label{32.82}
\E\left[\int_{0}^{\tau_n}e^{2 \int_{0}^{s}\rho_r{\rm d}r}|z_s|^2{\rm d}s\right]
\leq&
\E\left[\sup_{s\in[0,\tau_n]}\left(e^{2 \int_{0}^{s}\rho_r{\rm d}r}|y_s|^2\right)\right] +\frac{1}{\theta} \E\left[\int_{0}^{\tau_n}e^{2\int_0^s \rho_r{\rm d}r}|V_s|^2{\rm d}s\right]\nonumber\\
&+2\E\left[\int_{0}^{\tau_n}e^{2\int_0^s \rho_r{\rm d}r}|y_s||g(s,0,0)|{\rm d}s\right]\nonumber\\
\leq&
2\E\left[\sup_{s\in[0,\tau]}\left(e^{2\int_0^s \rho_r{\rm d}r}|y_s|^2\right)\right]+\frac{1}{\theta} \E\left[\int_{0}^{\tau}e^{2\int_0^s \rho_r{\rm d}r}|V_s|^2{\rm d}s\right]\nonumber\\
&+\E\left[\left(\int_{0}^{\tau}e^{\int_0^s \rho_r{\rm d}r}|g(s,0,0)|{\rm d}s\right)^2\right]
\end{align}
with
$$\tau_{n}:=\inf \left\{t \geq0: \int_{0}^{t}e^{2 \int_{0}^{s}\rho_r{\rm d}r}|z_s|^2 {\rm d}s \geq n\right\}\wedge \tau.$$
Finally, by taking the limit with respect to $n$ in the last inequality \eqref{32.82}, the use of Levi's lemma yields $z_\cdot\in M_\tau^2(\rho_\cdot;\R^{k\times d})$. Hence, $(y_\cdot,z_\cdot)\in H_\tau^2(\rho_\cdot;\R^{k}\times\R^{k\times d})$ and the proof of the fourth step is then complete.\vspace{0.2cm}

{\bf Fifth Step:} In this step, under assumptions \ref{A:H1}-\ref{A:H5} with $p=2$ and $\rho_\cdot\geq\mu_\cdot+\frac{\theta}{2}\nu_\cdot^2$, we use the conclusion of the fourth step and the fixed point theorem to prove that for each $\xi\in L_\tau^2(\rho_\cdot;\R^k)$, the BSDE \eqref{BSDE1.1} admits a weighted $L^2$ solution $(y_t,z_t)_{t\in[0,\tau]}$ in the space of $H_\tau^2(\rho_\cdot;\R^{k}\times\R^{k\times d})$.

Based on the fourth step, we know that for any given $V_\cdot\in M_\tau^2(\rho_\cdot;\R^{k\times d})$, BSDE \eqref{BSDE3.1} admits a unique weighted $L^2$ solution $(y_t,z_t)_{t\in[0,\tau]}$ in the space of $S_\tau^2(\rho_\cdot;\R^{k})\times M_\tau^2(\rho_\cdot;\R^{k\times d})$. Thus, by assigning $z_\cdot$ in the above solution to be the image of $V_\cdot$ we can construct a mapping
\begin{align*}
\begin{split}
\Phi: M_\tau^2(\rho_\cdot;\R^{k\times d})&\rightarrow M_\tau^2(\rho_\cdot;\R^{k\times d})\\
V_\cdot&\rightarrow z_\cdot.
\end{split}
\end{align*}
Now, suppose that for each $i=1,2$, $V_\cdot^i\in M_\tau^2(\rho_\cdot;\R^{k\times d})$ and $z_\cdot^i:=\Phi(V_\cdot^i)$. Denote
\begin{align*}
\begin{split}
&\overline{V}_t:=V_t^1-V_t^2, \ \ \overline{y}_t:=y_t^1-y_t^2, \ \ \overline{z}_t:=z_t^1-z_t^2, \ \ \ t\in [0,\tau].
\end{split}
\end{align*}
Applying It\^{o}'s formula to $e^{2 \int_{0}^{t}\rho_r{\rm d}r}|\overline{y}_t|^2$ and using \ref{A:H4} and \ref{A:H5}, we deduce that
\begin{align}\label{33.9}
\begin{split}
&|\overline{y}_0|^2+\int_0^{\tau} e^{2\int_0^s \rho_r{\rm d}r}|\overline{z}_s|^2{\rm d}s+2\int_0^{\tau} e^{2 \int_0^s \rho_r{\rm d}r}\rho_s|\overline{y}_s|^2{\rm d}s\\
&\ \ =2\int_0^{\tau} e^{2\int_0^s \rho_r{\rm d}r}\langle \overline{y}_s, ~g(s,y_s^1,V_s^1)-g(s,y_s^2,V_s^2)\rangle{\rm d}s-2\int_0^{\tau} e^{2\int_0^s \rho_r{\rm d}r}\langle \overline{y}_s,\overline{z}_s{\rm d}B_s\rangle\\
&\ \ \leq \int_0^{\tau} e^{2\int_0^s \rho_r{\rm d}r}\left(2\mu_s|\overline{y}_s|^2+2\nu_s|\overline{y}_s||\overline{V}_s|
\right){\rm d}s-2\int_0^{\tau} e^{2\int_0^s \rho_r{\rm d}r}\langle \overline{y}_s,\overline{z}_s{\rm d}B_s\rangle.
\end{split}
\end{align}
It follows from $2ab\leq \theta a^2+\frac{1}{\theta} b^2$ that
\begin{align}\label{3.9}
2\nu_s|\overline{y}_s||\overline{V}_s|\leq\theta \nu^2_s|\overline{y}_s|^2+\frac{1}{\theta}|\overline{V}_s|^2,\ \ s\in[0,\tau].
\end{align}
In light of \eqref{33.9}, \eqref{3.9} and the inequality $2\rho_s-2\mu_s-\theta{\nu_s}^2\geq0,\ s\in [0,\tau]$, we have
\begin{align}\label{4.21}
\int_0^{\tau} e^{2\int_0^s \rho_r{\rm d}r}|\overline{z}_s|^2{\rm d}s
\leq \frac{1}{\theta}\int_0^{\tau} e^{2\int_0^s \rho_r{\rm d}r}|\overline{V}_s|^2{\rm d}s-2\int_0^{\tau} e^{2\int_0^s \rho_r{\rm d}r}\langle \overline{y}_s,\overline{z}_s{\rm d}B_s\rangle.
\end{align}
Furthermore, it follows from BDG's inequality that
$$
\left(\int_0^{t\wedge\tau} e^{2\int_0^s \rho_r{\rm d}r}\langle \overline{y}_s,\overline{z}_s{\rm d}B_s\rangle\right)_{t\geq0}
$$
is a uniformly integrable martingale. Indeed, by Theorem 1 in \citet{Ren2008BDG} we know that
\begin{align*}
\begin{split}
&2\E\left[\sup_{t\in[0,\tau]}\left|\int_0^{t} e^{2\int_0^s \rho_r{\rm d}r}\langle \overline{y}_s,\overline{z}_s{\rm d}B_s\rangle\right|\right]
\leq 4\sqrt{2}\E\left[\left(\int_0^\tau e^{4 \int_{0}^{s}\rho_r{\rm d}r}|\overline{y}_s|^2|\overline{z}_s|^2{\rm d}s\right)^{\frac{1}{2}}\right]\\
&\ \ \leq  \frac{1}{2}\E\left[\sup_{s\in[0,\tau]}\left(e^{2 \int_{0}^{s}\rho_r{\rm d}r}|\overline{y}_s|^2\right)\right]+64\E\left[\int_0^\tau e^{2\int_0^s \rho_r{\rm d}r}|\overline{z}_s|^2{\rm d}s\right]<+\infty.
\end{split}
\end{align*}
Then, by taking the mathematical expectation on both sides of \eqref{4.21} we obtain
\begin{align*}\label{339}
\begin{split}
\E\left[\int_0^{\tau} e^{2\int_0^s \rho_r{\rm d}r}|\overline{z}_s|^2{\rm d}s\right]
\leq  \frac{1}{\theta}\E\left[ \int_0^{\tau} e^{2\int_0^s \rho_r{\rm d}r}|\overline{V}_s|^2{\rm d}s\right].
\end{split}
\end{align*}
Therefore, $\Phi$ is a strict contraction mapping on $M_\tau^2(\rho_\cdot;\R^{k\times d})$, and then it has a unique fixed point $z_\cdot$ i.e. $\Phi(z_\cdot)=z_\cdot$, which indicates that, according to the definition of mapping $\Phi$ and the assertion of the fourth step, there must exist $y_\cdot\in S_\tau^2(\rho_\cdot;\R^{k})$ such that $(y_\cdot,z_\cdot)$ is the weighted $L^2$ solution of BSDE \eqref{BSDE1.1} in $H_\tau^2(\rho_\cdot;\R^{k}\times\R^{k\times d})$. The proof of the existence part of \cref{thm:3.1} for the case of $p=2$ is completed.\vspace{0.2cm}

{\bf Sixth Step:} In this step, we use the assertion of the fifth step to finally complete the proof of the existence part of \cref{thm:3.1} for the general case.

Now, we assume that $\xi\in L_\tau^2(\rho_\cdot;\R^k)$ and the generator $g$ satisfies assumptions \ref{A:H1}-\ref{A:H5} with
\begin{equation}\label{3.401}
\rho_t\geq \mu_t+\frac{\theta}{2[1\wedge(p-1)]}\nu_t^2, \ \ t\in[0,\tau].
\end{equation}
Recalling the definition of $q_r(x):=\frac{xr}{|x|\vee r}$ for each $ x\in\R^k$ and $r>0$ in the third step. Define
\begin{equation}\label{3.40}
\gamma_t:=e^{-\int_{0}^{t}\rho_s{\rm d}s}, \ \ t\in[0,\tau].
\end{equation}
For each $n\geq1$ and $t\in[0,\tau]$, let
\begin{equation}\label{340}
\xi_n:=q_{n\gamma_\tau}(\xi) \ \ \ \text{and} \ \ \ \overline{g}^n(t,y,z):=g(t,y,z)-g(t,0,0)+q_{ne^{-t}\gamma_t}(g(t,0,0)).
\end{equation}
Then, we have for each $n\geq1$,
\begin{equation}\label{341}
|\xi_n|\leq n\gamma_\tau \ \ \ \text{and} \ \ \ |\overline{g}^n(t,0,0)|\leq ne^{-t}\gamma_t, \ \ \ t\in[0,\tau].
\end{equation}
Note from \eqref{3.401} that
\begin{equation}\label{3.41}
\rho_t\geq \mu_t+\frac{\theta}{2}\nu_t^2, \ \ t\in[0,\tau].
\end{equation}
In light of \eqref{3.40}, \eqref{341} and \eqref{3.41} along with assumptions \ref{A:H2}-\ref{A:H5} and the integrability condition of $\xi$, we know that for each $n\geq1$, $\xi_n$ and $\overline{g}^n$ satisfy all assumptions in the fifth step, and then BSDE $(\xi_n,\tau,\overline{g}^n)$ has a unique weighted $L^2$ solution in the space of $H_\tau^2(\rho_\cdot;\R^{k}\times\R^{k\times d})$, denoted by $(y_t^n,z_t^n)_{t\in[0,\tau]}$. Furthermore, by \ref{A:H4} and \ref{A:H5} we deduce that for each $n\geq1$ and $(y,z)\in \R^k\times\R^{k\times d}$,
\begin{align}\label{342}
\begin{split}
\left<\hat{y},~\overline{g}^n(t,y,z)\right>&=\left<\hat{y},~g(t,y,z)-g(t,0,0)
+q_{ne^{-t}\gamma_t}(g(t,0,0))\right>\\
&=\left<\hat{y},~g(t,y,z)-g(t,0,z)+g(t,0,z)-g(t,0,0)+q_{ne^{-t}\gamma_t}
(g(t,0,0))\right>\\
&\leq \mu_t|y|+\nu_t|z|+ne^{-t}\gamma_t, \ \ t\in[0,\tau],
\end{split}
\end{align}
which along with \eqref{3.41} implies that $\overline{g}^n$ satisfies assumption \ref{A:A} with $\bar{p}:=2$, $\bar\mu_t:={\mu}_t$, $\bar\nu_t:=\nu_t$, ${\bar\rho_t}:=\rho_t$ and $f_t:=ne^{-t}\gamma_t$. It then follows from \eqref{3.40}, \eqref{341} and \cref{pro2.02} with $r=t$ that there exists a constant $C_\theta>0$ depending only on $\theta$ such that for each $t\geq0$ and $n\geq1$,
\begin{align*}
\begin{split}
e^{2\int_{0}^{t\wedge\tau}\rho_r{\rm d}r}|y_{t\wedge\tau}^n|^2
&\leq
C_\theta\left(\E\left[e^{2\int_{0}^{\tau}\rho_r{\rm d}r}|\xi_n|^2\bigg|\F_{t\wedge\tau}\right]+\E\left[\left(\int_{0}^{\tau}
e^{\int_{0}^{s}\rho_r{\rm d}r}ne^{-s}\gamma_s{\rm d}s\right)^2\bigg|\F_{t\wedge\tau}\right]\right)\\
& \leq 2n^2C_\theta,
\end{split}
\end{align*}
which indicates that $(e^{\int_{0}^{t}\rho_r{\rm d}r}y_t^n)_{t\in[0,\tau]}$ is a bounded process and $y_\cdot^n$ belongs to $S_\tau^p(\rho_\cdot;\R^k)$ for each $n\geq1$. On the other hand, in light of \eqref{3.401}, \eqref{3.40} and \eqref{342}, we know that for each $n\geq1$, $\overline{g}^n$ also satisfies assumption \ref{A:A} with $\bar{p}:=p$, $\bar\mu_t:={\mu}_t$, $\bar\nu_t:=\nu_t$, ${\bar\rho_t}:=\rho_t$ and $f_t:=ne^{-t}\gamma_t$. It then follows from \cref{pro2.01} that $z_\cdot^n$ belongs to $M_\tau^p(\rho_\cdot;\R^{k\times d})$ for each $n\geq1$.

In the sequel, we further prove that  $\{(y_\cdot^n,z_\cdot^n)\}^{+\infty}_{n=1}$ is a Cauchy sequence in $H_\tau^p(\rho_\cdot;\R^{k}\times\R^{k\times d})$. For each pair of integers $n, i\geq1$, let
$$\hat{\xi}^{n,i}:=\xi_{n+i}-\xi_n, \  \ \hat y_\cdot^{n,i}:=y_\cdot^{n+i}-y_\cdot^n, \ \ \hat z_\cdot^{n,i}:=z_\cdot^{n+i}-z_\cdot^n.$$
Then, we have
\begin{align*}
\hat y_t^{n,i}=\hat{\xi}^{n,i}+\int_t^\tau\hat{g}^{n,i}(s,\hat y_s^{n,i},\hat z_s^{n,i}){\rm d}s-\int_t^\tau\hat z_s^{n,i}{\rm d}B_s, \ \ t\in[0,\tau],
\end{align*}
where for each $(y,z)\in \R^k\times\R^{k\times d}$, $$\hat{g}^{n,i}(t,y,z):=\overline{g}^{n+i}(t,y+y_t^n,z+z_t^n)-
\overline{g}^{n}(t,y_t^n,z_t^n), \ \ t\in[0,\tau].$$
It follows from \eqref{340} and assumptions \ref{A:H4} and \ref{A:H5} that for each $(y,z)\in\R^k\times\R^{k\times d}$,
\begin{align*}
\begin{split}
\left<\hat{y},~\hat{g}^{n,i}(t,y,z)\right>
&=\left<\hat{y},
~\overline{g}^{n+i}(t,y+y_t^n,z+z_t^n)-\overline{g}^{n}(t,y_t^n,z_t^n)\right>\\
&\leq \mu_t|y|+\nu_t|z|+|q_{(n+i)e^{-t}\gamma_t}(g(t,0,0))
-q_{ne^{-t}\gamma_t}(g(t,0,0))|\\
&\leq
\mu_t|y|+\nu_t|z|+|g(t,0,0)|{\bf 1}_{|g(t,0,0)|>ne^{-t}\gamma_t}, \ \ t\in[0,\tau],
\end{split}
\end{align*}
which along with \eqref{3.401} implies that $\hat{g}^{n,i}$ satisfies assumption \ref{A:A} with $\bar{p}:=p$, $\bar\mu_t:={\mu}_t$, $\bar\nu_t:=\nu_t$, ${\bar\rho_t}:=\rho_t$ and $f_t:=|g(t,0,0)|{\bf 1}_{|g(t,0,0)|>ne^{-t}\gamma_t}$.
According to \cref{pro2.1} with $r=t=0$, we deduce that there exists a constant $C_{p,\theta}>0$ depending only on $p$ and $\theta$ such that for each $n,i\geq1$,
\begin{align}\label{345}
\begin{split}
&\E\left[\sup_{s\in[0,\tau]}\left(e^{p \int_0^s\rho_r{\rm d}r}|\hat{y}_s^{n,i}|^p\right)\right]+\E\left[\left(\int_{0}^{\tau}e^{2 \int_0^s\rho_r{\rm d}r}|\hat{z}_s^{n,i}|^2{\rm d}s\right)^{\frac{p}{2}}\right]\\
&\ \ \ \leq C_{p,\theta}\E\left[e^{p \int_0^\tau \rho_r{\rm d}r}|\xi|^p{\bf 1}_{|\xi|>n\gamma_\tau}+\left(\int_0^\tau e^{ \int_0^s{\rho_r}{\rm d}r}|g(s,0,0)|{\bf 1}_{|g(s,0,0)|>ne^{-s}\gamma_s}{\rm d}s\right)^p\right].
\end{split}
\end{align}
Thus, in light of \ref{A:H1} and Lebesgue's dominated convergence theorem, by taking first the supremum with respect to $i$ and then the upper limit with respect to $n$ on both sides of \eqref{345}, we can conclude that $\{(y_\cdot^n,z_\cdot^n)\}^{+\infty}_{n=1}$ is a Cauchy sequence in $H_\tau^p(\rho_\cdot;\R^{k}\times\R^{k\times d})$.

Finally, we denote by $(y_t,z_t)_{t\in[0,\tau]}$ the limit of the Cauchy sequence $\{(y_t^n,z_t^n)_{t\in[0,\tau]}\}_{n=1}^\infty$ in the space of $H_\tau^p(\rho_\cdot;\R^{k}\times\R^{k\times d})$ and show that $(y_\cdot,z_\cdot)$ solves BSDE \eqref{BSDE1.1} by passing to the limit under ucp for BSDE $(\xi_n,\tau,\overline{g}^n)$. First of all, similar to the process of taking the limit at the end of the third step, passing to a subsequence if needed, still denoted by $y^n_\cdot$ and $z^n_\cdot$, we can verify that \eqref{2.02}, \eqref{2.011} and \eqref{0.22} are still valid because of $\essinf\limits_{t\in[0,\tau]}\alpha_t>0$ and $\int_0^\tau |\rho_r|{\rm d}r<+\infty$. Then, it suffices to prove that
\begin{align}\label{348}
\int_0^\tau \overline{g}^n(t,y_t^n,z_t^n){\rm d}t\rightarrow \int_0^\tau g(t,y_t,z_t){\rm d}t, \ \ \ n\rightarrow \infty.
\end{align}
In fact, it follows from \ref{A:H2}, \ref{A:H5} and \eqref{340} that for each $t\in [0,\tau],$ as $n\rightarrow \infty,$ we have
\begin{align}\label{349}
\overline{g}^n(t,y_t^n,z_t^n)\rightarrow g(t,y_t,z_t).
\end{align}
In light of \eqref{340}, \ref{A:H5} and \eqref{2.011}, we get that for each $n\geq1$ and $t\in [0,\tau],$
\begin{align}\label{350}
\begin{split}
|\overline{g}^n(t,y_t^n,z_t^n)|&\leq |g(t,y_t^n,z_t^n)-g(t,0,0)|+|g(t,0,0)|\\
&\leq |g(t,y_t^n,z_t^n)-g(t,y_t^n,0)+g(t,y_t^n,0)-g(t,0,0)|+|g(t,0,0)|\\
&\leq \psi(t,M_3\alpha_t)+\nu_t\sup_{n\geq1}|z_t^n|+|g(t,0,0)|.
\end{split}
\end{align}
By virtue of $\essinf\limits_{t\in[0,\tau]}\alpha_t>0$ and $\int_0^\tau |\rho_r|{\rm d}r<+\infty$, along with \ref{A:H1} and \ref{A:H3}, we can get
\begin{align}\label{346}
\begin{split}
\int_0^\tau \psi(t,M_3\alpha_t){\rm d}t&=\int_0^\tau \frac{1}{\alpha_t} \alpha_t\psi(t,M_3\alpha_t){\rm d}t\\
&\leq \frac{1}{\essinf\limits_{t\in[0,\tau]}\alpha_t} \int_0^\tau \alpha_t\psi(t,M_3\alpha_t){\rm d}t<+\infty
\end{split}
\end{align}
and
\begin{align}\label{347}
\begin{split}
\int_0^\tau |g(t,0,0)|{\rm d}t&=\int_0^\tau e^{-\int_0^t\rho_r{\rm d}r}~e^{\int_0^t\rho_r{\rm d}r}|g(t,0,0)|{\rm d}t\\
&\leq e^{\int_0^\tau|\rho_r|{\rm d}t} \int_0^\tau e^{\int_0^t\rho_r{\rm d}r}|g(t,0,0)|{\rm d}t<+\infty.
\end{split}
\end{align}
Since $z^n_\cdot$ converges in $M_\tau^p(\rho_\cdot;\R^{k\times d})$ to $z_\cdot$, we can assume, choosing a subsequence if necessary, still denoted by $z^n_\cdot$, that
$$
\left\{\E\left[\left(\int_0^\tau e^{2 \int_0^t \rho_r{\rm d}r}|z^n_t-z_t|^2{\rm d}t\right)^{\frac{p}{2}}\right]\right\}^\frac{1}{p}\leq \frac{1}{2^n}.
$$
Note that
\begin{align*}
\begin{split}
\E\left[\left(\int_0^\tau e^{2 \int_0^t \rho_r{\rm d}r}\sup_{n\geq 1}|z^n_t-z_t|^2{\rm d}t\right)^{\frac{p}{2}}\right]\leq&
\E\left[\left(\int_0^\tau e^{2 \int_0^t \rho_r{\rm d}r}\sum_{n=1}^{+\infty}|z^n_t-z_t|^2{\rm d}t\right)^{\frac{p}{2}}\right]\\
\leq& \E\left[\left(\int_0^\tau e^{2 \int_0^t \rho_r{\rm d}r}\left(\sum_{n=1}^{+\infty}|z^n_t-z_t|\right)^2{\rm d}t\right)^{\frac{p}{2}}\right]\\
=& \bigg\|\sum_{n=1}^{+\infty}|z^n_t-z_t|\bigg\|_{M_\tau^p}^p\leq
\left(\sum_{n=1}^{+\infty}\|z^n_t-z_t\|_{M_\tau^p}\right)^p\\
\leq& \left(\sum_{n=1}^{+\infty}\frac{1}{2^n}\right)^p<+\infty.
\end{split}
\end{align*}
We have
\begin{align*}
\begin{split}
\E\left[\left(\int_0^\tau e^{2 \int_0^t \rho_r{\rm d}r} \sup_{n\geq 1} |z^n_t|^2 {\rm d}t\right)^{\frac{p}{2}}\right]\leq &
2^p \E\left[\left(\int_0^\tau e^{2 \int_0^t \rho_r{\rm d}r}\sup_{n\geq 1}|z^n_t-z_t|^2{\rm d}t\right)^{\frac{p}{2}}\right]\\
+& 2^p \E\left[\left(\int_0^\tau e^{2 \int_0^t \rho_r{\rm d}r}|z_t|^2{\rm d}t\right)^{\frac{p}{2}}\right]<+\infty.
\end{split}
\end{align*}
It then follows from  H${\rm \ddot{o}}$lder's inequality that
\begin{align}\label{351}
\begin{split}
\int_0^\tau \nu_t\sup_{n\geq1}|z_t^n|{\rm d}t&=\int_0^\tau e^{-\int_0^t\rho_r{\rm d}r}\nu_t~ e^{\int_0^t\rho_r{\rm d}r}\sup_{n\geq 1}|z_t^n|~{\rm d}t\\
&\leq \left(\int_0^\tau e^{-2\int_0^t\rho_r{\rm d}r}\nu_t^2{\rm d}t\right)^{\frac{1}{2}} \left(\int_0^\tau  e^{2\int_0^t\rho_r{\rm d}r}\sup_{n\geq 1}|z_t^n|^2 {\rm d}t\right)^{\frac{1}{2}}\\
&\leq e^{\int_0^\tau|\rho_r|{\rm d}r}\left(\int_0^\tau\nu_t^2{\rm d}t\right)^{\frac{1}{2}}\left(\int_0^\tau e^{2\int_0^t\rho_r{\rm d}r}\sup_{n\geq 1} |z_t^n|^2 {\rm d}t\right)^{\frac{1}{2}}<+\infty.
\end{split}
\end{align}
Thus, in light of \eqref{349}-\eqref{351}, by Lebesgue's dominated convergence theorem we obtain \eqref{348}. The proof of the existence part of \cref{thm:3.1} is then completed.

\section{Existence of viscosity solutions to parabolic and elliptic PDEs}
\setcounter{equation}{0}

In this section, we will give applications of our results regarding BSDEs to the existence of viscosity solutions to parabolic and elliptic PDEs. To be more precise, under some more general assumptions than those employed in subsection 3.4 of \citet{Li2025}, three nonlinear Feynman-Kac formulas are obtained, as outlined in \cref{thm:5.1}, \cref{rmk:5.2} and \cref{thm:5.4}. For the sake of convenience, we always assume that $n$ is a positive integer, and that $p>1$, $\theta>1$ and $K_i>0~(i=1,2,3)$ are five constants.

\subsection{Probabilistic interpretation for parabolic PDEs\vspace{0.2cm}}

Assume that the terminal time $\tau\equiv T$ is a finite constant. For the notion of viscosity
solution to make sense, we consider the following parabolic PDE:
\begin{align}\label{PDE'}
\left\{\begin{aligned}
&\frac{\partial u_i(t,x)}{\partial t}+\mathcal{L}u_i(t,x)+g_i(t,x,u(t,x),(\nabla_x u_i\sigma)(t,x))=0,\\
&(t,x)\in \T\times\R^n, \ i=1,\cdots,k; \ \ u(T,x)=h(x),\ x\in\R^n,
\end{aligned}\right.
\end{align}
where $\mathcal{L}$ is the infinitesimal generator of the diffusion solution $X_\cdot^{t,x}$ to the following SDE
\begin{align}\label{SDE}
X_s^{t,x}=x+\int_t^s b(r,X_r^{t,x}){\rm d}r+\int_t^s \sigma(r,X_r^{t,x}){\rm d}B_r,\ \ s\in [t,T].
\end{align}
The nonlinear Feynman-Kac formula consists in proving that the function defined by
\begin{align}\label{u}
\forall(t,x)\in \T\times\R^n, \ \ \ u(t,x):=Y_t^{t,x},
\end{align}
where, for each $(t_0,x_0)\in [0,T]\times\R^n,$ $(Y_\cdot^{t_0,x_0},Z_\cdot^{t_0,x_0})$ represents the adapted solution to the BSDE
\begin{align}\label{BSDE-1}
Y_t=h(X_T^{t_0,x_0})+\int_t^T g(s,X_s^{t_0,x_0},Y_s,Z_s){\rm d}s-\int_t^T Z_s{\rm d}B_s,\ \ t\in[0,T],
\end{align}
is a solution, at least a viscosity solution, to PDE \eqref{PDE'}.

The objective of this subsection is to derive the above probabilistic representation for a viscosity solution to PDE \eqref{PDE'} when the functions $g$ and $h$ satisfy some general assumptions. Let us first recall the following definition of a viscosity solution to PDE \eqref{PDE'}.

\begin{dfn}
A continuous function $u(t,x):\T\times \R^n\rightarrow \R^k$ with $u(T,\cdot)=h(\cdot)$ is a viscosity sub-solution (super-solution) of PDE \eqref{PDE'} if for any $1\leq i\leq k$ and any function $\phi\in C^{1,2}(\T\times \R^n;\R)$, it holds that
$$
\frac{\partial \phi(t_0,x_0)}{\partial t}+\mathcal{L}\phi(t_0,x_0)+g_i(t_0,x_0,u(t_0,x_0),
(\nabla_x\phi\sigma)(t_0,x_0))\geq 0\  (\leq 0),
$$
as long as the function $u_i-\phi$ has a local maximum (minimum) at the point $(t_0,x_0)\in [0,T)\times \R^n$. A viscosity solution of PDE \eqref{PDE'} is both a viscosity sub-solution and viscosity super-solution.
\end{dfn}

Now, let us introduce our assumption concerning the coefficients of SDE \eqref{SDE}, which guarantee that for each $(t,x)\in \T\times \R^n$, SDE \eqref{SDE} admits a unique solution $X_\cdot^{t,x}\in S^{\bar p}_T(0;\R^n)$ for any $\bar p\geq1$.

\begin{enumerate}
\renewcommand{\theenumi}{(A1)}
\renewcommand{\labelenumi}{\theenumi}
\item\label{A:A1} both $b(t,x):[0,T]\times\R^n\rightarrow\R^n$ and $\sigma(t,x):[0,T]\times \R^n\rightarrow\R^{n\times d}$ are jointly continuous functions and satisfy that for each $(t,x_1,x_2)\in \T\times\R^n\times\R^n$,
    $$
    |b(t,x_1)-b(t,x_2)|+|\sigma(t,x_1)-\sigma(t,x_2)|\leq K_1|x_1-x_2| \ \ \ \text{and} \ \ \ |b(t,0)|+|\sigma(t,0)|\leq K_1.
    $$
\end{enumerate}

Next, we give our assumptions on the terminal value and the generator of BSDE \eqref{BSDE-1}.

\begin{enumerate}
\renewcommand{\theenumi}{(A2)}
\renewcommand{\labelenumi}{\theenumi}
\item\label{A:A2} both $h(x):\R^n \rightarrow \R^k$ and $g(t,x,y,z):[0,T]\times \R^n\times\R^k\times \R^{k\times d} \rightarrow \R^k$ are (jointly) continuous functions, the $i$th component $g_i$ of $g$, for each $1\leq i\leq k$, varies only when $(t,x,z^i)$ changes, where $z^i$ represents the $i$th row of the matrix $z$, and for some $\bar{q}>0$ and each $(t,x,y,y_1,y_2,z,z_1,z_2)\in [0,T]\times \R^n\times\R^{3k}\times \R^{3(k\times d)}$, we have
\begin{align}\label{A2'}
\begin{array}{c}
\left\langle y_1-y_2, g(t,x,y_1,z)-g(t,x,y_2,z)\right\rangle \leq m(t,x)|y_1-y_2|^2,\\
|g(t,x,y,z_1)-g(t,x,y,z_2)|\leq n(t,x)|z_1-z_2|,\\
|h(x)|+|g(t,x,y,0)|\leq K_1(1+|x|^{\bar{q}})(1+\varphi(|y|)),
\end{array}
\end{align}
where $\varphi(\cdot):\R_+\rightarrow \R_+$ is a continuous nondecreasing function, and both
$ m(t,x):[0,T]\times\R^n\rightarrow\R$ and $n(t,x):[0,T]\times\R^n\rightarrow\R_+$ are jointly continuous functions satisfying
\begin{align}\label{6.2}
\RE (t,x)\in \T\times \R^n,\ \ m(t,x)+\frac{\theta}{2[1\wedge(p-1)]}n^2(t,x)\leq K_1.
\end{align}
\end{enumerate}

Under the above assumptions \ref{A:A1}-\ref{A:A2}, it can be easily verified that for each $(t,x)\in \T\times\R^n$, $h(X_T^{t,x})\in L_T^p(\rho^{t,x}_\cdot;\R^k)$ and $g(r,X_r^{t,x},y,z){\bf 1}_{t\leq r\leq T}$ satisfies assumptions \ref{A:H1}, \ref{A:H2}, \ref{A:H4} and \ref{A:H5} with
$$
\mu_r^{t,x}:=m(r,X_r^{t,x}){\bf 1}_{t\leq r\leq T},\ \ \ \nu_r^{t,x}:=n(r,X_r^{t,x}){\bf 1}_{t\leq r\leq T} \ \ \  {\rm and} \ \ \ \rho_r^{t,x}:=\mu_r^{t,x}+\frac{\theta}{2[1\wedge(p-1)]}|\nu_r^{t,x}|^2,
$$
along with assumption \ref{A:H3d} with ${\tilde \mu}_r:=K_1(1+|X_r^{t,x}|^{\bar{q}})$ and $\hat{\mu}_r\equiv1.$
It then follows from \cref{thm:3.1} and \cref{rmk:3.4} that for each $(t,x)\in \T\times\R^n$, BSDE \eqref{BSDE-1} admits a unique weighted $L^p$ solution $(Y_s^{t,x},Z_s^{t,x})_{s\in [t,T]}$ in the space of $H_\tau^p(\rho_\cdot^{t,x};\R^{k}\times\R^{k\times d})$. In addition, by virtue of the Markov property of SDE \eqref{SDE} and BSDE \eqref{BSDE-1} we can conclude that $Y_t^{t,x}$ is deterministic for each $(t,x)\in \T\times\R^n$.\vspace{0.2cm}

Now we can state and prove the main result of this subsection with the help of \cref{pro2.1}, \cref{thm:3.1}, \cref{thm:3.2} and Girsanov's theorem.
\begin{thm}\label{thm:5.1}
Under the above assumptions \ref{A:A1}-\ref{A:A2}, the function $u(\cdot,\cdot)$ defined by \eqref{u} is continuous on $[0,T]\times \R^n$, and it is a viscosity solution of PDE \eqref{PDE'}. Moreover, there exist two constants $q>0$ and $C>0$ depending only on $(p,\theta,T,K_1)$ such that
\begin{align}\label{6.4}
\forall (t,x)\in \T\times\R^n, \ \ |u(t,x)|\leq C(1+|x|^q).
\end{align}
\end{thm}
\begin{proof}
By virtue of \cref{thm:3.2} and the continuity of $X_\cdot^{t,x}$ with respect to $(t,x)$ we can derive that $u(t,x):=Y_t^{t,x}$ is a continuous function on $[0,T]\times \R^n$. The desired assertion \eqref{6.4} can be easily derived from \cref{pro2.1} along with the assumptions on functions $h$ and $g$. We omit its proof here.

Next, we prove that $u(t,x)$ is a viscosity sub-solution. By an identical way it can be shown that $u(t,x)$ is also a viscosity super-solution. Let us take any $i=1,\cdots,k$, and function $\phi\in C^{1,2}(\T\times \R^n;\R)$ such that $u_i-\phi$ attains a local maximum at $(t_0,x_0)\in [0,T)\times\R^n$ and without loss of generality,
$$\phi(t_0,x_0)=u_i(t_0,x_0).$$
In order to show that $u(t,x)$ is a viscosity sub-solution, it suffices to prove that
$$\frac{\partial \phi(t_0,x_0)}{\partial t}+\mathcal{L}\phi(t_0,x_0)+g_i(t_0,x_0,u(t_0,x_0),(\nabla_x\phi\sigma)(t_0,x_0))\geq0.$$
In fact, if not, it is known from continuity that there exist $\gamma>0$ and $0<\delta\leq T-t_0$, when $t_0\leq t\leq t_0+\delta$ and $|y-x_0|\leq\delta$, such that
\begin{align}\label{6.5}
\left\{\begin{aligned}
&u_i(t,y)\leq \phi(t,y);\\
&\frac{\partial \phi(t,y)}{\partial t}+\mathcal{L}\phi(t,y)+g_i(t,y,u(t,y),(\nabla_x\phi\sigma)(t,y))\leq -\gamma.
\end{aligned}\right.
\end{align}
Define the following stopping time:
$$\vartheta:=\inf\{u\geq t_0:|X_u^{t_0,x_0}-x_0|\geq \delta\}\wedge(t_0+\delta)>t_0.$$
By virtue of the Markov property and BSDE \eqref{BSDE-1}, we deduce that $(\bar {Y}_t,\bar {Z}_t)_{t\in[t_0,\vartheta]}
:=((Y_{t}^{t_0,x_0})^i,(Z_t^{t_0,x_0})^i)_{t\in[t_0,\vartheta]}$ is an adapted solution of the following one-dimensional BSDE:
\begin{align}\label{BSDE1}
\bar{Y}_t=u_i(\vartheta,X_{\vartheta}^{t_0,x_0})+\int_t^{\vartheta} g_i(s,X_s^{t_0,x_0},u(s,X_s^{t_0,x_0}),\bar {Z}_s){\rm d}s-\int_t^{\vartheta} \bar{Z}_s{\rm d}B_s, \ t\in[t_0,\vartheta].
\end{align}
On the other hand, by It\^{o}'s formula we derive that $(\widetilde{Y}_t,\widetilde{Z}_t)_{t\in[t_0,\vartheta]}:=(\phi(t,X_t^{t_0,x_0}),
(\nabla_x\phi\sigma)(t,X_t^{t_0,x_0}))_{t\in[t_0,\vartheta]}$ is an adapted solution for the following BSDE:
\begin{align}\label{BSDE2}
\widetilde{Y}_t=\phi(\vartheta,X_{\vartheta}^{t_0,x_0})-\int_t^{\vartheta} \left\{\frac{\partial\phi}{\partial t}+\mathcal{L}\phi\right\}(s,X_s^{t_0,x_0})){\rm d}s-\int_t^{\vartheta} \widetilde{Z}_s{\rm d}B_s, \ \ t\in[t_0,\vartheta].
\end{align}
Combining \eqref{6.5}, \eqref{BSDE1} and \eqref{BSDE2} along with the definition of the stopping time $\vartheta$ and the Lipschitz continuity for $g$ in $z$, we obtain
\begin{align}\label{eq:5.21}
\begin{split}
\widetilde{Y}_{t_0}-\bar {Y}_{t_0}=&\ \phi(\vartheta,X_{\vartheta}^{t_0,x_0})
-u_i(\vartheta,X_{\vartheta}^{t_0,x_0})-\int_{t_0}^{\vartheta} (\widetilde{Z}_s-\bar {Z}_s){\rm d}B_s\\
&\ -\int_{t_0}^{\vartheta}\left(\left\{\frac{\partial\phi}{\partial t}+\mathcal{L}\phi\right\}(s,X_s^{t_0,x_0})+g_i(s,X_s^{t_0,x_0},
u(s,X_s^{t_0,x_0}),\widetilde{Z}_s)\right){\rm d}s\\
&\ +\int_{t_0}^{\vartheta} \left( g_i(s,X_s^{t_0,x_0},u(s,X_s^{t_0,x_0}),\widetilde{Z}_s)
-g_i(s,X_s^{t_0,x_0},u(s,X_s^{t_0,x_0}),\bar {Z}_s)\right){\rm d}s\\
\geq &\ \gamma(\vartheta-t_0) -\int_{t_0}^{\vartheta} n(s,X_s^{t_0,x_0})|\widetilde{Z}_s-\bar {Z}_s|{\rm d}s-\int_{t_0}^{\vartheta}(\widetilde{Z}_s-\bar {Z}_s){\rm d}B_s.
\end{split}
\end{align}
In light of the definition of the stopping time $\vartheta$ and the continuity of the function $n(t,x)$, according to Girsanov's theorem we know that the process
$$
B^{\mathbb{Q}}_s:=B_s+\int_{t_0}^{s\wedge\vartheta\vee t_0} n(r,X_r^{t_0,x_0})\frac{(\widetilde{Z}_r-\bar Z_r)^*}
{|\widetilde{Z}_r-\bar Z_r|}{\bf 1}_{|\widetilde{Z}_r-\bar Z_r|\neq 0} {\rm d}r,\ \ s\in [0,T]
$$
is a standard $d$-dimensional Brownian motion under a probability measure $\mathbb{Q}$ equivalent to $\mathbb{P}$ defined by
$$
\frac{{\rm d}\mathbb{Q}}{{\rm d}\mathbb{P}}:=\exp\left(-\int_{t_0}^\vartheta n(r,X_r^{t_0,x_0})
\frac{\widetilde{Z}_r-\bar {Z}_r}
{|\widetilde{Z}_r-\bar {Z}_r|}{\bf 1}_{|\widetilde{Z}_r-\bar {Z}_r|\neq 0} {\rm d}B_r-\frac{1}{2}\int_{t_0}^\vartheta n^2(r,X_r^{t_0,x_0}){\bf 1}_{|\widetilde{Z}_r-\bar {Z}_r|\neq 0}{\rm d}r\right).
$$
It then follows from \eqref{eq:5.21} that
$$
\widetilde{Y}_{t_0}-\bar Y_{t_0}\geq \gamma(\vartheta-t_0) -\int_{t_0}^{\vartheta}(\widetilde{Z}_s-\bar Z_s){\rm d}B^{\mathbb{Q}}_s.
$$
By taking the conditional mathematical expectation under $\mathbb{Q}$ with respect to $\F_{t_0}$ in the last inequality, we have $\widetilde{Y}_{t_0}>\bar Y_{t_0}$, i.e., $\phi(t_0,x_0)>u_i(t_0,x_0)$, which yields a contradiction. The proof is then complete.
\end{proof}

\begin{rmk}\label{rmk:5.2}
By the classical results on SDEs, if the condition of $|\sigma(t,0)|\leq K_1$ for each $t\in\T$ appearing in assumption \ref{A:A1} is strengthened to
\begin{align}\label{5.12-1}
\RE (t,x)\in \T\times \R^n,\ \ \ |\sigma(t,x)|\leq K_1,
\end{align}
then there exist two positive constants $C_1$ and $C_2$ depending only on $(K_1,T)$ such that the unique adapted solution $(X_s^{t,x})_{s\in [t,T]}$ of SDE \eqref{SDE} satisfies
\begin{align*}
\RE (t,x)\in \T\times \R^n,\ \ \ \E\left[\exp\left(C_1\sup\limits_{t\leq s\leq T}|X_s^{t,x}|^2\right)\right]\leq C_2\exp\left(C_2|x|^2\right).
\end{align*}
Thus, by virtue of \cref{thm:3.1}, \cref{thm:3.2} and \cref{pro2.1}, using an identical analysis as above yields that when \ref{A:A1} and \eqref{5.12-1} are fulfilled, the last condition in \eqref{A2'} is meanwhile weakened to
\begin{align*}
|h(x)|+|g(t,x,y,0)|\leq K_1\exp(K_2|x|^2)(1+\varphi(|y|))
\end{align*}
and \eqref{6.2} is relaxed to
\begin{align*}
m(t,x)+\frac{\theta}{2[1\wedge(p-1)]}n^2(t,x)\leq K_1+K_3|x|^2\ \ \ {\rm and}\ \ \ p(TK_3+K_2)\leq C_1,
\end{align*}
BSDE \eqref{BSDE-1} admits a unique weighted $L^p$ solution $(Y_s^{t,x},Z_s^{t,x})_{s\in [t,T]}$ in the space of $H_\tau^p(\rho_\cdot^{t,x};\R^{k}\times\R^{k\times d})$ for each $(t,x)\in \T\times \R^n$ with
$$
\rho_r^{t,x}:=m(r,X^{t,x}_r){\bf 1}_{t\leq r\leq T}+\frac{\theta}{2[1\wedge(p-1)]}n^2(r,X^{t,x}_r){\bf 1}_{t\leq r\leq T},\ \ r\in [0,T]
$$
and the function $u(\cdot,\cdot)$ defined by \eqref{u} is a viscosity solution of PDE \eqref{PDE'}. Moreover, there exits a constant $C>0$ depending only on $(p,\theta,T,K_1)$ such that
\begin{align*}
\forall (t,x)\in \T\times\R^n, \ \ |u(t,x)|\leq C\exp(C |x|^2).
\end{align*}
It is obvious that the above assertions require weaker conditions than those employed in Theorem 3.17 of \citet{Li2025}. On the other hand, we would like to emphasize that in \cref{thm:5.1}, the condition \eqref{6.2} allows both the monotonicity coefficient $m(t,x)$ and the Lipschitz coefficient $n(t,x)$ to have a general growth in $x$ due to the fact that $m(t,x)$
can take values in $\R$. For instance, \eqref{6.2} holds for $p=\theta=2$, $$m(t,x)\equiv -e^{2|x|}+\sin |x|\ \ \ {\rm and}\ \ \ n(t,x)\equiv e^{|x|}.$$
\end{rmk}

\subsection{Probabilistic interpretation for elliptic PDEs\vspace{0.2cm}}

Let both $b(x):\R^n\rightarrow\R^n$ and $\sigma(x): \R^n\rightarrow\R^{n\times d}$ be globally Lipschitz continuous functions, and for each $x\in \R^n$, let $(X_t^{x})_{t\geq 0}$ denote the solution of the following SDE:
\begin{align}\label{SDE-2}
X_t^{x}=x+\int_0^t b(X_s^{x}){\rm d}s+\int_0^t \sigma(X_s^{x}){\rm d}B_s, \ \ t\geq 0.
\end{align}
Assume that $G$ is an open bounded subset of $\R^n$, whose boundary $\partial G$ is of class $C^1$, and that there exists a constant $\lambda>0$ such that for all $\zeta\in \mathbb{R}^n$ and $x \in \bar {G}:=G\cup\partial G$,
$$
\sum_{i,j=1}^n (\sigma \sigma^*(x))_{i,j} \zeta_i \zeta_j \geq \lambda |\zeta|^2.
$$
Assume further that there is a constant $\eta> 0$ such that for all $y \in \partial G$, there exists $\tilde{y}\in \R^n\backslash G$ such that
$$
\bar {G} \cap \left\{\check{y}\in \mathbb{R}^n : |\check{y}-\tilde{y}| \leq \eta \right\} = \{ y \}.
$$

For each $x\in \bar {G}$, define the stopping time: $\tau_x\equiv\inf\{t\geq 0: X_t^{x}\notin \bar {G}\}$. According to Lemma 3.1 in \citet{BuckdahnNie(2016)}, it can be concluded that $\mathbb{P}(\tau_x<\infty)=1$ for all $x\in\bar {G}$,
\begin{align}\label{closed}
\Gamma:=\{x\in\partial G:\mathbb{P}(\tau_x>0)=0\}=\partial G
\end{align}
and there exists a constant $\varrho>0$ such that
\begin{align}\label{sup}
\sup_{x \in \bar{G}} E\left[\exp(\varrho\tau_x)\right]<+\infty.
\end{align}

Now, suppose that both $h(x):\R^n\rightarrow \R^k$ and $g(x,y,z): \R^n\times\R^k\times \R^{k\times d}\rightarrow\R^k$ are (jointly) continuous functions, the $i$th component $g_i$ of $g$, for each $1\leq i\leq k$, varies only when $(t,x,z^i)$ changes, where $z^i$ represents the $i$th row of the matrix $z$, and for each $(x,y,y_1,y_2,z,z_1,z_2)\in \R^n\times \R^{3k}\times\R^{3(k\times d)}$,
\begin{align}\label{condition-2}
\begin{array}{c}
\left\langle y_1-y_2,~g(x,y_1,z)-g(x,y_2,z)\right\rangle\leq m(x)|y_1-y_2|^2,\\
|g(x,y,z_1)-g(x,y,z_2)|\leq n(x)|z_1-z_2|,
\end{array}
\end{align}
where both $m(x):\R^n\rightarrow \R$ and $n(x):\R^n\rightarrow \R_+$ are continuous functions satisfying
\begin{align}\label{hb}
p\max\limits_{x\in \bar G}\left\{m(x)+\frac{\theta}{2[1\wedge(p-1)]} n^2(x)\right\}<\varrho.
\end{align}
In light of the definition of stopping time $\tau_x$ and the continuity of functions $h(\cdot)$ and $g(\cdot,0,0)$ along with \eqref{sup} and \eqref{hb}, it is not very difficult to verify that for each $x\in \bar {G}$,
\begin{align}\label{aa}
\E\left[e^{p\int_0^{\tau_x} l^x_s {\rm d}s}|h(X^x_{\tau_x})|^p+
\left(\int_0^{\tau_x}e^{\int_0^t l^x_s {\rm d}s}|g(X^x_t,0,0)| {\rm d}t\right)^p\right]<+\infty
\end{align}
with
$$
l^x_s:=m(X^x_s)+\frac{\theta}{2[1\wedge(p-1)]} n^2(X^x_s).\vspace{0.1cm}
$$
Then, in light of \eqref{aa} and \eqref{condition-2} along with the definition of stopping time $\tau_x$ and the continuity of function $g(\cdot,\cdot,0)$, by \cref{thm:3.1} we can deduce that for each $x\in \bar {G}$, the following BSDE
\begin{align}\label{BSDE-3}
Y_t^{x}=h(X_{\tau_x}^{x})+\int_t^{\tau_x} g(X_s^{x},Y_s^{x},Z_s^{x}){\rm d}s-\int_t^{\tau_x} Z_s^{x}{\rm d}B_s,\ \ t\in [0, \tau_x]
\end{align}
admits a unique weighted $L^p$ solution $(Y_t^{x},Z_t^{x})_{t\in [0,\tau_x]}$ in the space of $H_{\tau_x}^p(l^x_\cdot;\R^{k}\times\R^{k\times d})$.

In the sequel, we consider the following elliptic PDE with Dirichlet boundary condition:
\begin{align}\label{PDE-1}
\left\{\begin{aligned}
&\mathcal{\bar L}u_i(x)+g_i(x,u(x),(\nabla u_i\sigma)(x))=0,\ 1\leq i\leq k,\ x\in G;\\
&u_i(x)=h(x),\ 1\leq i\leq k,\ x\in\partial G,
\end{aligned}\right.
\end{align}
where
$$
\mathcal{\bar L}:=\frac{1}{2}\sum\limits_{i,j}(\sigma\sigma^*)_{i,j}(x)
\frac{\partial^2}{\partial x_i\partial x_j}+\sum\limits_{i}b_i(x)\frac{\partial}{\partial x_i},\ \ \nabla:=(\frac{\partial}{\partial x_1},\cdots,\frac{\partial}{\partial x_n}).
$$
and then give the definition of a continuous viscosity solution to \eqref{PDE-1} in our framework.
\begin{dfn}
A continuous function $u:\bar {G}\rightarrow\R^k$ is a viscosity sub-solution (super-solution) of PDE \eqref{PDE-1} if for any $1\leq i\leq k$ and any function $\phi\in C^2(\bar {G};\R)$, it holds that
$$
\begin{array}{c}
\mathcal{\bar L}\phi(x_0)+g_i(x_0,u(x_0),(\nabla\phi\sigma)(x_0))\geq 0\ (\leq 0), \ x_0\in G;\\
\max\{\mathcal{\bar L}\phi(x_0)+g_i(x_0,u(x_0),(\nabla\phi\sigma)(x_0)),
h_i(x_0)-u_i(x_0)\}\geq 0, \ x_0\in \partial G\\
(\min\{\mathcal{\bar L}\phi(x_0)+g_i(x_0,u(x_0),(\nabla\phi\sigma)(x_0)),
h_i(x_0)-u_i(x_0)\}\leq 0, \ x_0\in \partial G),
\end{array}
$$
as long as $u_i-\phi$ has a local maximum (minimum) at point $x_0\in \bar {G}$. A continuous function $u$ is a viscosity solution of PDE \eqref{PDE-1} if it is both a viscosity sub-solution and a viscosity super-solution.
\end{dfn}

Now we are at the position to give the probabilistic interpretation for elliptic PDE \eqref{PDE-1}.

\begin{thm}\label{thm:5.4}
Under the above assumptions, $u(x):=Y_0^{x},\ x\in \R^n$ is a continuous function on $\bar {G}$ and it is a viscosity solution of PDE \eqref{PDE-1}.
\end{thm}
\begin{proof}
First of all, by Proposition 5.76 in \citet{PardouxandRascanu2014} we know that under the condition \eqref{closed}, the mapping $x\rightarrow \tau_x$ is almost surely continuous on $\bar {G}$. Then, by \cref{thm:3.2} and the continuity of $X_\cdot^{x}$ in $x$, we can deduce that $u(x):=Y_0^x$ is a continuous function on $\bar {G}$.

We only prove that $u$ is a viscosity sub-solution of PDE \eqref{PDE-1}. Let $1\leq i\leq k,$ $\phi\in C^2(\bar {G};\R)$ such that $u_i-\phi$ attains a local maximum at $x_0\in \bar {G}$. If $x_0\in \Gamma$, then $\tau_{x_0}=0$, and hence $u(x_0)=h(x_0)$. If $x_0\in G\cup (\partial G\cap \Gamma^c)$, then $\tau_{x_0}>0$. For the latter case, we assume without loss of generality that
$$
u_i(x_0)=\phi(x_0).\vspace{-0.1cm}
$$
We now assume that
$$\mathcal{\bar L}\phi(x_0)+g_i(x_0,u(x_0),(\nabla\phi\sigma)(x_0))<0,$$
and we will lead to a contradiction. In fact, by continuity let $\delta,\gamma>0$ such that whenever $|y-x_0|\leq \delta,$
\begin{align}\label{6.10}
\left\{\begin{aligned}
&u_i(y)\leq \phi(y);\\
&\mathcal{\bar L}\phi(y)+g_i(y,u(y),(\nabla\phi\sigma)(y))\leq -\gamma.
\end{aligned}\right.
\end{align}
Define the following stopping time:
$$\vartheta:=\inf\{t\geq 0:|X_t^{x_0}-x_0|+\int_0^t n^2(X_s^{x_0}){\rm d}s\geq \delta\}\wedge \tau_{x_0}\wedge \delta>0.$$
By virtue of the Markov property and BSDE \eqref{BSDE-3}, we deduce that $(\bar {Y}_t,\bar {Z}_t)_{t\in [0,\vartheta]}:=((Y^{x_0}_t)^i, (Z^{x_0}_t)^i)_{t\in [0,\vartheta]}$ is an adapted solution of the following one-dimensional BSDE:
\begin{align}\label{6.11}
\bar Y_t=u_i(X_{\vartheta}^{x_0})+\int_t^{\vartheta} g_i(X_s^{x_0},u(X_s^{x_0}),\bar Z_s){\rm d}s-\int_t^{\vartheta} \bar Z_s{\rm d}B_s,\ \ t\in [0,\vartheta].
\end{align}
On the other hand, by It\^{o}'s formula we derive that
$(\widetilde{Y}_t,\widetilde{Z}_t)_{t\in [0,\vartheta]}:=(\phi(X_t^{x_0}),
(\nabla\phi\sigma)(X_t^{x_0}))_{t\in [0,\vartheta]}$ is an adapted solution of the following BSDE:
\begin{align}\label{BSDE5.27}
\widetilde{Y}_t=\phi(X_{\vartheta}^{x_0})-
\int_t^{\vartheta} \mathcal{\bar L}\phi(X_{s}^{x_0}){\rm d}s-\int_t^{\vartheta} \widetilde{Z}_s{\rm d}B_s,\ \ t\in [0,\vartheta].
\end{align}
Finally, in light of \eqref{6.10}-\eqref{BSDE5.27}, the definition of the stopping time $\vartheta$, Girsanov's theorem and the Lipschitz continuity for $g$ in $z$, by a similar argument as that employed in the proof of \cref{thm:5.1} we can deduce that $\widetilde{Y}_0>\bar Y_0$, i.e., $\phi(x_0)>u_i(x_0) $, which yields a contradiction. The proof is complete.
\end{proof}

\begin{rmk}\label{rmk:5.6}
It is obvious that in contrast to Theorem 3.20 of \citet{Li2025}, \cref{thm:5.4} requires weaker condition (namely \eqref{hb}) due to the fact that the function $m(\cdot)$ appearing in \eqref{condition-2} can take values in $\R$. We also mention that some extra conditions are usually required to guarantee the uniqueness of viscosity solutions for PDE \eqref{PDE'} and PDE \eqref{PDE-1}. See for example \cite{BarlesBurdeau(1995),BarlesMurat(1995),CrandallIshiiLions1992BAMS} for more details.
\end{rmk}

\section{Dual representation of unbounded dynamic concave utilities \vspace{0.2cm}}
\setcounter{equation}{0}

It is widely known that a dynamic concave utility is a popular concept in financial mathematics and has received much attention. See \citet{DelbaenHuBao2011,DelbaenPengRosazza Gianin2010,Drapeau2016AIHPPS} and \citet{FanHuTang2025} as well as their references for details. In this section, with the help of \cref{thm:3.1} and \cref{thm:3.3}, we will state and prove a dual representation result of an unbounded dynamic concave utility defined on the general weighted space $L_\tau^p(\rho_\cdot;\R)$, see \cref{thm:6.1}, via the weighted solution of a scalar BSDE with possibly unbounded stochastic coefficients. This result allows us to compute this unbounded dynamic concave utility by solving a BSDE via numerical algorithms such as Monte Carlo method.

Let us first recall that $\tau$ is an $(\F_t)$-stopping time taking values in $[0,+\infty]$, both $p>1$ and $\theta>1$ are given constants, and $\rho_\cdot\in \R$, $\mu_\cdot\in \R$ and $\nu_\cdot\in\R_+$ are three $(\F_t)$-progressively measurable processes satisfying $\int_0^\tau (|\rho_t|+|\mu_t|+\nu_t^2){\rm d}t<+\infty$ and
$$
\rho_\cdot\geq \mu_\cdot+\frac{\theta}{2[1\wedge(p-1)]}\nu_\cdot^2.
$$
In what follows, let $(h_t)_{t\in[0,\tau]}$ be an $(\F_t)$-progressively measurable nonnegative real-valued process, and assume that the random core function
$$
f(\omega,t,q):\Omega \times [0,\tau]\times \R^{1\times d}\mapsto \R\cup\{+\infty\}
$$
is $(\F_t)$-progressively measurable for each $q\in\R^{1\times d}$ and satisfies the following assumption:
\begin{enumerate}
\renewcommand{\theenumi}{(A3)}
\renewcommand{\labelenumi}{\theenumi}
\item\label{A:A3} ${\rm d}\mathbb{P}\times{\rm d} t-a.e.$, $f(\omega,t,\cdot)$ is lower semi-continuous and convex such that \\
    $|f(\omega,t,0)|\leq h_t(\omega)$, $f(\omega,t,q)\geq -h_t(\omega)$ for any $q\in\R^{1\times d},$ and $f(\omega,t,q)\equiv+\infty$ when $|q|>\nu_t(\omega)$.
\end{enumerate}

For each $\F_\tau$-measurable real-valued random variable $\xi$, the random core function $f$, the processes $\mu_\cdot$ and $\nu_\cdot$, we define the following process space:
\begin{align*}
\hcal(\xi,f,\mu_\cdot,\nu_\cdot):=\bigg\{\ \ &(q_t)_{t\in[0,\tau]}~\text{is an } (\F_t)\text{-progressively measurable~}\R^{1\times d} \text{-valued process}:\nonumber\\
\ \ &\as,\ \ |q_t|\leq \nu_t\ \ \text{and} \\
\ \ &\E_{\Q^q}\Big[e^{\int_0^\tau \mu_r\dif r}|\xi|+\int_{0}^{\tau}e^{\int_0^s \mu_r\dif r}(h_s+|f(s,q_s)|)\dif s\Big]<+\infty \\
\ \ &\text{with}\ \  L_t^q:=\exp\Big(\int_{0}^{t}q_s\dif B_s-\frac{1}{2}\int_{0}^{t}|q_s|^2\dif s\Big), ~t\in[0,\tau]\ \ \text{being a uniformly}\\
\ \ &\text{integrable martingale, and}\ \  \frac{\dif \Q^q}{\dif \mathbb{P}}:=L_{\tau}^q \nonumber\ \ \bigg\}
\end{align*}
and state the main result of this subsection as follows.

\begin{thm}\label{thm:6.1}
Assume that $\xi\in L_\tau^p(\rho_\cdot;\R)$, $\int_0^\tau e^{\int_{0}^{s}\rho_r{\rm d}r}h_s{\rm d}s\in L_\tau^p(0;\R_+)$ and the random core function $f$ satisfies assumption \ref{A:A3}. Then, the process $\{U_t(\xi),t\in[0,\tau]\}$ defined by
\begin{align}\label{Ut}
U_{t\wedge \tau}(\xi):=e^{-\int_0^{t\wedge \tau}\mu_r\dif r }\essinf\limits_{q\in\hcal(\xi,f,\mu_\cdot,\nu_\cdot)}\E_{\Q^q}\Big[e^{\int_0^\tau \mu_r\dif r}\xi+\int_{t\wedge \tau}^{\tau}e^{\int_0^s \mu_r\dif r}f(s,q_s)\dif s\big|\F_{t\wedge \tau}\Big], \ \ t\geq 0
\end{align}
can be represented as the value process $y_\cdot$ of the unique weighted $L^p$ solution $(y_t,z_t)_{t\in[0,\tau]}$ in the space of $H_\tau^p(\rho_\cdot;\R\times \R^{1\times d})$ of the following BSDE
\begin{align}\label{BSDE*}
 y_t=\xi+\int_t^\tau [\mu_s y_s-g(s,-z_s)] {\rm d}s-\int_t^\tau z_s{\rm d}B_s, \ \ t\in[0,\tau],
\end{align}
where the random function $g$ is the Legendre-Fenchel transform of $f$, i.e.,
\begin{align}\label{g_1}
g(\omega,t,z):=\sup_{q\in\R^{1\times d}}\big(q\cdot z-f(\omega,t,q)\big),\ \ (\omega,t,z)\in\Omega\times[0,\tau]\times\R^{1\times d}.
\end{align}
Consequently, the operator $\{U_t(\cdot),t\in[0,\tau]\}$ constitutes a dynamic concave utility function defined on the space of $L_\tau^p(\rho_\cdot;\R)$, i.e., it satisfies the following properties: for each $t\geq 0$,
\begin{itemize}
\item [\textbullet]\text{Monotonicity:} $U_{t\wedge \tau}(\xi)\geq U_{t\wedge \tau}(\eta)$ for each $\xi, \eta\in L_\tau^p(\rho_\cdot;\R)$ such that $\xi\geq\eta$;
\item [\textbullet]\text{Concavity:} $U_{t\wedge \tau}(\lambda\xi+(1-\lambda)\eta)\geq \lambda U_{t\wedge \tau}(\xi)+(1-\lambda) U_{t\wedge \tau}(\eta)$ for all $\xi, \eta\in L_\tau^p(\rho_\cdot;\R)$ and $\lambda\in(0,1)$;
\item [\textbullet]\text{Time consistency:} $U_{s\wedge \tau}(\xi)=U_{s\wedge \tau}(U_{t\wedge \tau}(\xi))$ for each $\xi\in L_\tau^p(\rho_\cdot;\R)$ and each $s\in [0,t]$.
\end{itemize}
Moreover, let $\tilde q_s\in\partial g(s,-z_s), \ s\in[0,\tau]$. By this notation we mean that $(\tilde q_t)_{t\in [0,\tau]}$ is an $(\F_t)$-progressively measurable $\R^{1\times d}$-valued process such that for each $s\in[0,\tau]$ and $u\in \R^{1\times d}$,
\begin{align*}
g(s,u)-g(s,-z_s)\geq \tilde q_s\cdot (u+z_s)\ \ \ {\rm or \ equivalently}\ \ \ \tilde q_s\cdot u-g(s,u)\leq -\tilde q_s\cdot z_s-g(s,-z_s).
\end{align*}
Then, the infimum in \eqref{Ut} is attained for $\tilde q_\cdot$ provided that anyone of the following three conditions holds:
\begin{itemize}
\item [(i)] $\E[\exp\left(\frac{\bar p}{2}\int_0^\tau \nu_t^2 \dif t\right)]<+\infty$ for any $\bar p>1$;

\item [(ii)] $\E[\exp\left(\frac{1}{2}\int_0^\tau \nu_t^2 \dif t\right)]<+\infty$ and $\int_0^\tau e^{\int_{0}^{s}\mu_r{\rm d}r}h_s{\rm d}s\leq K$ for some constant $K>0$;

\item [(iii)] the process $L_t^{\tilde q}:=\exp(\int_{0}^{t}\tilde q_s\dif B_s-\frac{1}{2}\int_{0}^{t}|\tilde q_s|^2\dif s), \ t\in[0,\tau]$
is a uniformly integrable martingale and $\int_0^\tau e^{\int_{0}^{s}\mu_r{\rm d}r}h_s{\rm d}s\leq K$ for some constant $K>0$.\vspace{0.2cm}
\end{itemize}
\end{thm}

\begin{rmk}\label{rmk:5.3}
Two closely related results to \cref{thm:6.1} were obtained in \citet[Theorem 3.2]{DelbaenPengRosazza Gianin2010} and
\citet[(v) of Theorem 3.2]{FanHuTang2025}, where $\tau$ is a finite positive constant, both $\mu_\cdot$ and $\nu_\cdot$ are nonnegative constants, and $\xi$ belongs to the space of $L^\infty(\Omega,\F_\tau,\mathbb{P};\R)$ or a more general space smaller than $L^1_\tau(0;\R)$ but bigger than $\cap_{q>1} L^q_\tau (0;\R)$. It should be emphasized that in light of \ref{A:A3} and \eqref{Ut}, the process $\{U_t(\xi),t\in[0,\tau]\}$ in \cref{thm:6.1} admits a more general form than those in \cite{DelbaenPengRosazza Gianin2010,FanHuTang2025} since $\mu_\cdot\in \R$ and $\nu_\cdot\in \R_+$ can be two unbounded stochastic processes without any moment integrability. In particular, noticing that a random variable $\xi$ belonging to $L^p_\tau(\rho_\cdot,\R)$ does not necessarily belong to $L^1_\tau(0;\R)$ in the case that $\rho_\cdot$ is a process which may take negative values, we can conclude that the dual representation in \cref{thm:6.1} holds still for some non-integrable random variables under some certain scenarios. Readers are also referred to \citet{DelbaenHuBao2011}, \citet{Drapeau2016AIHPPS} and \citet{FanHuTang2025} for the dual representation results of dynamic concave utilities under various conditions on the core function $f$.\vspace{0.2cm}
\end{rmk}

\begin{proof}[\bf {Proof of \cref{thm:6.1}}]
To simplify the calculations, we assume without loss of generality that $d = 1$. First of all, by virtue of \cref{thm:3.1} we verify that for each $\xi\in L_\tau^p(\rho_\cdot;\R)$, BSDE \eqref{BSDE*} admits a unique weighted $L^p$ solution $(y_t,z_t)_{t\in[0,\tau]}$ such that
\begin{align}\label{supyt1}
(y_\cdot,z_\cdot)\in H_\tau^p(\rho_\cdot;\R\times\R).
\end{align}
Indeed, in light of \eqref{g_1} and \ref{A:A3}, we deduce that for each $t\in [0,\tau]$ and $z\in\R$,
\begin{align}\label{eq:6.6}
-h_t(\omega)\leq g(\omega,t,z)=\sup_{q\in\R}\big(q z-f(\omega,t,q)\big)=\sup_{|q|\leq \nu_t(\omega)}\big(q z-f(\omega,t,q)\big)
\leq \nu_t(\omega)|z|+h_t(\omega).
\end{align}
It then follows from the convexity of $g(\omega,t,\cdot)$ that $\as$, for each $\lambda\in (0,1)$ and each $z_1,z_2\in \R$,
\begin{align*}
g(\omega,t,z_1)&=\Dis g\left(\omega,~t,~\lambda z_2+(1-\lambda)\frac{z_1-\lambda z_2}{1-\lambda}\right)\\
&\leq  \Dis \lambda g(\omega,t,z_2)+(1-\lambda)\left(\nu_t(\omega) \frac{|z_1-\lambda z_2|}{1-\lambda}+h_t(\omega)\right)\\
&= \Dis \lambda g(\omega,t,z_2)+(1-\lambda) h_t(\omega)+\nu_t(\omega) |z_1-\lambda z_2|,
\end{align*}
and then, by letting $\lambda\To 1$ in the last inequality, we obtain
\begin{align}\label{eq:6.7}
|g(\omega,t,z_1)-g(\omega,t,z_2)|\leq \nu_t(\omega) |z_1-z_2|.
\end{align}
In light of \eqref{eq:6.6} and \eqref{eq:6.7} along with the assumption of $\E\left[\left(\int_0^\tau e^{\int_{0}^{s}\rho_r{\rm d}r}h_s{\rm d}s\right)^p\right]<+\infty$, it can be easily verified that the generator of BSDE \eqref{BSDE*}, namely
\begin{align}\label{eq:6.8}
\tilde{g}(\omega,t,y,z):=\mu_t(\omega) y-g(\omega,t,-z),\ \ (\omega,t,y,z)\in\Omega\times[0,\tau]\times\R\times\R
\end{align}
satisfies assumptions \ref{A:H1}-\ref{A:H5} with $\alpha_t:=e^{-\int_0^t \mu_r \dif r}, \ t\in[0,\tau]$. Thus, the desired existence and uniqueness result follows immediately from \cref{thm:3.1}.

Next, we prove the dual representation in the following three steps. To do this, we need to verify that for each $t\in[0,\tau]$, $U_t(\xi)$ defined by \eqref{Ut} is identical to $y_t$, where $(y_t,z_t)_{t\in[0,\tau]}$ is the unique weighted $L^p$ solution of BSDE \eqref{BSDE*} in the space of $H_\tau^p(\rho_\cdot;\R\times\R)$. Note that from the assumptions of \cref{thm:6.1}, it can be easily derived that $0\in \hcal(\xi,f,\mu_\cdot,\nu_\cdot)$, and then the space of $\hcal(\xi,f,\mu_\cdot,\nu_\cdot)$ is non-empty.\vspace{0.2cm}

{\bf First Step:} In this step, we aim to obtain that for each $t\geq 0$,
\begin{align}\label{5.02*}
y_{t\wedge \tau}\leq U_{t\wedge \tau}(\xi):=e^{-\int_0^{t\wedge \tau}\mu_r\dif r }\essinf\limits_{q\in\hcal(\xi,f,\mu_\cdot,\nu_\cdot)}\E_{\Q^q}\Big[e^{\int_0^\tau \mu_r\dif r}\xi+\int_{t\wedge \tau}^{\tau}e^{\int_0^s \mu_r\dif r}f(s,q_s)\dif s\big|\F_{t\wedge \tau}\Big].
\end{align}
Indeed, for each $q_\cdot\in\hcal(\xi,f,\mu_\cdot,\nu_\cdot)$, we need to prove that for each $t\geq 0$,
\begin{align}\label{5.02}
y_{t\wedge \tau}\leq e^{-\int_0^{t\wedge \tau}\mu_r\dif r }\E_{\Q^q}\Big[e^{\int_0^\tau \mu_r\dif r}\xi+\int_{t\wedge \tau}^{\tau}e^{\int_0^s \mu_r\dif r}f(s,q_s)\dif s\big|\F_{t\wedge \tau}\Big].
\end{align}
According to \ref{A:A3} and \eqref{g_1}, by the dual representation of a convex function, we know that
\begin{align}\label{5.03}
f(\omega,s,q)=\sup_{z\in\R}\big(qz-g(\omega,s,z)\big),~~~(\omega,s,q)\in\Omega\times[0,\tau]\times\R.
\end{align}
For each $n\geq1$ and $t\in[0,\tau]$, define the following stopping time
$$
\tau_n^t:=\inf\Big\{s\geq t: \int_t^se^{2\int_0^u \mu_r \dif r}|z_u|^2\dif u\geq n\Big\}\wedge \tau.
$$
With the help of the exponential shift transformation and Girsanov's theorem, it follows from \eqref{BSDE*} and \eqref{5.03} that for each $n\geq1$ and $t\in [0,\tau]$,
\begin{align*}
e^{\int_0^{t}\mu_r\dif r}y_t&=e^{\int_0^{\tau_n^t}\mu_r\dif r }y_{\tau_n^t}-\int_t^{\tau_n^t} e^{\int_0^{s}\mu_r\dif r}g(s,-z_s)\dif s-\int_t^{\tau_n^t} e^{\int_0^s\mu_r\dif r}z_s \dif B_s\\
&=e^{\int_0^{\tau_n^t}\mu_r\dif r }y_{\tau_n^t}+\int_t^{\tau_n^t} e^{\int_0^{s}\mu_r\dif r}\left[-q_sz_s-g(s,-z_s)\right]\dif s-\int_t^{\tau_n^t} e^{\int_0^{s}\mu_r\dif r}z_s \dif B_s^q\\
&\leq e^{\int_0^{\tau_n^t}\mu_r\dif r }y_{\tau_n^t}+\int_t^{\tau_n^t} e^{\int_0^{s}\mu_r\dif r}f(s,q_s)\dif s-\int_t^{\tau_n^t} e^{\int_0^{s}\mu_r\dif r}z_s \dif B_s^q,
\end{align*}
where the shifted Brownian motion $B_s^q:=B_s-\int_0^s q_r \dif r, \ s\in[0,\tau]$ is a standard one-dimensional Brownian motion under the new probability measure $\Q^q$. Taking the conditional mathematical expectation with respect to $\F_{t\wedge \tau}$ under $\Q^q$ on both sides of the last inequality, we obtain that for each $n\geq 1$ and $t\geq 0$,
\begin{align}\label{5.04}
y_{t\wedge \tau}\leq e^{-\int_0^{t\wedge \tau}\mu_r\dif r}\E_{\Q^q}\Big[e^{\int_0^{\tau_n^t}\mu_r\dif r }y_{\tau_n^t}+\int_{t\wedge \tau_n^t}^{\tau_n^t}e^{\int_0^s \mu_r\dif r}f(s,q_s)\dif s\bigg|\F_{t\wedge \tau}\Big].
\end{align}
Furthermore, by virtue of $|q_\cdot|\leq \nu_\cdot$, $\rho_\cdot\geq \mu_\cdot+\frac{\theta}{2[1\wedge(p-1)]}\nu_\cdot^2$, H${\rm \ddot{o}}$lder's inequality and \eqref{supyt1}, we can verify that the process $(e^{\int_0^{s}\mu_r\dif r}|y_s|L_s^q)_{s\in[0,\tau]}$ is uniformly integrable. In fact, there exists a constant
$$
p':=\frac{(1+\frac{\theta}{1\wedge(p-1)})p}{1+\frac{\theta}{1\wedge(p-1)}+p}\in(1,p)
$$
such that for each $s\geq 0$,
\begin{align*}
&\E\left[\left(e^{\int_0^{s\wedge \tau}\mu_r\dif r}|y_{s\wedge \tau}|L_{s\wedge \tau}^q\right)^{p'}\right]\\
&\ \ =\E\left[{e^{p'\int_0^{s\wedge \tau}\mu_r\dif r}} e^{p'\int_0^{s\wedge \tau}\frac{\theta}{2[1\wedge(p-1)]}|q_r|^2 \dif r}|y_{s\wedge \tau}|^{p'}e^{p'\int_0^{s\wedge \tau} q_r\dif B_r-(\frac{p'}{2}+\frac{p'\theta}{2[1\wedge(p-1)]})\int_0^{s\wedge \tau} |q_r|^2\dif r}\right]\\
&\ \ \leq \E\left[{e^{p'\int_0^{s\wedge \tau}\rho_r\dif r}} |y_{s\wedge \tau}|^{p'} e^{p'\int_0^{s\wedge \tau} q_r\dif B_r-(\frac{p'}{2}+\frac{p'\theta}{2[1\wedge(p-1)]})\int_0^{s\wedge \tau} |q_r|^2\dif r}\right]\\
&\ \ \leq \left\{\E\left[{e^{p\int_0^{s\wedge \tau}\rho_r\dif r}} |y_{s\wedge \tau}|^{p}\right]\right\}^{\frac{p'}{p}} \left\{\E\left[e^{\frac{pp'}{p-p'}\int_0^{s\wedge \tau} q_r\dif B_r-\frac{p}{p-p'}(\frac{p'}{2}+\frac{p'\theta}{2[1\wedge(p-1)]})\int_0^{s\wedge \tau} |q_r|^2\dif r}\right]\right\}^{\frac{p-p'}{p}}\\
&\ \ =\left\{\E\left[e^{p\int_0^{s\wedge \tau}\rho_r\dif r} |y_{s\wedge \tau}|^{p}\right]\right\}^{\frac{p'}{p}} \left\{\E\left[e^{\frac{pp'}{p-p'}\int_0^{s\wedge \tau} q_r\dif B_r-\frac{p^{2}p'^{2}}{2(p-p')^2}\int_0^{s\wedge \tau} |q_r|^2\dif r}\right]\right\}^{\frac{p-p'}{p}}\\
&\ \ \leq\left\{\E\left[\sup_{s\in[0,\tau]}\left(e^{p\int_0^{s}\rho_r\dif r} |y_s|^{p}\right)\right]\right\}^{\frac{p'}{p}}<+\infty
\end{align*}
and then
\begin{align}\label{eq:6.11*}
\sup_{s\geq 0}\E\left[\left(e^{\int_0^{s\wedge \tau}\mu_r\dif r}|y_{s\wedge \tau}|L_{s\wedge \tau}^q\right)^{p'}\right]<+\infty.
\end{align}
Thus, in light of the last inequality and the fact that $q_\cdot\in \hcal(\xi,f,\mu_\cdot,\nu_\cdot)$, the desired assertion \eqref{5.02} and then \eqref{5.02*} follows immediately by sending $n$ to infinity in \eqref{5.04}.\vspace{0.2cm}

{\bf Second Step:} In this step, we assume further that $\int_0^{\tau}\nu_t^2 \dif t\leq K$ for some constant $K>0$ and prove that for any $\tilde{q}_s\in\partial g(s,-z_s),\ s\in[0,\tau]$, we have $\tilde{q}_\cdot\in \hcal(\xi,f,\mu_\cdot,\nu_\cdot)$ and
\begin{align}\label{5.05}
y_{t\wedge \tau}= e^{-\int_0^{t\wedge \tau}\mu_r\dif r}\E_{\Q^{\tilde{q}}}\Big[e^{\int_0^{\tau}\mu_r\dif r }\xi+\int_{t\wedge \tau}^{\tau}e^{\int_0^s \mu_r\dif r}f(s,\tilde{q}_s)\dif s\big|\F_{t\wedge \tau}\Big],\ \ t\geq 0.
\end{align}
First of all, since $\tilde{q}_s\in\partial g(s,-z_s),\ s\in[0,\tau]$, by \eqref{5.03} we have
\begin{align}\label{eq:6.13}
f(s,\tilde{q}_s)=\sup_{z\in\R}\big(\tilde{q}_s z-g(s,z)\big)=-\tilde{q}_{s}z_{s}-g(s,-z_s),\ \ s\in[0,\tau].
\end{align}
Then, in light of assumption \ref{A:A3}, we can conclude that $|\tilde{q}_\cdot|\leq \nu_\cdot$, $\int_0^{\tau}|\tilde{q}_t|^2 \dif t\leq K$, the stochastic exponential process $L_t^{\tilde{q}}:=\exp(\int_{0}^{t}\tilde{q}_s\dif B_s-\frac{1}{2}\int_{0}^{t}|\tilde{q}_s|^2\dif s), \ t\in[0,\tau]$ is a uniformly integrable martingale, and $L_\tau^{\tilde{q}}$ admits moments of any order. To show that $(\tilde{q}_t)_{t\in[0,\tau]}\in \hcal(\xi,f,\mu_\cdot,\nu_\cdot)$, it remains to prove that
\begin{align}\label{5.06}
\E_{\Q^{\tilde{q}}}\Big[e^{\int_0^\tau \mu_r\dif r}|\xi|+\int_{0}^{\tau}e^{\int_0^s \mu_r\dif r}(h_s+|f(s,\tilde{q}_s)|)\dif s\Big]<+\infty.
\end{align}
Indeed, from \ref{A:A3}, $\xi\in L_\tau^p(\rho_\cdot;\R)$, $\int_0^\tau e^{\int_{0}^{s}\rho_r{\rm d}r}h_s{\rm d}s\in L_\tau^p(0;\R_+)$ and H${\rm \ddot{o}}$lder's inequality it follows that
\begin{align}\label{5.07}
&\E_{\Q^{\tilde{q}}}\left[e^{\int_0^\tau \mu_r\dif r}|\xi|+\int_{0}^{\tau}e^{\int_0^s \mu_r\dif r}(h_s+f^{-}(s,\tilde{q}_s))\dif s\right]\leq 2\E\left[\left(e^{\int_0^\tau \rho_r\dif r}|\xi|+\int_{0}^{\tau}e^{\int_0^s \rho_r\dif r}h_s\dif s\right) L_\tau^{\tilde{q}}\right]\nonumber\\
& \ \ \ \leq 2\left\{\E\left[\left(e^{\int_0^\tau \rho_r\dif r}|\xi|+\int_{0}^{\tau}e^{\int_0^s \rho_r\dif r}h_s\dif s\right)^p\right]\right\}^{\frac{1}{p}} \left\{\E\left[(L_\tau^{\tilde{q}})^{\frac{p}{p-1}}
\right]\right\}^{\frac{p-1}{p}}<+\infty.
\end{align}
In light of \eqref{eq:6.13}, by an identical computation to \eqref{5.04} we derive that for each $n\geq 1$ and $t\geq 0$,
\begin{align}\label{5.08}
y_{t\wedge \tau}=e^{-\int_0^{t\wedge \tau}\mu_r\dif r}\E_{\Q^{\tilde{q}}}\Big[e^{\int_0^{\tau_n^t}\mu_r\dif r }y_{\tau_n^t}+\int_{t\wedge \tau_n^t}^{\tau_n^t}e^{\int_0^s \mu_r\dif r}f(s,\tilde{q}_s)\dif s\bigg|\F_{t\wedge \tau}\Big].
\end{align}
In light of \eqref{supyt1}, by H${\rm \ddot{o}}$lder's inequality we deduce that
\begin{align}\label{5.07*}
\begin{split}
\E_{\Q^{\tilde{q}}}\left[\sup_{t\in[0,\tau]}\left(e^{\int_0^t\mu_r\dif r}|y_t|\right)\right]
\leq \left\{\E\left[\sup_{t\in[0,\tau]}\left(e^{p\int_0^t \rho_r\dif r}|y_t|^p \right) \right]\right\}^{\frac{1}{p}}\left\{\E\left[\left(L_\tau^{\tilde{q}}\right)^{\frac{p}{p-1}}\right]\right\}^{\frac{p-1}{p}}<+\infty.
\end{split}
\end{align}
Then, by virtue of \eqref{5.08} with $t=0$, \eqref{5.07*} and \eqref{5.07} we obtain
\begin{align}\label{eq:6.18}
&\sup_{n\geq1}\E_{\Q^{\tilde{q}}}\Big[\int_0^{\tau_n^0} e^{\int_0^{s}\mu_r\dif r} f^{+}(s,\tilde{q}_s)\dif s\Big]\nonumber\\
& \ \ \ \leq |y_0|+\E_{\Q^{\tilde{q}}}\left[\sup_{t\in[0,\tau]}\left(e^{\int_0^t\mu_r\dif r}|y_t|\right)\right]+\E_{\Q^{\tilde{q}}}\Big[\int_0^{\tau} e^{\int_0^{s}\mu_r\dif r} f^{-}(s,\tilde{q}_s)\dif s\Big]<+\infty,
\end{align}
which along with Fatou's lemma and \eqref{5.07} yields \eqref{5.06}.

Finally, sending $n$ to infinity in \eqref{5.08} and applying Lebesgue's dominated convergence theorem, we get \eqref{5.05}. Furthermore, combining \eqref{5.02*} and \eqref{5.05} yields that $U_t(\xi)=y_t$ for each $t\in[0,\tau]$ under the condition of $\int_0^{\tau}\nu_t^2 \dif t\leq K$, and the infimum in \eqref{Ut} is attained for $q_s\in\partial g(s,-z_s), \ s\in[0,\tau]$.\vspace{0.2cm}

{\bf Third Step:} In this step, we remove the condition of $\int_0^{\tau}\nu_t^2 \dif t\leq K$ to prove the dual representation. In light of the assertion \eqref{5.02*} in the first step, we only need to prove that for each $t\geq 0$,
\begin{align}\label{5.041}
y_{t\wedge \tau}\geq e^{-\int_0^{t\wedge \tau}\mu_r\dif r }\essinf\limits_{q\in\hcal(\xi,f,\mu_\cdot,\nu_\cdot)}\E_{\Q^q}\Big[e^{\int_0^\tau \mu_r\dif r}\xi+\int_{t\wedge \tau}^{\tau}e^{\int_0^s \mu_r\dif r}f(s,q_s)\dif s\big|\F_{t\wedge \tau}\Big].
\end{align}

For each $n\geq 1$ and $t\in [0,\tau]$, we define $\nu_t^{n}:=\nu_t\wedge(ne^{-t})$ satisfying $\int_0^\tau (\nu_t^{n})^2\dif t\leq n^2$ and define
\begin{equation}\label{f^n}
f_n(\omega,t,q):=\begin{cases}
f(\omega,t,q), & |q|\leq \nu_t^n(\omega);\\
+\infty, & \text{otherwise},
\end{cases}
\end{equation}
which is an $(\F_t)$-progressively measurable real-valued process for each $q\in\R$. In light of \eqref{f^n}, it follows from \ref{A:A3} of the core function $f$ that for each $n\geq 1$, $f_n$ also satisfies \ref{A:A3} with $\nu_\cdot^n$ instead of $\nu_\cdot$. And, by an identical way to prove \eqref{eq:6.6}-\eqref{eq:6.8}, we can verify that for each $n\geq 1$, the following generator
\begin{align}\label{eq:6.20}
\tilde{g}_n(\omega,t,y,z):=\mu_t(\omega) y-g_n(\omega,t,-z),\ \ (\omega,t,y,z)\in\Omega\times[0,\tau]\times\R\times\R
\end{align}
satisfies assumptions \ref{A:H1}-\ref{A:H5} with
$\alpha_t:=e^{-\int_0^t \mu_r \dif r},~t\in[0,\tau]$,
where $g_n$ is defined as the Legendre-Fenchel transform of $f_n$, i.e.,
\begin{align}\label{g^n_1}
g_n(\omega,t,z):=\sup_{q\in\R}\big(qz-f_n(\omega,t,q)\big),~~~(\omega,t,z)\in\Omega\times[0,\tau]\times\R,
\end{align}
Then, it follows from \cref{thm:3.1} that for each $n\geq1$ and $\xi\in L_\tau^p(\rho_\cdot;\R)$, the following scalar BSDE
\begin{align}\label{BSDE*1}
 y_t^n=\xi+\int_t^\tau \tilde{g}_n(s,y_s,z_s){\rm d}s-\int_t^\tau z_s^n{\rm d}B_s, \ \ t\in[0,\tau],
\end{align}
admits a unique weighted $L^p$ solution $(y_t^n,z_t^n)_{t\in[0,\tau]}$ such that
\begin{align}\label{5.10*}
 (y_\cdot^n,z_\cdot^n) \in H_\tau^p(\rho_\cdot;\R\times\R).
\end{align}
Furthermore, for each $n\geq 1$, let us set $\bar{q}_s^n\in\partial g_n(s,-z_s^n),\ s\in [0,\tau]$. By the conclusion of the second step we know that for each $n\geq 1$ and $t\geq 0$, $\bar{q}_\cdot^n\in\hcal(\xi,f_n,\mu_\cdot,\nu^n_\cdot)$ and
\begin{align}\label{5.10}
y_{t\wedge \tau}^n= e^{-\int_0^{t\wedge \tau}\mu_r\dif r } \E_{\Q^{\bar{q}^{n}}}\Big[e^{\int_0^\tau \mu_r\dif r}\xi+\int_{t\wedge \tau}^{\tau}e^{\int_0^s \mu_r\dif r}f_n(s,\bar{q}_s^n)\dif s\bigg|\F_{t\wedge \tau}\Big].
\end{align}
Note that $\hcal(\xi,f_n,\mu_\cdot,\nu^n_\cdot)\subset\hcal(\xi,f,\mu_\cdot,\nu_\cdot)$ since for any $q_\cdot\in \hcal(\xi,f_n,\mu_\cdot,\nu^n_\cdot)$, we have $|q_\cdot|\leq \nu^n_\cdot\leq\nu_\cdot$ and then it follows from \eqref{f^n} that $f_n(t,q_t)=f(t,q_t)$. By \eqref{5.10} we derive that for each $n\geq 1$ and $t\geq 0$,
\begin{align}\label{5.11}
y_{t\wedge \tau}^n\geq e^{-\int_0^{t\wedge \tau}\mu_r\dif r }\essinf\limits_{q\in\hcal(\xi,f,\mu_\cdot,\nu_\cdot)}\E_{\Q^q}\Big[e^{\int_0^\tau \mu_r\dif r}\xi+\int_{t\wedge \tau}^{\tau}e^{\int_0^s \mu_r\dif r}f(s,q_s)\dif s\bigg|\F_{t\wedge \tau}\Big].
\end{align}

Next, in light of \eqref{5.11}, in order to obtain the desired inequality \eqref{5.041} it suffices to prove that there exists a subsequence of $\{y^n_\cdot\}$, denoted still by $\{y^n_\cdot\}$, such that
\begin{align}\label{5.131}
\lim\limits_{n\rightarrow\infty}\sup_{t\in[0,\tau]}|y_t^n-y_t|=0.
\end{align}
First, by \eqref{g^n_1}, \eqref{f^n}, \eqref{g_1} and \ref{A:A3} we know that for each $n\geq1$ and $(\omega,t,z)\in \Omega\times [0,\tau]\times\R$,
\begin{align}\label{5.111}
g_n(\omega,t,z)=\sup_{|q|\leq\nu_t^n(\omega)}\big(q z-f(\omega,t,q)\big)
\end{align}
and
\begin{align}\label{5.112}
g(\omega,t,z)=\sup_{|q|\leq\nu_t(\omega)}\big(qz-f(\omega,t,q)\big).
\end{align}
It follows from assumption \ref{A:A3} that there exists two $(\F_t)$-progressively measurable real-valued processes
$$
a_t(\omega):\Omega \times [0,\tau]\mapsto \R_{-}\ \ \ {\rm and}\ \ \  b_t(\omega):\Omega \times [0,\tau]\mapsto \R_{+}
$$
such that $f(\omega,t,q)<+\infty$ for any $q\in \D:=(a_t(\omega),b_t(\omega))\subset (-\nu_t(\omega),\nu_t(\omega))$, and $f(\omega,t,q)=+\infty$ when $q\in\D^c\backslash\{a_t(\omega),b_t(\omega)\}$. Then, by denoting $\D^n\triangleq (a_t(\omega)\vee (-ne^{-t}),b_t(\omega)\wedge (ne^{-t}))$ and in light of \eqref{5.111} and \eqref{5.112}, we have for each $n\geq 1$ and $t\in [0,\tau]$,
\begin{align}\label{eq:6.31}
g_n(\omega,t,z)=\sup_{
q\in\D^n}\big(qz-f(\omega,t,q)\big)~~~\text{and}~~~
g(\omega,t,z)=\sup_{q\in\D}\big(qz-f(\omega,t,q)\big),
\end{align}
which implies that for each $t\in [0,\tau]$,
$g_n(t,-z_t)\rightarrow g(t,-z_t)$ as $n\rightarrow\infty$ and, in light of \eqref{eq:6.8} and \eqref{eq:6.20},
\begin{align}\label{5.113}
\tilde{g}_n(t,y_t,z_t) \rightarrow \tilde{g}(t,y_t,z_t)\ \ \text{as}\ \ n\rightarrow\infty.
\end{align}
On the other hand, setting $\theta'\in(0,\theta-1)$ and $\rho'_\cdot:=\rho_\cdot-\frac{\theta'}{2[1\wedge(p-1)]}\nu_\cdot^2$, by assumption \ref{A:H5}, \eqref{eq:6.6} and H${\rm \ddot{o}}$lder's inequality we can deduce that for each $n\geq 1$,
\begin{align}\label{5.12}
&\int_0^\tau e^{ \int_0^t\rho'_r{\rm d}r}|\tilde{g}_n(t,y_t,z_t)-\tilde{g}(t,y_t,z_t)|{\rm d}t=\int_0^\tau e^{ \int_0^t\rho'_r{\rm d}r}|g_n(t,-z_t)-g(t,-z_t)|{\rm d}t\nonumber\\
&\ \ \ \leq \int_0^\tau e^{ \int_0^t\rho'_r{\rm d}r}|g_n(t,0)|{\rm d}t+\int_0^\tau e^{ \int_0^t\rho'_r{\rm d}r}|g(t,0)|{\rm d}t+2\int_0^\tau e^{ \int_0^t\rho'_r{\rm d}r}\nu_t |z_t|{\rm d}t\nonumber\\
&\ \ \ \leq2 \int_0^\tau e^{\int_0^t\rho_r{\rm d}r}h_t{\rm d}t+2 \int_0^\tau e^{\int_0^t(\rho_r-\frac{\theta'}{2[1\wedge(p-1)]}\nu_r^2){\rm d}r}\nu_t|z_t|{\rm d}t\\
&\ \ \ \leq2 \int_0^\tau e^{\int_0^t\rho_r{\rm d}r}h_t{\rm d}t+\frac{2[1\wedge(p-1)]}{\theta'} \left(\int_0^\tau e^{2\int_0^t\rho_r{\rm d}r}|z_t|^2{\rm d}t\right)^{\frac{1}{2}}
=:X. \nonumber
\end{align}
And, it follows from the fact of $\int_0^\tau e^{\int_{0}^{s}\rho_r{\rm d}r}h_s{\rm d}s\in L_\tau^p(0;\R_+)$ and \eqref{supyt1} that
\begin{align}\label{5.13}
\E\left[X^p\right]<+\infty.
\end{align}
In light of $\rho_\cdot\geq \mu_\cdot+\frac{\theta}{2[1\wedge(p-1)]}\nu_\cdot^2$, we have $$\rho'_\cdot\geq\mu_\cdot+\frac{\theta-\theta'}{2[1\wedge(p-1)]}\nu_\cdot^2$$ with $\theta-\theta'>1$. By virtue of $\rho'_\cdot\leq \rho_\cdot$, \eqref{supyt1} and \eqref{5.10*}, we obtain that $\xi\in L_\tau^p(\rho_\cdot;\R)\subseteq L_\tau^p(\rho'_\cdot;\R)$ and the unique weighted $L^p$ solutions of both BSDE \eqref{BSDE*} and BSDE \eqref{BSDE*1} belong to $H_\tau^p(\rho'_\cdot;\R\times\R)$. Thus, according to \cref{thm:3.3} and Lebesgue's dominated convergence theorem, by \eqref{5.113}-\eqref{5.13} we get
\begin{align}\label{eq:6.35}
\lim_{n\rightarrow\infty}\E\left[\sup_{t\in[0,\tau]}\left(e^{p \int_{0}^{t}\rho'_r{\rm d}r}|y_t^n-y_t|^p\right)+\left(\int_0^{\tau} e^{2\int_0^t \rho'_r{\rm d}r}|z_t^n-z_t|^2{\rm d}t\right)^{\frac{p}{2}}\right]=0.
\end{align}
Finally, in light of \eqref{eq:6.35} and the fact of $\int_0^\tau |\rho'_r|{\rm d}r<+\infty$, a similar analysis to \eqref{3.61} yields the desired assertion \eqref{5.131}. As a result, we have proved the dual representation in \cref{thm:6.1}.\vspace{0.2cm}

In the sequel, it can be easily verified from \eqref{Ut} that the operator $\{U_t(\cdot),t\in[0,\tau]\}$ satisfies monotonicity and concavity. And, according to the dual representation along with the structure property of BSDE \eqref{BSDE*}, we can also show the time consistency of $\{U_t(\cdot),t\in[0,\tau]\}$. Therefore, $\{U_t(\cdot),t\in[0,\tau]\}$ defined via \eqref{Ut} constitutes a dynamic concave utility function defined on the space of $L_\tau^p(\rho_\cdot;\R)$.\vspace{0.2cm}

To complete the proof of \cref{thm:6.1}, it remains to prove the attainability of infimum in \eqref{Ut} under three different cases. Let $\tilde{q}_s\in\partial g(s,-z_s),\ s\in[0,\tau]$. It follows from \eqref{eq:6.13} and \ref{A:A3} that $|\tilde{q}_\cdot|\leq \nu_\cdot$. Define
\begin{align}\label{eq:6.36*}
L_t^{\tilde{q}}:=\exp\left(\int_{0}^{t}\tilde{q}_s\dif B_s-\frac{1}{2}\int_{0}^{t}|\tilde{q}_s|^2\dif s\right), \ \ t\in[0,\tau].
\end{align}
We first consider the first case that $\E[\exp\left(\frac{\bar p}{2}\int_0^\tau \nu_t^2 \dif t\right)]<+\infty$ for any $\bar p>1$. In this case, the stochastic exponential process $(L_t^{\tilde q})_{t\in [0,\tau]}$ defined in \eqref{eq:6.36*} is a uniformly integrable martingale. Furthermore, by H\"{o}lder's inequality along with the fact of $|\tilde{q}_\cdot|\leq \nu_\cdot$ we can deduce that for each $\bar p>1$,\vspace{0.1cm}
\begin{align*}
\E\left[\left(L_\tau^{\tilde q}\right)^{\bar p}\right]
&=\E\left[\exp\left(\bar p\int_0^\tau \tilde q_s {\rm d}B_s-\frac{\bar p^3}{2}\int_0^\tau |\tilde q_s|^2 {\rm d}s\right)
\exp\left(\frac{\bar p^3-\bar p}{2}\int_0^\tau |\tilde q_s|^2 {\rm d}s\right)\right]\\
&\leq \left\{\E\left[\exp\left(\bar p^2\int_0^\tau \tilde q_s {\rm d}B_s-\frac{\bar p^4}{2}\int_0^\tau |\tilde q_s|^2 {\rm d}s\right)\right]\right\}^{\frac{1}{\bar p}}
\left\{\E\left[\exp\left(\frac{\bar p^2(\bar p+1)}{2}\int_0^\tau |\nu_s|^2 {\rm d}s\right)\right]\right\}^{\frac{\bar p-1}{\bar p}}\\
&<+\infty,
\end{align*}
which means that $L_\tau^{\tilde{q}}$ admits moments of any order. Thus, by an identical argument to that in the first and second steps we obtain the desired  attainability.

Next, we verify the second case that $\E[\exp\left(\frac{1}{2}\int_0^\tau \nu_t^2 \dif t\right)]<+\infty$ and $\int_0^\tau e^{\int_{0}^{s}\mu_r{\rm d}r}h_s{\rm d}s\leq K$ for some constant $K>0$. In this case, the stochastic exponential process $(L_t^{\tilde q})_{t\in [0,\tau]}$ defined in \eqref{eq:6.36*} is a uniformly integrable martingale. However, it is uncertain that $L_\tau^{\tilde{q}}$ admits moments of any order. Since \eqref{5.02*} remains true, it reduces to prove that $\tilde{q}_\cdot\in \hcal(\xi,f,\mu_\cdot,\nu_\cdot)$ and \eqref{5.05} remains valid. To show that $(\tilde{q}_t)_{t\in[0,\tau]}\in \hcal(\xi,f,\mu_\cdot,\nu_\cdot)$, we need to verify that \eqref{5.06} is true. Indeed, in light of $|\tilde{q}_\cdot|\leq \nu_\cdot$, by an identical way to prove \eqref{eq:6.11*} we can derive that for some $p'\in (1,p)$,
\begin{align}\label{eq:6.37*}
\sup_{s\geq 0}\E\left[\left(e^{\int_0^{s\wedge \tau}\mu_r\dif r}|y_{s\wedge \tau}|L_{s\wedge \tau}^{\tilde q}\right)^{p'}\right]<+\infty,
\end{align}
which means that $(e^{\int_0^{t}\mu_r\dif r}|y_t|L_t^{\tilde{q}})_{t\in[0,\tau]}$ is uniformly integrable. Then, it follows from \ref{A:A3} that
\begin{align}\label{eq:6.38*}
\E_{\Q^{\tilde{q}}}\left[e^{\int_0^\tau \mu_r\dif r}|\xi|+\int_{0}^{\tau}e^{\int_0^s \mu_r\dif r}(h_s+f^{-}(s,\tilde{q}_s))\dif s\right]&\leq 2\E\left[\left(e^{\int_0^\tau \mu_r\dif r}|\xi|+\int_{0}^{\tau}e^{\int_0^s \mu_r\dif r}h_s\dif s\right) L_\tau^{\tilde{q}}\right]\nonumber\\
& \leq 2\E\left[e^{\int_0^\tau \mu_r\dif r}|y_\tau| L_\tau^{\tilde{q}}\right]+2K\E\left[L_\tau^{\tilde{q}}\right]<+\infty.
\end{align}
On the other hand, since the equality \eqref{5.08} remains true for each $n\geq 1$ and $t\geq 0$, by virtue of \eqref{5.08} with $t=0$, \eqref{eq:6.37*} and \eqref{eq:6.38*} we obtain
\begin{align}\label{eq:6.18}
&\sup_{n\geq1}\E_{\Q^{\tilde{q}}}\Big[\int_0^{\tau_n^0} e^{\int_0^{s}\mu_r\dif r} f^{+}(s,\tilde{q}_s)\dif s\Big]\nonumber\\
& \ \ \ \leq |y_0|+\sup\limits_{s\geq 0}\E_{\Q^{\tilde{q}}}\left[e^{\int_0^{s\wedge \tau}\mu_r\dif r}|y_{s\wedge \tau}|\right]+\E_{\Q^{\tilde{q}}}\Big[\int_0^{\tau} e^{\int_0^{s}\mu_r\dif r} f^{-}(s,\tilde{q}_s)\dif s\Big]<+\infty,
\end{align}
which along with Fatou's lemma and \eqref{eq:6.38*} yields \eqref{5.06}. Consequently, $(\tilde{q}_t)_{t\in[0,\tau]}\in \hcal(\xi,f,\mu_\cdot,\nu_\cdot)$. Finally, in light of \eqref{5.06} and \eqref{eq:6.37*},
sending $n\To \infty$ in \eqref{5.08} yields \eqref{5.05}. This case is then proved.

It remains to check the third case that the stochastic exponential process $(L_t^{\tilde q})_{t\in [0,\tau]}$ defined in \eqref{eq:6.36*} is a uniformly integrable martingale and $\int_0^\tau e^{\int_{0}^{s}\mu_r{\rm d}r}h_s{\rm d}s\leq K$ for some constant $K>0$. In this case, it is uncertain that $\E[\exp\left(\frac{1}{2}\int_0^\tau \nu_t^2 \dif t\right)]<+\infty$. However, the desired assertion can be proved just as the second case. The proof of \cref{thm:6.1} is then complete.
\end{proof}

\begin{rmk}\label{rmk:5.4}
We have the following three remarks.
\begin{itemize}
\item [(i)] The main idea in the first two steps of the proof of \cref{thm:6.1} is inspired by \citet[(v) of Theorem 3.2]{FanHuTang2025}, while the proof of the third step is totally novel. The construction of the approaching sequences $\{\nu_\cdot^n\}$ and $\{f_n\}$ and the utilization of \cref{thm:3.1} and \cref{thm:3.3} are two key points.

\item [(ii)] In order to obtain the accessibility of the infimum in \eqref{Ut}, imposing some extra moment integrability conditions on $\int_0^\tau \nu_t^2 \dif t$ is quite natural since when $\tau$ is a positive constant and $\nu_\cdot$ is a nonnegative constant, these conditions are trivially fulfilled. However, the validity of assertion (iii) in \cref{thm:6.1} does not require these moment integrability conditions.

\item [(iii)] It is obvious that the condition that $\int_0^\tau e^{\int_{0}^{s}\mu_r{\rm d}r}h_s{\rm d}s\leq K$ for some constant $K>0$ appearing in the assertions (ii) and (iii) of \cref{thm:6.1} can be relaxed to the following condition:
    \begin{align*}
    \E\left[\left(\int_0^\tau e^{\int_{0}^{s}\mu_r{\rm d}r}h_s{\rm d}s\right)L_\tau^{\tilde q}\right]<+\infty.
    \end{align*}
    In particular, when $h_\cdot\equiv 0$, this condition is trivially satisfied.
\end{itemize}
\end{rmk}

\setlength{\bibsep}{2pt}

\end{document}